\documentclass[11pt]{amsart}
\usepackage[utf8]{inputenc}
\usepackage{amsfonts, mathtools, nccmath, amssymb, graphicx, multicol, commath, bbm, mathrsfs, bm, subfig, mathrsfs}
\usepackage{amsthm, mathabx}
\usepackage{esint}
\usepackage{braket}
\usepackage[foot]{amsaddr}
\usepackage{bibentry}
\usepackage{color,dsfont}
\usepackage[top=1in, bottom=1in, left=1in, right=1in]{geometry}
\usepackage{hyperref}
\hypersetup{colorlinks,breaklinks,
             linkcolor=blue,urlcolor=blue,
             anchorcolor=blue,citecolor=blue}

\usepackage{cite}

\newcommand{\CP}{\mathbb{C}}
\newcommand{\R}{\mathbb{R}}
\newcommand{\Z}{\mathbb{Z}}

\newcommand{\N}{\mathbb{N}}

\renewcommand{\P}{\mathscr{P}}

\newcommand{\M}{\mathcal{M}}

\newcommand{\E}{\mathbb{E}}

\newcommand{\Leb}[1][]{\mathscr{L}}
\newcommand{\bQ}{\mathbb{Q}}
\newcommand{\bP}{\mathbb{P}}
\newcommand{\bW}{\mathbb{W}}

\newcommand{\F}{\mathcal{F}}

\newcommand{\B}{\mathcal{B}}

\renewcommand{\L}{\mathcal{L}}

\newcommand{\sH}{H}
\newcommand{\sF}{F}
\newcommand{\sG}{\Gamma}
\newcommand{\C}{\mathcal{C}}

\newcommand{\m}{\mathfrak{m}}

\newcommand{\one}{\mathds{1}}
\newcommand{\eps}{\varepsilon}

\DeclareMathOperator{\supp}{supp}

\DeclareMathOperator{\tr}{tr}

\DeclareMathOperator{\Ent}{Ent}

\DeclareMathOperator{\kl}{KL}

\DeclareMathOperator{\law}{Law}

\DeclareMathOperator{\Y}{\mathcal{Y}}

\newtheorem{theorem}{Theorem}[section]
\newtheorem{assumption}{Assumption}[section]

\newtheorem{lemma}[theorem]{Lemma}
\newtheorem{proposition}[theorem]{Proposition}
\newtheorem{corollary}[theorem]{Corollary}
\theoremstyle{definition}
\newtheorem{definition}[theorem]{Definition}

\definecolor{byzantine}{rgb}{0.74, 0.2, 0.64}
\definecolor{darkgreen}{rgb}{0.1,0.6,0.1}
\definecolor{darkred}{rgb}{0.6,0,0}
\definecolor{lightgray}{rgb}{0.5,0.5,0.5}

 \newenvironment{listi}
  {\begin{list} 
 {(\roman{broj})}
{ \usecounter{broj}}
     \setlength{\labelwidth}{25pt}
  }
{   \end{list} }
\newcounter{broj}

\def\XXint#1#2#3{{\setbox0=\hbox{$#1{#2#3}{\int}$}
     \vcenter{\hbox{$#2#3$}}\kern-.5\wd0}}

\newtheorem{remarkx}[theorem]{Remark}
\newenvironment{remark}{\pushQED{\qed}\remarkx}{\popQED\endremarkx}

\numberwithin{equation}{section}

\title{A PDE approach to quantitative asymptotic stability of the optimal filter in the linear detectable setting}
\author{Minh V.H. Nguyen, Sangmin Park} 
\address{M.V.H. Nguyen, S. Park: Department of Computing and Mathematical Sciences, California Institute of Technology, Pasadena, CA 91125, USA}
\email{mvnguyen@caltech.edu, sangminp@caltech.edu}

\begin{document}

\begin{abstract}
The filtering problem asks to estimate a signal process, the initial state of which is unknown, given its partial and noisy observations. From a probabilistic perspective, the best estimate at each time is the optimal filter, which is the conditional law of the signal given all observations up to the time. The problem of asymptotic stability of the optimal filter, which asks whether the filter initialized with the incorrect initial law converges to the correctly initialized filter, is of fundamental practical and theoretical importance.

This paper studies the quantitative asymptotic stability of the optimal filter for signal-observation models with linear drift and additive Brownian noise. When the initial law is additionally assumed to be Gaussian, the optimal filter reduces to the celebrated Kalman-Bucy filter, for which the detectability condition on the signal and observation matrices is the sharp condition ensuring asymptotic stability. In the linear setting where the detectability condition holds but the initial law need not be Gaussian, we make the following contributions by studying the Zakai equation satisfied by the unnormalized filter. First, we study the asymptotic behavior of the deterministic part of the Zakai equation in the general case with nonlinear drift, and place the detectability condition in this context. Second, in the case of linear drift, we characterize the solution of the Zakai equation (hence the optimal filter) by a gauge transformation that separates the finite-dimensional stochastic dynamics from the infinite-dimensional deterministic dynamics. Third, we use this representation to establish asymptotic exponential filter stability in the $\chi^p$-divergence (with respect to a reference optimal filter), the total variation norm, and the $2$-Wasserstein distance.
\end{abstract}

\maketitle




\section{Introduction}
Let $b:\R^{d}\rightarrow\R^{d}$, $h:\R^{d}\rightarrow\R^{k}$ be continuous vector fields with $k\leq d$, and let $\Gamma\succ 0$ be a positive definite $k$ by $k$ matrix. 
Let $W^1=(W_t^1)_{t\geq 0}$ and $(W_t^2)_{t\geq 0}$ be independent standard Brownian motions of respective dimensions $d$ and $k$. Consider the pair of signal process $X=(X_t)_{t\geq 0}$ and observation process $Y=(Y_t)_{t\geq 0}$ characterized by the following system of stochastic differential equations (SDE): 
\begin{equation}\label{eq:signal-obs}
\begin{cases}
    dX_t&=b(X_t)\,dt+\,dW_t^1 \\
    dY_t&=h(X_t)\,dt+\sqrt{\sG}\,dW_t^2.
\end{cases}
\end{equation}
A more general case where the signal noise has any constant nondegenerate covariance can be reduced to the above case by a coordinate change.

Given $\mu\in\P(\R^d)$, the optimal filter (or simply the filter) $\pi^{\mu}_t$ is the conditional law  of $X_t$ given $(Y_s)_{s\leq t}$ and $X_0\sim\mu$. More precisely, denoting by $\bP^\mu$ the probability measure under which $X_0\sim\mu$ and $\E^\mu$ the (conditional) expectation under $\bP^\mu$,
\begin{equation}\label{def:optimal filter}
    \pi_t^{\mu}(\psi):=\E^{\mu}[\psi(X_t)|(Y_s)_{s\leq t}] \text{ for each } \psi\in C_b(\R);
\end{equation}
see Section~\ref{sssec:SigObs} for a more precise definition and setting.

A central theoretical question regarding the filter is its almost-sure asymptotic stability:
\begin{equation}\label{eq:pi_t;asymp_stab}
    \pi_t^\mu-\pi_t^\nu \rightarrow 0 \text{ a.s. } \text{ for a broad range of initial data }\mu,\nu\in\P(\R^d)
\end{equation}
in a suitable sense of convergence of probability measures. For simplicity, we will often refer to this property simply as filter stability. Intuitively, the asymptotic stability \eqref{eq:pi_t;asymp_stab} attests to the quality of observations, as it says that the accumulation of information from the observation process overrides the (usually spurious) information contained in the initial law. As filtering is applied in situations where precise knowledge of the initial law is inaccessible, the asymptotic stability can be understood as the well-posedness of the filtering problem. Asymptotic stability in the almost-sure sense is natural as, in most realistic situations, access to a single realization of the observation process is provided.

The question of filter stability is almost as old as the study of filter itself. Stratonovich introduced the optimal filter~\cite{Stratonovich59, Stratonovich60conditional}. Concurrently, the special case of the optimal filter in the \emph{linear-Gaussian setting}, which (in this paper) refers to the case where the vector fields $b,h$ in \eqref{eq:signal-obs} are linear and additionally the initial state $X_0$ is  assumed to be Gaussian, was introduced by Kalman~\cite{Kalman60} in discrete time (known as the Kalman filter) and Kalman and Bucy~\cite{KalmanBucy61} in continuous time (known as the Kalman-Bucy filter). In the linear-Gaussian setting, the optimal filter remains Gaussian at all times, reducing the study of the filter to their mean and covariance. In their paper~\cite{KalmanBucy61}, Kalman and Bucy established filter stability under controllability and observability. However, it was not until about 50 years later that the analogous result, namely filter stability under uniform observability, was established for the optimal filter in the general setting (much more general than \eqref{eq:signal-obs}) by van Handel~\cite{vanHandel09uniform}. Roughly speaking, observability is a condition on $b,h$ that ensures that no two distinct initial laws give rise to the same observations. In other words, the signal and observation vector fields interact in a way that all dynamical information is witnessed by the observations; we refer to the elucidating survey by van Handel~\cite{vanHandel10proceedings} for further discussions. 

However, the sharp condition for the Kalman-Bucy filter to be asymptotically stable is detectability, a condition due to Wonham~\cite{Wonham68riccati} that only requires observations to see in directions that are not asymptotically stable. Outside of the linear-Gaussian setting, filter stability under detectability is limited. Van Handel formulated the detectability condition in the abstract probabilistic setting, obtaining filter stability when the state space is finite~\cite{vanH09PTRF}. When $b,h$ are linear and satisfy the detectability condition, Ocone and Pardoux established filter stability in the weak sense: any bounded and uniformly continuous moments of the two filters with different (possibly non-Gaussian) measures must coincide in large time~\cite{OconePardeoux96asymptotic}. In the absence of tightness for optimal filter, which need not hold unless the signal process itself is asymptotically stable, the asymptotic stability of moments of the filter does not, in general, imply that \eqref{eq:pi_t;asymp_stab} holds in the sense of weak convergence of probability measures.

The goal of this paper is to put into context the finite-dimensional linear-Gaussian theory within the infinite-dimensional setting by studying the equation satisfied by an unnormalized density of the optimal filter, namely the Duncan-Mortensen-Zakai equation~\cite{Duncan67,Mortensen66,Zakai69optimal} (henceforth referred to simply as the Zakai equation; see Section~\ref{sssec:filtering_eq}), in the case where $b,h$ are linear and satisfy the detectability condition, but the initial law need not be Gaussian. Henceforth, we simply refer to this setting as the \emph{linear detectable setting}.

To introduce the Zakai equation, let $\L$ be the infinitesimal generator of the signal process $X=(X_t)_{t\geq 0}$ from \eqref{eq:signal-obs}, and $\L^\ast$ its adjoint with respect to the Lebesgue measure -- i.e.
\begin{equation}\label{def:L}
    \L u = \frac12\Delta u+ b\cdot\nabla u,\qquad \L^\ast \nu=\frac12\Delta\nu-\nabla\cdot(\nu b).
\end{equation}
Then, the Zakai equation~\cite{Zakai69optimal} is a linear stochastic PDE satisfied by the unnormalized filter density (see Section~\ref{ssec:canonical_filter}), formally written as
\begin{equation}\label{eq:Zakai;PDE}
    d\rho_t=\L^\ast \rho_t\;dt +(h^T \sG^{-1} dY_t)\rho_t;
\end{equation}
$dY_t$ is understood as the It\^{o} integral.
Taking a logarithm to \eqref{eq:Zakai;PDE} and formally applying It\^{o}'s formula, the negative log potential of the unnormalized density $u(t,\cdot):=-\log\rho_t$ satisfies the viscous Hamilton-Jacobi equation with stochastic forcing (cf. Mitter and Fleming~\cite{FlemingMitter1982})
\begin{equation}\label{eq:-logZakai} 
    d u-\frac12\Delta u+\frac12|\nabla u|^2+b\cdot\nabla  u=\frac12 |h|_{\sG}^2-h^T\sG^{-1}\; d Y_t+\nabla\cdot b;
\end{equation}
here, $dY_t$ can be understood in either the It\^{o} or Stratonovich sense.

Contributions of this paper are threefold.
\begin{itemize}

    \item[\textbf{(C1)}] We place detectability from linear systems theory in the context of PDEs by connecting it with the asymptotic behavior of the deterministic part of the negative log-Zakai equation, \eqref{eq:-logZakai}.

    \item[\textbf{(C2)}] We precisely characterize the solution of the Zakai equation~\eqref{eq:Zakai;PDE} by a gauge transformation that separates the stochastic, finite-dimensional evolution following an SDE from the infinite-dimensional component satisfying a partial differential equation (PDE) (Theorem~\ref{thm:gauge}).
    
    \item[\textbf{(C3)}] We quantitatively characterize the large-time behavior of the filter,
    establishing that the $\chi^p$-divergence of the filter with respect to a reference optimal filter decays asymptotically exponentially for any $p\in [1,+\infty)$ (Theorem~\ref{thm:OF_stability}), from which asymptotic exponential filter stability in total variation and the 2-Wasserstein distance (Corollary~\ref{cor:TV,W2->0}) follows.
\end{itemize}

More precisely, our first contribution \textbf{(C1)} verifies that, when $b,h$ are linear, the detectability condition allows us to precisely characterize the asymptotic behavior of the viscous Hamilton-Jacobi equation that is the deterministic part of \eqref{eq:-logZakai}. This fact relies on the classical idea of Doob's $h$-transform~\cite{Doob57htransform} in probability or the ground state transform~\cite{Agmon82} in the study of Schr\"{o}dinger operators (or more generally, linear elliptic operators). We organize the relevant results in the general setting where $b,h$ is nonlinear in a way that is convenient to verify consistency with detectability from linear systems theory. 

Our exposition involves both the viscous Hamilton-Jacobi equation and the (equivalent) linear parabolic equation obtained by the Hopf-Cole transform. Rigorous results are formulated in terms of the latter to allow for general initial data. On the other hand, the former will be preferred for (sometimes formal) calculations and developing ideas; as we will see in the summary of main ideas, Section~\ref{sssec:main ideas}, the connections to the Kalman-Bucy filter, which is formulated in terms of the covariance matrices, are more apparent this way. Relatedly, the clean separation of the Zakai equation in the linear setting into finite-dimensional SDE and determinitic PDE, which is our second contribution \textbf{(C2)}, is more natural in terms of the solution of the viscous Hamilton-Jacobi equation.

Combining \textbf{(C1)} with \textbf{(C2)}, leads to our main contribution \textbf{(C3)}, which establishes quantitative filter stability in the linear coefficient setting under the detectability assumption. To the best of our knowledge, quantitative filter stability in this setting was not available previously; in fact, this somewhat surprising gap in the literature largely motivated this this paper, as such a result is necessary for quantitative asymptotic accuracy of the mean-field ensemble Kalman filter with respect to the optimal filter studied in a companion paper~\cite{HoffmannParkStuart2026EnKF}. Moreover, even the qualitative filter stability in the total variation norm appears to be unknown the stronger conditions of observability; the decay in $\chi^p$-divergence for arbitrary finite $p\geq 1$ is much stronger, although the reference optimal filter is carefully chosen for this mode of convergence. 

This speaks to the effectiveness of PDE techniques in studying the dynamic properties of the optimal filter in concrete settings such as \eqref{eq:signal-obs}. The downside is that we need to address subtle uniqueness and regularity issues of the Zakai equation~\eqref{eq:Zakai;PDE}, which the intrinsic methods solely relying on the conditional law representation \eqref{def:optimal filter} are free from.

We are hopeful that some of these techniques can be extended to the more general setting where $b,h$ are allowed to be nonlinear, where understanding of filter stability remains very limited. At the same time, the additional routine but detailed hard analysis leading to \textbf{(C3)} in the linear detectable setting hints at the difficulty of analyzing filter stability in the form \eqref{eq:pi_t;asymp_stab} in the nonlinear setting, where analogous explicit estimates may be difficult to obtain.

The remainder of this manuscript is organized as follows. We fix recurring notation in Section~\ref{ssec:notation}, and describe our setting in detail in Section~\ref{ssec:setting}. Section~\ref{ssec:summary} summarizes the main ideas and results, and Section~\ref{ssec:literature} provides a (partial) survey of the literature. Then our contributions \textbf{(C1)}, \textbf{(C2)}, and \textbf{(C3)} are developed in Sections~\ref{sec:C1}, \ref{sec:filter_charac}, and \ref{sec:asymp}, respectively.

\subsection{Notation}\label{ssec:notation} 
We write $\R$, $\R_+$, $\R_{++}$ for the reals, nonnegative reals, and strictly positive reals, respectively; $\CP$ denotes the complex plane.
We write $\N$ for the natural numbers and $\Z_{++}:=\N\setminus\{0\}.$ 
For $q\in\mathbb{Z_{++}}$, $\R^q$ is the $q$-dimensional Euclidean space. For $l\in\N$, the space of $l$-times continuously differentiable functions (resp. space of functions with up to $l$ continuous and bounded derivatives) from $\R^q$ to $\R^p$,  given $p,q\in\N$, is denoted by $\C^l(\R^q;\R^p)$ (resp. $\C_b^l(\R^q;\R^p)$), where the superscript $l$ is omitted if $l=0$; when $p=1$ we write $\C^l(\R^q)$ (resp. $\C_b^l(\R^q)$). Given $f\in\C^1(\R^q)$, its gradient is denoted by $\nabla f$, and its Hessian by $D^2 f$ if $f\in\C^2(\R^q)$. We write $\C_c^l(\R^q;\R^p)\subset \C_b^l(\R^q;\R^p)$ for the space of $l$-times continuously differentiable functions with compact support. For $r\geq 1$, the $r$-Lebesgue space on $\R^q$ is denoted by $L^r(\R^q)$. We write $\one$ to refer to the indicator function on $\R^q$.

We write $\R^{p\times q}$ for real $p$ by $q$ matrices. The identity matrix on $\R^q$ is denoted $I_q\in\R^{q\times q}_{++}$; when $q$ is clear from context, we often omit the subscript and write $I=I_q$. Given symmetric matrix $M\in\R^{q\times q}$, we write $M\succeq 0$ and $M\succ 0$ for, respectively, positive semi-definite and positive definite $M$. We use $\R^{q\times q}_{+},\R^{q\times q}_{++}$ to denote the cones of positive semi-definite and positive definite matrices, respectively. Given $M\in\R^{q\times q}_{++}$, we write the weighted inner-product on $\R^q$ as $\langle v, w \rangle _{M}:=v^T M^{-1} w$, and $|\cdot|_{M}$ to denote the induced norm:
\begin{equation}\label{def:Mweightednorm}
    |v|_M^2 := v^T M^{-1} v \text{ for each } v\in\R^q.
\end{equation}
When $M=I$ we obtain the standard Euclidean norm and simply write $|\cdot|$, noting that the same notation is also used
for complex modulus.

Given $M\in\R^{q\times q}$, we define its spectral abscissa, for $\Re:\CP \to \R$ denoting map to the real part of
a complex number,
\begin{equation}\label{def:alpha}
    \alpha(M):=\max\{\Re(\lambda):\;\lambda\text{ is an eigenvalue of } M\}.
\end{equation}
Whenever the norm $\|\cdot\|$ is applied to matrices, it refers to the operator norm with respect to the Euclidean norm.

We write $\B(E)$ to denote the Borel $\sigma$-algebra on a Polish space $E$.
The space of Radon measures (resp. nonnegative Radon measures) on $\R^q$ is denoted by $\M(\R^q)$ (resp. $\M_+(\R^q)$), and the space of Borel probability measures on $\R^q$ by $\P(\R^q)$. We write $X\sim\mu$ for the $q$-dimensional random vector $X$ distributed according to
$\mu\in\P(\R^q)$. 

Given $\mu\in\P(\R^q)$ and any Borel measurable function $\psi:\R^q\rightarrow\R$ growing at most quadratically at infinity, we write $\mu(\psi)=\int_{\R^q}\psi(u)\;\mu(du)$. For $\boldsymbol{\psi}=(\psi^1,\cdots,\psi^q)$ 
we write $\mu(\boldsymbol{\psi})=(\mu(\psi^1),\cdots,\mu(\psi^q))\in\R^q$.

Given a sequence $(\eta_n)_{n\in\Z_{++}}$ and $\eta$ in $\M_+(\R^q)$ we say the sequence $(\eta_n)_{n\in\Z_{++}}$ converges vaguely to $\mu$ (and write $\eta_n\xrightharpoonup[]{v}\eta$) as $n\rightarrow\infty$ if
\begin{equation}\label{def:vague}
    \eta_n(\psi)\rightarrow  \eta(\psi) \text{ as } n\rightarrow\infty \text{ for each } \psi\in C_c(\R^q).
\end{equation}
We say the sequence $(\eta_n)_{n\in\Z_{++}}$ converges to $\eta$ narrowly to $\eta$ if
\begin{equation}\label{def:narrow}
    \eta_n(\psi)\rightarrow  \eta(\psi) \text{ as } n\rightarrow\infty \text{ for each } \psi\in C_b(\R^q).
\end{equation}

We use the translation operator for measures and functions. Given $z\in\R^q$, the translation operator $\tau_z$ on measure $\eta\in\M(\R^q)$ is defined as the pushforward of $\nu$ by the map $x\mapsto x+z$, whereas $\tau_z$ acts on functions $f:\R^d\rightarrow\R^d$ the consistent way -- i.e.
\begin{equation}\label{def:tau}
    \int_{\R^q} \psi(x)\tau_z \eta(dx)=\int_{\R^q} \psi(x+z)\eta(dx),\qquad \tau_{z} f(x)= f(x-z).
\end{equation}

Gaussian measures on $\R^q$ with mean $z\in\R^q$ and covariance $M\in\R^{q\times q}_{+}$ are denoted by $N(z,M)$ or $\gamma_M^z$. When $z=0$, we will simply write $\gamma_M$, so that $N(z,M)=\gamma_M^z=\tau_z \gamma_M$. We overload the notation and also use $\gamma_M$ to refer to the density of the corresponding Gaussian measure.

We use $(\Omega,\F,(\F_t)_{t\geq 0})$ to denote the filtered measurable space in which the signal-observation \eqref{sssec:SigObs} takes place, and $\bP^\mu$ the probability measure under which $X_0\sim\mu$ and $W^1,W^2$ are standard Brownian motions; see Section~\ref{sssec:SigObs} for the explicit construction. The expectation (and conditional expectation) under $\bP^\mu$ is written by $\E^{\mu}$. When $\mu=\delta_{x}$ for some $x\in\R^d$, we use the short-hand notation $\bP^x:=\bP^{\delta_x}$, $\E^{x}:=\E^{\delta_x}$.

Given $\mu\in\P(\R^d)$, we denote the weighted entropy of each positive measurable function $f \in L^1(\mu)$ by
\begin{equation}\label{def:Ent_mu}
    \Ent_\mu(f)=\int_{\R^d} f\log f\;d\mu - \int_{\R^d} f\;d\mu \log \left(\int_{\R^d} f\;d\mu\right).
\end{equation}
Note that the normalization ensures that $\Ent \mu$ is $1$-positive homogeneous: $\Ent_{\mu}(cf)=c\Ent_{\mu}(f)$ for $c\geq 0$.

A stochastic process $U=(U_t)_{t\in\R_{+}}$ in $\R^q$ is said to be ergodic if there exists a unique probability measure $\mu\in\P(\R^q)$ such that
\begin{equation}\label{eq:ergodic mu}
    \mu(E) = \int_{\R^d} \bP^x(U_t\in E)\ \mu(dx), \qquad\text{for all }t>0
\end{equation}
for all $E\in\mathcal{B}(\R^d)$. Uppercase letters $B,W$ are reserved for standard Brownian motions in $\R^d$.

Given two stochastic processes on $\R^q$, $U=(U_t)_{t\geq 0}=((U_t^1,\cdots,U_t^q))_{t\geq 0}$ and $\tilde U=(\tilde U_t)_{t\geq 0}=$ 
$((\tilde U_t^1,\cdots,\tilde U_t^q))_{t\geq 0}$, their quadratic covariation is denoted by $[U,\tilde U]_t$, which is matrix with their $(i,j)$-component is the real-valued quadratic covariance $[U^i,\tilde U^j]_t$. We write the quadratic variation by $[U]_t:=[U,U]_t$.

\subsection{Setting}\label{ssec:setting}
We now describe our setting. Section~\ref{sssec:SigObs} constructs the probability space for the signal-observation process~\eqref{eq:signal-obs} explicitly, primarily to clarify subtleties related to the dependence of the probability measure on the initial law of $X_0$. Then, in Section~\ref{sssec:filtering_eq} we discuss the stochastic partial differential equations (SPDEs) describing the optimal filter, namely the Kushner-Stratonovich and the Zakai equations. 

\subsubsection{Signal-observation process}\label{sssec:SigObs}
We begin with the construction of the canonical probability space, inspired by Section 1.1.2 of van Handel's thesis~\cite{vanHandel06Thesis}. For $i=1,2$, let $q_1=d$, $q_2=k$, and let $\Omega^i\subset \C(\R_+;\R^{q_i})$ be the space of consisting of continuous paths $w^i:\R_+\rightarrow\R^{q_i}$ with $w^i(0)=0$. Let $\bW^q$ be the $q$-dimensional Wiener measure. Fix the sample space $\Omega=\R^d\times \Omega^1\times \Omega^2$. For each $\mu\in\P(\R^d)$, define the Borel probability measure $\bP^\mu$ on $\B(\Omega)$ defined by
\begin{equation}\label{def:Pmu}
    \bP^\mu:= \mu\otimes \bW^d\otimes\bW^k.
\end{equation}
For each $\omega=(x_0,w^1,w^2)\in\R^d\times\C_0(\R_+;\R^d)\times\C_0(\R_+;\R^k)$ define
\[X_0(\omega):=x_0,\quad W_t^1(\omega)=w^1(t),\quad W_t^2(\omega)=w^2(t).\]
Thus, under $\bP^\mu$, $X_0,W^1,W^2$ are independent, with $X_0\sim\mu$, and $W^1\sim\bW^1$, and $W^2\sim\bW^2$ being standard Brownian motions each in dimension $d$ and $k$. In particular, the probability $\bP^\mu$ depends on $\mu$ only through the initial condition $X_0$, and $W^i$ have the same law under $\bP^\mu,\bP^\nu$ for any $\mu,\nu\in\P(\R^d)$.

Define the (raw) filtration
\begin{equation}\label{def:Ft}
    \F_t:=\sigma\{X_0,W_s^1,W_s^2:\;s\leq t\} \text{ for } t\geq 0, \text{ and } \F=\sigma\left(\bigcup_{t\geq 0} \F_t\right).
\end{equation}
Let $X=(X_t)_{t\geq 0}$ be the strong solution of the signal SDE \eqref{eq:signal-obs}. Then, $Y=(Y_t)_{t\geq 0}$ is defined accordingly with initial datum $Y_0=0$. Thus, define the observation filtration 
\begin{equation}\label{def:Yt}
    \Y_t:=\sigma\{Y_s:\;s\leq t\},\quad \Y=\sigma\left(\bigcup_{t\geq 0}\Y_t\right).
\end{equation}
Then, the (conditional) expectation $\E^\mu$ is the expectation under the probability measure $\bP^\mu$.

Note that in our construction of the probability space, the probability measure $\bP^\mu$ in \eqref{def:Pmu} depends on the choice of the initial measure $\mu$. Thus, we need some technical care in defining what almost sure statements (such as \eqref{eq:pi_t;asymp_stab}) mean. 

In our model~\eqref{eq:signal-obs}, the observation noise $\Gamma^{-1/2}W^2$ is independent of the signal process $X$. Thus, by the classical result of Kailath and Zakai~\cite[Proposition 2]{KaliathZakai71AnnMathStat}, 
\begin{equation}\label{eq:P_YT;equivalence}
    \bP^\mu|_{\Y_T} \sim \bP^\nu|_{\Y_T} \text{ for any }\mu,\nu\in\P(\R^d) \text{ and } T\in(0,+\infty).
\end{equation}
In particular, the null sets in $\Y_T$ are independent of the choice of the initial law for $X_0$. This means semimartingale property of a process is independent of the initial law on every finite horizon. In fact, in our case of nondegenerate signal noise, it has been observed that $\bP^\mu|_{\Y} \sim \bP^\nu|_{\Y}$~\cite[Remark 3.8]{vanHandel09uniform}. But for our purposes, it suffices to take a countable union (resp. intersection) of the null sets (full-measure sets) over increasing horizons $[0,T_k]$ with $T_k\xrightarrow[]{k\rightarrow\infty}\infty$ to obtain a $\Y$-measurable null-measure (resp. full-measure) set that are common to all $\bP^\mu$. Accordingly, we use the following convention throughout the paper.

\textbf{Convention.} For any $\Y$-measurable statements, we say a property holds almost surely if there exists a set $\Omega'\in \Y$ such that
\begin{equation}\label{con:a.s}
    \bP^\mu(\Omega')=1 \text{ for all }\mu\in\P(\R^d)
\end{equation}
and the property holds for all $\omega\in\Omega'$.

\subsubsection{Filtering equations}\label{sssec:filtering_eq}

The relevant probability space for the optimal filter $(\pi_t^\mu)_{t\geq 0}$~\eqref{def:optimal filter} is $(\Omega,(\Y_t)_{t\geq 0},\Y,\bP^\mu|_{\Y})$. Thus, by the discussion at the end of Section~\ref{sssec:SigObs}, almost-sure statements are independent of the initial law $\mu$. 

The optimal filter $\pi_t=\pi_t^\mu$ (we suppress the superscript for generic discussions where initial law is not relevant) is characterized by the Kushner-Stratonivich equation~\cite{Kushner64}, a nonlocal and nonlinear stochastic partial differential equation (PDE) which can be formally written as
\begin{equation}\label{def:KS}
    d\pi_t=\L^\ast\pi_t\;dt+\pi_t(h-\pi_t(h))^T\Gamma^{-1}(dY_t-\pi_t(h)\;dt) \text{ in } (0,\infty)\times\R^d,;
\end{equation}
$dY_t$ here is understood in the sense of the It\^{o} integral. The optimal filter can be recovered from the solution $\rho_t$ of \eqref{eq:Zakai;PDE}
\begin{equation}\label{eq:pit-rhot}
    \pi_t(x)=\frac{\rho_t(x)}{\int_{\R^d}\rho_t(y)\;dy}.
\end{equation}

We define the precise sense in which we understand the classical solutions of the Zakai equation~\eqref{eq:Zakai;PDE} (cf. \cite[Definition 7.5, Corollary 7.2]{RozovskyLototsky18book}). 
\begin{definition}[Classical solution of the Zakai equation]\label{def:Zakai;classical}
    
    We say $\rho:(0,+\infty)\times\R^d\times \Omega\rightarrow [0,+\infty)$ is a classical solution of the Zakai equation~\eqref{eq:Zakai;PDE}  with initial trace $\nu\in\P(\R^d)$ if, for every $T>0$, the restriction of $\rho$ to $(0,T]$ satisfies the following:
    \begin{listi}
        \item $(t,x,\omega)\mapsto \rho(t,x,\omega)$ is $\B((0,T]\times \R^d)\otimes\Y_T$-measurable.
        \item We have
        \begin{equation}\label{eq:Zakai;C2}
        \rho,\nabla\rho,D^2\rho\in C((0,T]\times\R^d) \text{ almost surely.}
        \end{equation}
        \item For every $\eps\in(0,T)$ and $x\in\R^d$, the process $(\rho_t(x))_{t\in[\eps,T]}$ is a continuous $\Y_t$-semimartingale under every $\bP^\mu$, and almost surely, for all $t\in[\eps, T]$
        \begin{equation}\label{eq:Zakai;semimartingale}
            \rho_t(x)=\rho_\eps(x) + \int_{\eps}^t \frac12\Delta\rho_s(x)-\nabla\cdot(\rho_s b)(x)\;ds + \int_\eps^t h(x)^T\sG^{-1}\rho_s(x)\cdot dY_s.
        \end{equation}
        \item Almost surely, $\rho_t(x)\;dx$ converges vaguely to $\nu$ as $t\rightarrow 0$. 
    \end{listi}
    The almost-sure statements in this definition follows \textbf{Convention} in the previous section -- i.e. the $\Y$-measurable full-measure event is independent of $x,T,\nu$.
\end{definition}

Given that $\L^\ast$ is a uniformly elliptic operator, regularity is not an issue at any positive time, whereas we understand the initial datum in the vague sense to avoid imposing unnecessary regularity assumptions on it.

\subsection{Summary}\label{ssec:summary}
We now summarize the main ideas and results of this paper, organized by each contribution, $\textbf{(C2)}$ and $\textbf{(C3)}$.

\subsubsection{Summary of main ideas}\label{sssec:main ideas}
In the linear-Gaussian setting, where $b(x)=Fx, h(x)=Hx$ and additionally the initial law is Gaussian, the optimal filter is parametrized by the mean and covariance $\pi_t^\mu=N(m_t,C_t)$ where $\mu=N(m_0,C_0)$ and the mean and covariance satisfy the equations for the Kalman-Bucy filter
\begin{equation}\label{eq:Kalman-Bucy}
\begin{split}
	d m_t &= Fm_t\;dt +C_t H^T\Gamma^{-1}(dY_t-H m_t\;dt),\\
	\dot C_t &=C_t F^T+FC_t+I_d-C_t H^T\Gamma^{-1} HC_t.
\end{split}
\end{equation}
The equation \eqref{eq:Kalman-Bucy} is a special case the negative log-Zakai equation~\eqref{eq:-logZakai}, which, in the linear setting, is
\begin{equation}\label{eq:-logZakai;linear}
du-\frac12\Delta u+\frac12|\nabla u|^2+(Fx)\cdot\nabla u = \frac12|Hx|_{\sG}^2+\tr F- (Hx)^T \sG^{-1} dY_t \text{ in } (0,\infty)\times\R^d,.
\end{equation}
Setting $u(t,x)$ to be $\frac12 (x-m_t)^T C_t^{-1} (x-m_t)$ plus a scalar additive semimartingale term (which is not important for our discussion here), direct calculation verifies that \eqref{eq:Kalman-Bucy} provides a solution to \eqref{eq:-logZakai;linear}. In particular, the shape of the solution, characterized by the covariance, is deterministic and independent of the observation process $Y=(Y_t)_{t\geq 0}$.

Based on this observation, we might hope that the optimal filter in the linear setting may separate into an infnite-dimensional deterministic part characterizing the shape, and a finite-dimensional stochastic dynamics translating the distribution according to the observation process. This leads to the following ansatz:
\begin{equation}\label{eq:ansatz}
	u(t,x)=w(t,x-q_t)+x^T p_t + r_t
\end{equation}
where $w$ is a solution to the viscous Hamilton-Jacobi equation
\begin{equation}\label{eq:vHJE;linear}
     \partial_t w-\frac12\Delta w+\frac12|\nabla w|^2+(Fx)\cdot\nabla  w=\frac12|Hx|_{\sG}^2+\tr F \text{ in } (0,\infty)\times\R^d,
\end{equation}
and $(q_t,p_t,r_t)\in\R^d\times\R^d\times \R$, accounting for the stochastic effects, are to be solved for.
The role of $q_t$ in \eqref{eq:ansatz} is clear; the observations continuously cause random translation of the optimal filter. To see that $q_t$ alone is insufficient even for \eqref{eq:Kalman-Bucy}, note that the evolution of the mean $m_t$ depends on the covariance, which is natural given that data/observation plays a larger role in the Bayes update when the prior has high uncertainty. The linear tilt $x^T p_t$ is the simplest term that can achieve such a correction without affecting the Hessian of $u$, to be consistent with that the covariance evolution in \eqref{eq:Kalman-Bucy} is independent of the data. The scalar term $r_t$ provides an additional degree of freedom, and vanishes after normalization, thus is inconsequential for the optimal filter density.

Given \eqref{eq:ansatz}, we can understand the dynamics of the optimal filter by understanding the dynamics of $(q_t,p_t,r_t)_{t\geq 0}$ and $w$ separately. In subsequent sections, we highlight the main contributions of this paper.

\subsubsection{\textbf{(C1)}: Asymptotic behavior of the viscous Hamilton-Jacobi equation}\label{sssec:PDE}

Section~\ref{sec:C1} discusses the large-time behavior the solution of \eqref{eq:vHJE;linear}, which is characterized by the critical solution $(\lambda,\varphi)$ of the ergodic problem: 
\begin{equation}\label{eq:EP;linear}
    \lambda-\frac12\Delta\phi+\frac12|\nabla\phi|^2+(Fx)\cdot\nabla\phi=\frac12|Hx|_{\sG}^2+\tr F \text{ in } \R^d.
\end{equation}
Given a solution $(\lambda,\phi)$, the ratio
\begin{equation}\label{eq:f;ratio}
f(t,x):=\frac{e^{-w(t,x)}}{e^{-\lambda t-\phi(x)}}
\end{equation}
satisfies the parabolic equation
\begin{equation}\label{eq:f;linear}
\partial_t f = \frac12 f - (F x + \nabla \phi(x))x\cdot \nabla f \text{ in } (0,\infty)\times\R^d;
\end{equation}
this observation is made in Ichihara and Sheu~\cite[Theorem 4.10]{IchiharaSheu13} (see also Ichihara~\cite{Ichi11recurrence,Ichihara12}) in the context of viscous Hamilton-Jacobi equations, and corresponds to a special case of Doob's $h$-transform~\cite{Doob57htransform} for elliptic operators (cf. Pinsky~\cite{Pinsky95book}), which is also known as the ground-state transform in the context of Schr\"{o}dinger operators~\cite{Simon81JFA} (cf. Agmon~\cite{Agmon82} and Pinchover~\cite{Pinchover88Duke}). 

Thus, given a solution of the ergodic problem \eqref{eq:EP;linear} such that the diffusion
\begin{equation}\label{def:Xtphi}
dX_t^\phi = -(FX_t^\phi+\nabla \phi(X_t^\phi))\;dt + dB_t 
\end{equation}
is ergodic with an invariant measure satisfying the log-Sobolev inequality, the convergence of $w$ to $\lambda t + \phi$ can be characterized quantitatively using standard arguments in diffusion semigroups; see Lemma~\ref{lem:ent_dissipation}. 

In Section~\ref{sec:C1}, we carefully study the dynamic properties of the deterministic equations under these conditions. Most of the results are already known, and our aim is to to organize these results under the ergodicity of $X^\phi=(X_t^\phi)_{t\geq 0}$, which allows us to conveniently connect with the notion of detectability from linear systems theory. In particular, Lemma~\ref{lem:detectability} verifies that detectability of $(F,H)$ allows us to construct $\Sigma_\phi,Q_\infty\in\R_{++}^{d\times d}$ such that, $\phi(x)=\frac12 x^T \Sigma_\phi x$ is a solution to \eqref{eq:EP;linear} such that \eqref{def:Xtphi} is ergodic with invariant measure $\gamma_{Q_\infty}$.

\subsubsection{\textbf{(C2)}: Characterization of the optimal filter via gauge transformation}\label{sssec:summary;gauge}

The main result of Section~\ref{sec:filter_charac} is Theorem~\ref{thm:gauge}, which provides the following decomposition of the optimal filter under detectability of $(F,H)$: let $q_0\in\R^d$ and $\mu(dx)=e^{-w_0(x-q_0)}\;dx \in\P(\R^d)$, then the solution of the Zakai equation $\rho_t^\mu$ and optimal filter $\pi_t^\mu$~\eqref{def:optimal filter} with initial law $\mu$ are given by
\begin{equation}\label{eq:OF_gauge;summary}
\begin{split}
    \rho_t^\mu(x)= \exp\left(-w(t,x-q_t)-x\cdot p_t-r_t\right),\qquad \pi_t^\mu(dx)=\frac{\rho_t^\mu(x)\;dx}{\int_{\R^d} \rho_t^\mu(z)\;dz}, 
\end{split}
\end{equation}
where $w$ is a solution of \eqref{eq:vHJE;linear} with initial data $w_0$ (corresponding to a subprobability solution; see Proposition~\ref{prop:subprob_sol}) and $(q_t,p_t,r_t)_{t\geq 0}$ are characterized by a linear SDE \eqref{eq:qpbeta} driven by the observations $(Y_t)_{t\geq 0}$. In particular, for each $w_0,q_0$, such a decomposition is unique. The rigorous result allows general probability measure $\mu\in\P(\R^d)$ as initial law.

While formally verifying that $\rho_t$ in \eqref{eq:OF_gauge;summary} is a solution of the Zakai equation is straightforward, the rigorous justification of the characterization requires technical work. This is largely due to the surprising unavailability of uniqueness results for the Zakai equation in our setting, where observations are unbounded, initial law is allowed to be a general probability measure, and the initial datum is interpreted in a weak sense. The uniqueness of the solution of the Zakai equation in the sense of Definition~\ref{def:Zakai;classical} is proven as part Theorem~\ref{thm:gauge}, whereas Lemma~\ref{lem:Benes-Karatzas} verifies that the unnormalized filter density defined by the Kallianpur-Striebel formula is also a solution in this sense, thus identifying the gauge-transformed solution with the optimal filter.
\medskip

\subsubsection{\textbf{(C3)}: Large-time behavior of the optimal filter}\label{sssec:summar;large-time}
Section~\ref{sec:asymp} discusses the asymptotic behavior of the solution of the Zakai equation under the detectability assumption on $(F,H)$. Fortunately, Doob's $h$-transform \eqref{eq:f;ratio} combines particularly well with the decomposition~\eqref{eq:ansatz}. As $(t,x)\mapsto \lambda t+\phi(x)$ solves \eqref{eq:vHJE;linear}, writing $\mu_\phi=N(q_0,\Sigma_\phi)$, $\rho^{\mu_\phi}_t(x):=e^{-\lambda t -\phi(x-q_t)-p_t\cdot x - r_t}$ is a solution of the Zakai equation~\eqref{eq:Zakai;PDE} with invariant shape. Given any initial law of the form $\mu(dx)=e^{-w_0(x-q_0)}\;dx\in\P(\R^d)$, 
\[\frac{\rho_t^\mu(x)}{\rho_t^{\mu_\phi}(x)}=\exp\left(-w(t,x-q_t)-p_t\cdot x -r_t + \lambda t + \phi(x-q_t)+p_t\cdot x+r_t\right)=f(t,x-q_t).\]
In particular, once translated by $q_t$, this ratio is deterministic, and well-behaved. This leads to our first quantitative result, Corollary~\ref{cor:entropydissipation}, which states
\begin{equation}\label{eq:entropydissipation;summary}
    \Ent_{\tau_{q_t}\gamma_{Q_\infty}}\left(\frac{\rho_t^\mu}{\rho_t^{\mu_\phi}}\right)\leq e^{-\kappa_\phi t}\Ent_{\tau_{q_0}\gamma_{Q_\infty}}\left(\frac{d\mu}{d\mu_\phi}\right)
\end{equation}
where $\kappa_\phi>0$ is the log-Sobolev constant of $\gamma_{Q_\infty}$. While the simplicity of \eqref{eq:entropydissipation;summary} is appealing, $q_t$ need not (and will not, in general) faithfully track the true signal producing the observations. Moreover, the estimate is formulated in terms of the ratio of the unnormalized filter densities, not the optimal filters themselves. Thus, Corollary~\ref{cor:entropydissipation} is perhaps best understood as identifying a particularly well-behaved analytic quantity related to the optimal filter, rather than as a statement of asymptotic stability of optimal filter; see Remark~\ref{rmk:entropydissipation} for a more detailed discussion.

In order to obtain more probabilistically meaningful statements on the behavior of the optimal filter, note
\begin{equation}\label{eq:OF;ratio;summary}
    \frac{d\pi_t^\mu}{d\pi_t^{\mu_\phi}}(x)=Z_t^{-1}\frac{\rho_t^\mu(x)}{\rho_t^{\mu_\phi}(x)}=Z_t^{-1}f_t(x-q_t)
\end{equation}
where $Z_t$ is the random normalization constant (note that $\pi_t^{\mu_\phi}=N(q_t-\Sigma_\phi p_t,\Sigma_\phi)$)
\begin{equation}\label{def:Zt}
    Z_t:=\int_{\R^d} f_t(x-q_t)\;d\pi_t^{\mu_\phi} = \int_{\R^d} f_t(x-\Sigma_\phi p_t)\;d\gamma_{\Sigma_\phi}(x).
\end{equation}
As
\begin{align*}
    \chi^p(\pi_t^\mu||\pi_t^{\mu_\phi}) = \int_{\R^d} \left|\frac{d\pi_t^{\mu}}{d\pi_t^{\mu_\phi}}-1\right|^p\;d\pi_t^{\mu_\phi} = Z_t^{-p} \int_{\R^d} |f_t(x-\Sigma_\phi p_t)-Z_t|^p\;d\gamma_{\Sigma_\phi}(x),
\end{align*}
after a bit of hard analysis, we obtain our main result, Theorem~\ref{thm:OF_stability}: given $\mu\in\P(\R^d)$ satisfying $(\gamma_{\beta Q_\infty}e^{\phi}\tau_{-q_0}\mu)(\R^d)<+\infty$ for any $\beta>1$, we have
\begin{equation}\label{eq:chip_decay;summary}
    \limsup_{t\rightarrow\infty}\frac1t \log(\chi^p(\pi_t^\mu||\pi_t^{\mu_\phi}))^{1/p} \leq -\alpha_\phi<0 \text{ a.s. }
\end{equation}
The tail condition on $\mu$ holds automatically when $(F,H)$ satisfies the stronger condition called observability; see Remark~\ref{rmk:observability}.
Combined with standard functional inequalities, we obtain filter stability in total variation $\|\cdot\|_{TV}$ and the 2-Wasserstein distance $W_2$, Corollary~\ref{cor:TV,W2->0}: for initial laws $\mu,\nu\in\P(\R^d)$ satisfying the assumptions of Theorem~\ref{thm:OF_stability},
\begin{equation}\label{eq:TV,W2_decay;summary}
    \limsup_{t\rightarrow\infty}\frac1t \log\|\pi_t^\mu-\pi_t^\nu\|_{TV},\quad \limsup_{t\rightarrow\infty} \frac1t\log W_2(\mu,\nu) \leq -\alpha_\phi <0 \text{ a.s. }
\end{equation}

\subsection{Literature}\label{ssec:literature}
We discuss some of the earlier works most relevant for placing the current manuscript in context. Section~\ref{ssec:literature;OF} discusses literature related to the optimal filter and filtering equations, and Section~\ref{ssec:literature;vHJE} the large-time behavior of linear parabolic and Hamilton-Jacobi equations. We apologize for any omissions.

\subsubsection{Optimal filter and filtering equations}\label{ssec:literature;OF}
The body of works of the optimal filter and the filtering equations is too vast to be summarized here. Thus, we can only include here previous works that are most directly relevant to this paper, and refer the interested readers to the book of Bain and Crisan~\cite[Chapter 1.3]{BainCrisan09} for a detailed mathematical and historical account.

Stratonovich~\cite{Stratonovich59,Stratonovich60conditional} initiated the study of recursive procedures for the evolution of the optimal filter~\eqref{def:optimal filter} in discrete time, and Kushner~\cite{Kushner64} formally derived the stochastic partial differential equation characterizing the optimal filter in continuous time, now known as the Kushner--Stratonovich equation. Kallianpur and Striebel~\cite{KalStr68} introduced the change-of-measure representation of the optimal filter, now known as the Kallianpur--Striebel formula, which became the basis of the reference probability approach to nonlinear filtering. Duncan~\cite{Duncan68evaluation,Duncan70absolute}, Mortensen~\cite{Mortensen66}, and Zakai~\cite{Zakai69optimal} independently derived the linear SPDE satisfied by the unnormalized conditional law, the Duncan--Mortensen--Zakai equation; for simplicity, we refer to this equation as the Zakai equation throughout the paper. Fujisaki, Kallianpur, and Kunita~\cite{FujisakiKallianpurKunita1972} derived the SPDE characterizing the optimal filter in a more general setting than \eqref{eq:signal-obs}. We refer to these equations collectively as the filtering equations.

The vast body of works studied the well-posedness of the filtering and related equations. Just to mention a few, these equations were analyzed using functional analytic approach by Pardoux~\cite{Pardoux1979,Pardoux1982} (see also Chaleyat-Maurel, Michel, and Pardoux~\cite{ChaleyatMaurelMichelPardoux1990}), and stochastic parabolic PDE approach by Krylov and Rozovskii~\cite{KrylovRozovskii77,KrylovRozovskii1979,KrylovRozovskii82}. Fleming and Mitter~\cite{FlemingMitter1982} introduced the negative log-Zakai equation~\ref{eq:-logZakai} thereby establishing connections with stochastic optimal control. We also mention the works of Baras, BLankenship, and Hopkins~\cite{BarasBlankenshipHopkins1983}, Kurtz and Ocone~\cite{KurtzOcone1988}, and Bhatt, Kallianpur, and Karandikar~\cite{BhattKallianpurKarandikar1995}. Some of the more recent works include those of Krylov~\cite{Krylov10well-posedness} and Sun, Zhou, Yau~\cite{SunZhouYau26well-posedness} which studied well-posedness without boundedness assumptions on the coefficients, and the work of Crisan and Pardoux~\cite{CrisanPardoux26} whichestablished uniqueness in spaces of measures under remarkably general assumptions on signal and observation processes, including correlated noises and observation-dependent coefficients. However, the well-posedness results above require either boundedness assumptions on coefficients (in particular the observations), or growth or integrability conditions on the initial density or the solution. As such assumptions are incompatible with our solution class (Definition~\ref{def:Zakai;classical}) which is convenient to verify for the the unnormalized optimal filter and the gauge-transformed solution \eqref{eq:OF_gauge;summary}, our paper provides a uniqueness result tailored to our setting.

Asymptotic stability of the filter with respect to the initial data can be thought of as the well-posedness of the filtering problem, and is of fundamental importance: it implies that, regardless of the initialization, the filter will eventually reach the same conclusion as if it was initialized with the true initial law. Van Handel~\cite{vanHandel09stability} showed that, in broad generality, asymptotic stability of the signal process with nondegenerate observation noise implies that the filter is stable; see also the earlier works of Budhiraja~\cite{Budhiraja2003} Baxter, Chigansky, and Lipster~\cite{BaxChiPav04Wonham}. On the other hand, one intuitively expects filter stability when the system is observable -- i.e. no two signals with different initial laws the same observation sequence. Van Handel showed that filter stability follows from observability on compact state space~\cite{vanH09PTRF} and from the stronger condition of uniform observability on noncompact state spaces~\cite{vanHandel09uniform}; the latter applies to our linear coefficient setting when $(F,H)$ satisfies the observability condition.

Works on the quantitative stability of the optimal filter are more limited. Exponential asymptotic stability was established by Atar and Zeitouni on finite~\cite{AtarZeitouni97lyapunov} and compact~\cite{AtarZeitouni97exp} state spaces, where compactness leads to ergodicity of the signal, and by Atar~\cite{Atar98AoP} for noncompact state space under full observations. For nonergodic signals and partial observations, some quantitative stability results for nonergodic signals are available with strong assumptions on the initial data, some quantitative stability results are also available; we mention the papers of Stannat~\cite{Stannat05time-dep,Stannat06stability} based on a variational approach and the ground-state transform, and the thesis of van Handel~\cite[Chapter 4.1]{vanHandel06Thesis} using a probabilistic approach. Interestingly, Stannat applies the ground-state transform to the pathwise filter equation with time-dependent coefficients, as opposed to our approach which applies the ground-state transform to the deterministic part, which is used then in conjunction with the gauge transform~\eqref{eq:OF_gauge;summary} to deduce behavior of the filter.

From connections to linear systems theory, one might expect filter stability to hold under a weaker condition than observability, namely detectability, which, roughly speaking, only asks for observability in the unstable directions of the signal process. The short and illuminating survey~\cite{vanHandel10proceedings} elaborates on connections between filter stability and systems theory. Asymptotic stability of the optimal filter under detectability is known only the case of finite state space~\cite{vanH09PTRF}, the linear-Gaussian setting (namely the Kalman filter), or in the linear setting with non-Gaussian initial data in a weak sense (convergence of bounded uniformly continuous moments)~\cite{OconePardeoux96asymptotic}. We also mention the work of Mitter and Newton~\cite{MitterNewton03_variational} on the connection between filtering and optimal control via the variational characterization of the Bayes formula.

The Kalman filter, introduced in the pioneering work of Rudolf Kalman~\cite{Kalman60}, is a special case of the optimal filter in the linear-Gaussian setting; its continuous-time counterpart, the Kalman-Bucy filter, was introduced by Kalman and Bucy~\cite{KalmanBucy61}. 
Bucy established the global well-posedness and asymptotic stability of the covariance matrices under the assumptions of controllability and observability~\cite{Bucy67global}, and Wonham analyzed the equation under the weaker conditions of stabilizability and detectability~\cite{Wonham68riccati}. There are also several
results of a more quantitative nature: Anderson~\cite{anderson71stability} used the time-varying Lyapunov function to obtain exponential convergence rate of the covariance matrices, and Bougerol~\cite{Bougerol93} showed that the Riccati recursions are strict contractions in a natural Riemannian metric for positive definite matrices. 

Ocone and Pardoux~\cite{OconePardeoux96asymptotic} proved the asymptotic stability of the Kalman-Bucy filter in the probabilistic sense, namely that two Kalman-Bucy filter distributions with different initial data converge weakly (as probability measures) almost surely. Moreover, they derived the asymptotic exponential rates, where the exponent is (essentially) given by the spectral abscissa of the limiting observer system. Moreover, they adapted the techniques of Makowski~\cite{Makowski86} (see also Makowski and Sowers~\cite{SowersMakowski90}) and Bene\v{s} and Karatzas~\cite{BenesKaratzas83} to establish filter stability in a weak sense given non-Gaussian initial data.

\subsubsection{Large-time behavior of viscous Hamilton-Jacobi equations and linear parabolic equations}\label{ssec:literature;vHJE}

Applying the Hopf--Cole transform to the ergodic problem~\eqref{eq:ep} leads to the elliptic equation associated with the operator $\L^\ast-V+\lambda$ (where $\L^\ast$ is as in \eqref{def:L}). In the case $b=0$, this is known as the Schr\"{o}dinger operator. Simon~\cite{Simon81JFA} introduced the notion of criticality for Schr\"{o}dinger operators and studied the large-time behavior of the $L^p$-operator norms of the associated semigroup; see also his survey~\cite{Simon82BAMS}. Murata~\cite{Murata84Largetime, Murata1986} subsequently connected the large-time behavior of the Schr\"{o}dinger semigroup with the existence and growth properties of positive solutions of the corresponding elliptic equation.

In a more general setting where $b$ is a gradient vector field, Agmon~\cite{Agmon82} introduced the notion of a ground state as a positive solution of minimal growth at infinity. Pinchover~\cite{Pinchover88Duke} systematically developed the criticality theory for the case  general non-gradient $b$, in terms of positive solutions and positive minimal Green functions. Pinsky~\cite{Pinsky95book} further developed the probabilistic interpretation of this theory, relating criticality and subcriticality to recurrence and transience, respectively, of the diffusions obtained through Doob's $h$-transform which corresponds to $X^\phi$ in \eqref{def:Xtphi} in our setting. For a more thorough historical account, see the notes and discussion in~\cite[Chapter~4]{Pinsky95book}.

In the context of viscous Hamilton--Jacobi equations, Ichihara~\cite{Ichi11recurrence,Ichihara12} and Ichihara and Sheu~\cite{IchiharaSheu13} studied the ergodic problem, recurrence properties of the associated diffusions, and the large-time behavior of solutions using a combination of viscosity-solution methods, stochastic optimal control, and probabilistic techniques building on Pinsky's approach. On the other hand, using viscosity solutions and PDE techniques, Barles and Meireles~\cite{barles2016unbounded} studied the ergodic problem \eqref{eq:ep} and Barles, Quaas, and Rodr\'{i}guez-Paredes~\cite{barles2020large} established large-time behavior of the solutions of the viscous Hamilton-Jacobi equations.

\section{Asymptotic behavior of associated PDEs and the role of detectability}\label{sec:C1}

Let $b\in C(\R^d;\R^d)$, and $V\in C(\R^d)$. We consider initial value problems for the linear parabolic PDE
\begin{equation}\label{eq:parabolic}
\begin{split}
    \partial_t\varrho=\frac12\Delta \varrho - \nabla\cdot(b\varrho)-V\varrho \text{ in } (0,\infty)\times\R^d,
\end{split}
\end{equation}
where initial trace is interpreted in the vague trace sense: $\varrho(t,x)\;dx \rightharpoonup \nu\in\M_+(\R^d)$ vaguely. Equivalently, we consider the negative log potential $w=-\log\varrho$, which satisfies the viscous quadratic Hamilton-Jacobi equation
\begin{equation}\label{eq:hje}
    \partial_t w -\frac{1}{2}\Delta w + \frac12|\nabla w|^2 + b\cdot\nabla w = V+\nabla\cdot b \text{ in } (0,\infty)\times\R^d,
\end{equation}
where initial data is interpreted in the vague trace sense for $t\mapsto e^{-w(t,x)}\;dx$.

We are interested in the asymptotic behavior of the solution as $t\rightarrow\infty$. To this end, we consider the eigenvalue problem for the corresponding elliptic equation: let $\bar\varrho:\R^d\rightarrow\R_+$ satisfy
\begin{equation}\label{eq:parabolic;eigenfunc}
    \frac12\Delta\bar\varrho-\nabla\cdot(b\bar\varrho)-V\bar\varrho +\lambda\bar\varrho =0 \text{ in } \R^d,\quad \int_{\R^d}\bar\varrho(x)\;dx <+\infty;
\end{equation}
often we will set $\int_{\R^d}\bar\varrho(x)\;dx =1$.
Equivalently, the negative log potential $\phi=-\log\bar\varrho$ is the ergodic profile satisfying the ergodic problem
\begin{align}\label{eq:ep}
    \lambda-\frac12\Delta\phi +\frac12|\nabla\phi|^2+ b\cdot\nabla\phi=V+\nabla\cdot b \quad\text{ in }\R^d,
\end{align}
and characterizes the large-time behavior of the solution to \eqref{eq:hje}. Note that if \eqref{eq:parabolic;eigenfunc} holds, $-\lambda$ is the eigenvalue of the operator $\L^\ast \rho=\frac12\Delta \rho - \nabla\cdot(b\rho)$; we have chosen this sign convention so that $(t,x)\mapsto \lambda t+\phi(x)$ satisfies the equation \eqref{eq:hje}.

The goal of this section is twofold. Firstly, we study the asymptotic behavior of the solutions of \eqref{eq:parabolic}-\eqref{eq:hje} under the following assumptions.
\begin{assumption}\label{ass:1}
    \begin{listi}
    \item $b \in C^{1,\alpha}_{loc}(\R^d;\R^d)$, $b(x)\leq C(1+|x|)$ for some $C>0$;
    \item $V\in C^{0,\alpha}_{loc}(\R^d)$ is nonnegative;
\end{listi}
\end{assumption}

\begin{assumption}\label{ass:2}
\begin{listi}
    \item There exists a solution $(\lambda,\phi)$ to \eqref{eq:ep} such that $e^{-\phi}\in L^1(\R^d)$ 
    and the diffusion process $X^\phi$~\eqref{def:Xtphi},
    \[dX_t^\phi = - (b(X_t)+\nabla\phi(X_t))dt +dB_t, \quad t\geq 0\]
    is nonexplosive and ergodic; $B=(B_t)_{t\geq 0}$ here is a standard Brownian motion.

    \item The invariant measure $\m_\phi$ of $X^\phi$ satisfies the log-Sobolev inequality
    \begin{equation}\label{eq:lsi}
        \tag{LSI}\Ent_{\m_\phi}(g)\leq \frac{1}{2\kappa^\phi}\int_{\R^d}\frac{|\nabla g|^2}{g}d\mu \text{ for all }g\in C_c^\infty(\R^d;\R_+).
    \end{equation}
\end{listi}

\end{assumption}

Note that $\phi$ in Assumption~\ref{ass:2} is defined only up to an additive constant; this will be inconsequential for the optimal filter after normalization. 

These assumptions are formulated for our second goal, which is to verify that these are satisfied in the linear setting where $b(x)=Fx$ and $h(x)=Hx$  (with $F\in \R^{d\times d},\;H\in\R^{k\times d}$) satisfy the detectability condition from the linear systems theory, which we define below.

\begin{assumption}[Detectability]\label{ass:detectable}
    There exists a matrix $L\in\R^{d\times k}$ such that
    \begin{equation}\label{eq:detectable}
        \alpha(\sF-L\sH):=\max\{\Re(\lambda):\;\lambda\text{ is an eigenvalue of } \sF- L\sH\}<0.
    \end{equation}
    If \eqref{eq:detectable}, we say the pair $(\sF,\sH)$ is detectable.
\end{assumption}
In our setting~\eqref{eq:signal-obs} with nondegenerate signal noise, which implies controllability, detectability is a necessary and sufficient structural condition on $F$ and $H$ for asymptotic stability of the Kalman-Bucy filter~\cite{Wonham68riccati,anderson71stability}. 

This section is organized as follows. After establishing the basic properties of the solution of the ergodic problem $\phi,\m_\phi$ in Section~\ref{ssec:C1:phi,mphi}, we study the diffusion semigroup generated by $X^\phi$~\eqref{def:Xtphi} in Section~\ref{ssec:C1:diffusion_semigroup}. Then, Section~\ref{ssec:C1;subprob} discusses the relationship between the semigroup solution and subprobability solution of \eqref{eq:parabolic} via Doob's $h$-transform, which will be useful later in specifying a unique solution of the equation, which is nontrivial due to the interpretation of initial datum in the vague trace sense. All these sections are formulated for general nonlinear $b,h$ satisfying Assumptions~\ref{ass:1}-\ref{ass:2}. Then, Section~\ref{ssec:consistency} states the main result of this section, Lemma~\ref{lem:detectability}, which verifies that Assumption~\ref{ass:detectable} implies Assumption~\ref{ass:2}, and establishes useful facts for later sections.

As already mentioned in Section~\ref{sssec:PDE}, most statements in the current section concerning the general equations \eqref{eq:parabolic}-\eqref{eq:hje} are already known in some form, and we do not claim any novelty. Yet, we provide complete proofs of each statement as either the corresponding results are not available in the precise form desirable for our setting, or doing so avoids introducing heavier machinery (and notation) unnecessary for our purpose; we have tried to provide reference to related results when possible. Thus, this section is best understood as providing a framework of preliminary results upon which the later sections can be conveniently built on, treating with care subtleties that arise in our setting.

\subsection{Ergodic problem and invariant measure}\label{ssec:C1:phi,mphi}
From the classical interior Schauder estimates, regularity of $\phi$ and the density of $\m_\phi$ follows.
\begin{lemma}\label{lem:coeff_reg}
    Let Assumption~\ref{ass:1} (i)-(ii) and Assumption~\ref{ass:2} (i) hold. Then the solution $\phi$ to the ergodic problem \eqref{eq:ep} and the Lebesgue density of the invariant measure $\m_\phi$ is in $C_{loc}^{2,\alpha}$.
\end{lemma}

\begin{proof}
    As $\bar\varrho=e^{-\phi}$ solves \eqref{eq:parabolic;eigenfunc} the interior Schauder estimates imply $\bar\varrho,\phi\in C_{loc}^{2,\alpha}$. Moreover, the density of $\m_\phi$, which we temporarily denote by $m_\phi$ here, solves the elliptic equation with $C^{1,\alpha}$ coefficients
    \[\frac12\Delta m_\phi + \nabla\cdot((b+\nabla\phi) m_\phi)=\frac12\Delta m_\phi +(b+\nabla\phi)\cdot\nabla m_\phi + \nabla\cdot(b+\nabla\phi) m_\phi=0.\]
    Thus, again by the interior Schauder estimates, we conclude $m_\phi \in C_{loc}^{2,\alpha}$ density.
\end{proof}

From the criticality theory of linear elliptic operators, it follows that there is at most one $\phi$ satisfying Assumption~\ref{ass:2} (i). We leave this uniqueness as a remark to avoid rigorously introducing objects in the criticality theory, which will not be necessary for the remainder of this paper; we refer the interested readers to the book of Pinsky~\cite{Pinsky95book} on the subject. 
\begin{remark}[Uniqueness of $\phi$]\label{rmk:phi_unique}
    From the criticality theory of elliptic operators, one can deduce uniqueness (up to an additive constant) of the solution $(\lambda,\phi)$ of \eqref{eq:ep} such that $X^\phi$ is ergodic. For a second order elliptic operator $L$, $L-\lambda$ is critical (or equivalently, $\lambda$ is a critical eigenvalue for $L$) if $L-\lambda$ possesses a nonnegative harmonic function but not a Green's function. Criticality of the associated elliptic operator has deep connections recurrence of the corresponding diffusion and the relationship is carefully detailed in the book of Pinsky~\cite{Pinsky95book}. Temporarily, let $h$ be a sufficiently regular, strictly positive function (not to be confused with the observation vector field). Writing
    \[Lu = \frac12\nabla\cdot(A\nabla u) + b\cdot\nabla u- Vu = L_0 u + cu,\]
    Doob's $h$-transform~\cite{Doob57htransform}
    \begin{equation}\label{eq:Doobh}
    L^h u := \frac{1}{h}L(h u)= L_0 u + A\nabla\log h \cdot \nabla u + \frac{Lh}{h}
    \end{equation}
    preserves criticality~\cite[Theorem 4.3.3. (iv); see also Theorem 4.2.2]{Pinsky95book}. Taking $L_0=\L$, $c=-V$, and $h=\bar\varrho=e^{-\phi}$, $L\bar\varrho=-\lambda\bar\varrho$. Thus 
    \[\L_\phi:=\frac12\Delta-(b-\nabla\phi)\cdot\nabla =L^h + \lambda.\] 
    Thus, the generator $\L_\phi$ of $X^\phi$ \eqref{def:Xtphi} is critical iff $\L-V+\lambda$ is critical.

    We outline a short argument for the uniqueness. Suppose $(\lambda,\phi),(\lambda',\phi')$ are such solutions. Then direct calculations reveal that $\L_\phi e^{-\phi'+\phi}=(\lambda-\lambda')e^{-\phi'+\phi}$. Recurrence of $X^\phi,X^{\phi'}$ implies that $\L_\phi,\L_{\phi'}$ are critical, hence that $-\lambda,-\lambda'$ are critical eigenvalues of $\L-V$~\cite[Theorem 4.3.3. (i), (iv)]{Pinsky95book}. As the critical eigenvalue for $\L$, if it exists, is unique, $\lambda=\lambda'$~\cite[Theorem 4.3.2]{Pinsky95book}. In particular, $e^{-\phi'+\phi}$ is a positive harmonic function for the operator $\L_\phi$, which has no zeroth order term, thus it must be a constant by the Liouville-type theorem~\cite[Corollary 4.3.6]{Pinsky95book}.
\end{remark}

\subsection{Diffusion semigroup}\label{ssec:C1:diffusion_semigroup}

In this section, we consider the $L^1(\m_\phi)$-semigroup generated by the diffusion $X^\phi=(X_t^\phi)_{t\geq 0}$. Define the diffusion semigroup $(P_t^\phi)_{t\geq 0}$
\begin{equation}\label{def:Ptf}
    P_t^\phi f(x)=\E^x[f_0(X_t^\phi)] \text{ for bounded measurable } f:\R^d\rightarrow\R.
\end{equation}
As $P_t^\phi$ is a contraction on any $L^p(\m_\phi)$ for $1\leq p<+\infty$ by Jensen's inequality, the semigroup extends naturally to these spaces. Unless otherwise stated, we will consider $P_t^\phi$ as an operator from $L^1(\m_\phi)$ to itself. The infinitesimal generator $\L_\phi$ is defined by
\begin{equation}\label{def:Lphi}
    \L_\phi f := \lim_{t\rightarrow 0}\frac{P_tf - f}{t} = \frac12\Delta f- (b+\nabla\phi)\cdot\nabla f,
\end{equation}
where the equality on the right is understood in $L^1(\m_\phi)$.
Moreover, when $f,g$ are sufficiently nice functions, the dual operator $\L_\phi^\ast$ with respect to the $\m_\phi$-duality pairing -- i.e. 
\begin{equation}\label{eq:Lphi;duality}
    \langle \L_\phi f,g\rangle_{\m_\phi}=\langle f,\L_\phi^\ast g\rangle_{\m_\phi}
\end{equation}
is given by
\begin{equation}\label{def:Lphi^star}
    \L_\phi^\ast g= \frac12\Delta g + (\nabla\log\m_\phi+b+\nabla\phi)\cdot\nabla g, 
\end{equation}
where the equality is understood in duality with $L^1(\m_\phi)$. We will only apply the duality formula \eqref{eq:Lphi;duality} for nice functions (only once, in fact, in the proof of Lemma~\ref{lem:ent_dissipation}), thus we will not be too careful with the domains of the generators and the duals in the remainder of this paper. 

As the drift coefficient satisfies $b+\nabla\phi\in C_{loc}^{1,\alpha}$ (Lemma~\ref{lem:coeff_reg}), the semigroup $(P_t f)_{t>0}$ solves the corresponding parabolic equation in the classical sense, with $L^1(\m_\phi)$-initial trace. We record this in the definition below.
\begin{definition}[Semigroup solution]\label{def:semigroup}
    Let $0\leq f_0\in L^1(\m_\phi)$. Then we say
    \[f(t,x):= P_t^\phi f_0(x)=\E[f_0(X_t^\phi)|X_0=x]=\E^x[f_0(X_t^\phi)]\]
    is the $L^1(\m_\phi)$-semigroup solution (or simply the semigroup solution) of the equation
    \begin{equation}\label{eq:PDE;semigroup}
        \begin{cases}
            \partial_t f = \L_\phi f = \frac12\Delta f-(b+\nabla\phi)\cdot\nabla f \text{ in } (0,\infty)\times\R^d,\\
            f_t \rightarrow f_0 \text{ in } L^1(\m_\phi) \text{ as } t\rightarrow 0.
        \end{cases}
    \end{equation}
\end{definition}

Given the invariant measure, the usual entropy dissipation formula for diffusion holds despite the nonreversibility. This seems to be widely known and has been proven in generality by Fontbona and Jourdain~\cite[Corollary 1.6]{FontbonaJourdain16}; see also \cite{ArnoldCarlenJu08}. Nevertheless, we include an analytic proof here that is more suited for our setting.
\begin{lemma}\label{lem:ent_dissipation}
    Let Assumption~\ref{ass:2} (i) hold, and $f_t:=P_t^\phi f_0$ with $f_0\in L^1(\m_\phi)$ satisfying $\Ent_{\m_\phi}(f_0)<+\infty$. Then
    \begin{equation}\label{eq:Ent_dissipation}
        \Ent_{\m_\phi}(f_t)\leq \Ent_{\m_\phi}(f_0)-\frac12\int_0^t \int_{\R^d}\frac{|\nabla f_s|^2}{f_s}\;d\m_\phi\;ds.
    \end{equation}
    In particular, under Assumption~\ref{ass:2} (ii),
    \begin{equation}\label{eq:Ent_decay}
        \Ent_{\m_\phi}(f_t)\leq e^{-\kappa_\phi t}\Ent_{\m_\phi}(f_0).
    \end{equation}
\end{lemma}

\begin{proof}
    The entropy decay \eqref{eq:Ent_decay} immediately follows from \eqref{eq:Ent_dissipation}, the log-Sobolev inequality \eqref{eq:lsi}, and the Gronwall inequality.

    Let us first assume de Bruijn's identity
    \begin{equation}\label{eq:deBruijn}
        \frac{d}{dt}\Ent_{\m_\phi}(P_t^\phi g) = - \frac12\int_{\R^d}\frac{|\nabla P_t^\phi g|^2}{P_s^\phi g}\;d\m_\phi  \text{ for every } g\in C_c^\infty(\R^d;\R_+)
    \end{equation}
    holds for our diffusion semigroup; we will verify it at the end of the proof. By integrating~\eqref{eq:deBruijn}, we obtain
    \begin{equation}\label{eq:Ent_dissipation;ineq}
        \Ent_{\m_\phi}(P_t^\phi g)= \Ent_{\m_\phi}(g)-\frac12\int_0^t \int_{\R^d}\frac{|\nabla P_s^\phi g|^2}{P_s^\phi g}\;d\m_\phi\;ds \text{ for every } g\in C_c^\infty(\R^d;\R_+).
    \end{equation}
    Now, let $g_n\in C_c^\infty(\R^d;\R_+)$ be a sequence such that 
    \[\lim_{n\rightarrow\infty}\|g_n-f_0\|_{L^1(\m_\phi)},\quad \lim_{n\rightarrow\infty}\Ent_{\m_\phi}(g_n)=\Ent_{\m_\phi}(f_0), \text{ and } \int \int P_t^\phi g_n\;d\m_\phi=1;\]
    the first two are possible by standard truncation and mollification, whereas the normalization does not cause any issues due to the one-homogeneity of $\Ent_{\m_\phi}$. As $P_t^\phi$ is an $L^1(\m_\phi)$-contraction, $P_t^\phi g_n\xrightarrow[]{n\rightarrow\infty} f_t$ in $L^1(\m_\phi)$. Due to nonnegativity, note that this implies that $(P_t^\phi g_n) \m_\phi$ converges to $f_t \m_\phi$ in total variation, hence narrowly.
    
    Then we may obtain the inequality \eqref{eq:Ent_dissipation} by lower semicontinuity of $\Ent_{\m_\phi}$ and the Fisher information. More precisely, relative entropy $(\mu,\nu)\mapsto \kl(\mu||\nu):=\Ent_{\nu}(d\mu/d\nu)$ is lower semicontinuous with respect to the narrow convergence, thus
    \begin{equation}
    \Ent_{\m_\phi}(f_t)=\kl(f_t\m_\phi||\m_\phi)\leq\liminf_{\eps\downarrow 0}\kl((P_t^\phi g_n) \m_\phi||\m_\phi)\leq\liminf_{\eps\downarrow 0}\Ent_{\m_\phi}(P_t^\phi g_n),
    \end{equation}
    whereas the Fisher information lower semicontinuous with respect to the weak topology of $L^1(\m_\phi)$\cite[Lemma 4.10]{AGS14Inventiones}. 

    Thus, it remains to show \eqref{eq:deBruijn}. As in the proof in the reversible case (e.g. \cite[Proposition 5.2.2]{BGL14}), all the calculations are justified first for functions for functions in $C_c^\infty(\R^d;\R_+)+\eps$ 
    for some $\eps>0$ to avoid issues with the logarithm, then extended by $\eps\downarrow 0$. From direct calculations, we can readily verify that the $L^2(\m_\phi)$-adjoint $\L_\phi^\ast$ of $\L_\phi$ is given by
    \[\L_\phi^\ast=\frac12\Delta + (\nabla\log\m_\phi+b+\nabla\phi)\cdot\nabla.\]
    Accordingly, define the symmetric and antisymmetric part
    \begin{align*}
        \L_\phi^{sym}&:=\frac12(\L_\phi+\L_\phi^\ast)=\frac12\Delta+\frac12\nabla\log\m_\phi\cdot\nabla
        \\ \L_\phi^{anti}&:= \frac12(\L_\phi-\L_\phi^\ast)=-(b+\nabla\phi+\frac12\nabla \log\m_\phi)\cdot\nabla.
    \end{align*}
    As $\L_\phi^{anti}$ is a differential operator of first order, $\L_\phi^{anti}\log u=\L_\phi^{anti}u/u$, and thus
    \begin{align*}
        \int_{\R^d} \log u\L_{\phi}^{anti} u\;d\m_\phi&=-\int_{\R^d} u \L_\phi^{anti}\log u \;d\m_\phi = -\int_{\R^d} \frac{u\L_\phi^{anti}u}{u}\;d\m_\phi 
        \\&= -\int_{\R^d} \L_\phi^{anti} u\;d\m_\phi= \int_{\R^d} u\L_\phi^{anti} 1\;d\m_\phi =0.
    \end{align*}
    Thus, the entropy dissipation in the nonreversible case reduces to the well known calculations in the reversible case
    \begin{align*}
        \frac{d}{dt}\Ent_{\m_\phi}(P_t g)&=\int_{\R^d} (1+\log P_t g)\L_\phi P_t g\;d\m_\phi = \int_{\R^d} \log P_t g \L_\phi P_t g d\m_\phi 
        \\&= \int_{\R^d} \log P_t g\L_\phi^{sym} P_t g\;d\m_\phi=-\frac12\int_{\R^d}\frac{|\nabla P_t g|^2}{P_t g}\;d\m_\phi.
    \end{align*}
    This completes the proof.
\end{proof}

\subsection{Doob's \texorpdfstring{$h$}{h}-transform}\label{ssec:C1;subprob}

We now relate the semigroup solution from Definition~\ref{def:semigroup}, with the solution of the linear parabolic equation~\eqref{eq:parabolic}.

The correspondence is formally provided by Doob's $h$-transform, which can be used to the zeroth order term of the elliptic operator using the ground state. As we already saw in Remark~\ref{rmk:phi_unique}  \eqref{eq:Doobh}, Doob's $h$-transform is defined as follows: given an elliptic operator $L$ and a strictly positive regular function $h$, its $h$-transform $L^h$ is defined by
\[L^h u:=\frac{1}{h}L (hu).\]
Applying Doob's $h$-transform to $L=\L-V$ and $h=\bar\varrho$, $L^h=\L_\phi-\lambda$. Thus, given the ground state $\bar\varrho$, we can understand the behavior of $\varrho$ satisfying $\partial_t\varrho=(\L-V)\varrho$ through $f$ that satisfies $\partial_t f=\L_\phi f$, which is well-behaved (Lemma~\ref{lem:ent_dissipation}).

The challenge in making this correspondence rigorous for our setting lies in the (surprising) subtlety with uniqueness of the solution of the parabolic equation \eqref{eq:parabolic} with vague initial trace, which is well documented in \cite[Chapter 9]{BKRS2022FPK}. In order to rigorously establish a bijection between the semigroup solution $f$, which is unique, and $\varrho$, we need to single out a particular solution $\varrho$. We will use the notion of subprobability solution, and establish the correspondence in Proposition~\ref{prop:subprob_sol}.

We state the existence and uniqueness of the subprobability solution of the Fokker-Planck equation with nonnegative zeroth order coefficient due to Bogachev, Krylov, R\"{o}ckner and Shaposhnikov~\cite[Theorem 6.6.2, Theorem 9.4.3.]{BKRS2022FPK}, simplified to our setting. We emphasize that we are borrowing only a small fraction of the power of their results, which is applicable to a much more general setting, in order to resolve the uniqueness issue with vague initial trace.
\begin{proposition}\label{prop:subprob_uniqueness}
    Fix $T>0$, and let $\varrho:(0,T)\times\R^d\rightarrow\R_+$ with $\int_{\R^d}\varrho(t,x)\;dx \leq 1$ for all $t\in (0,T)$. We say $\varrho$ is a subprobability solution of \eqref{eq:parabolic} on $(0,T)\times\R^d$ with initial trace $\nu\in\P(\R^d)$ if it is a classical solution on $(0,T)\times\R^d$ and
    \begin{listi}
        \item (Vague initial trace) $\int_{\R^d} \psi(x)\varrho(t,x)\;dx\xrightarrow[]{t\rightarrow 0} \int_{\R^d} \psi(x)\;d\nu(x)$ \text{ for each } $\psi\in C_c(\R^d)$.
        \item (Appropriate loss of mass) for almost every $t\in(0,T)$:
        \begin{equation}\label{eq:massloss}
            \int_{\R^d} \varrho(t,x)\;dx\leq \nu(\R^d)-\int_0^t \int_{\R^d} V(x)\varrho(s,x)\;dx\;ds.
        \end{equation}
    \end{listi}
    Under Assumption~\ref{ass:1}, the subprobability solution of \eqref{eq:parabolic} with initial trace $\nu\in\P(\R^d)$ exists and is unique.
\end{proposition}

\begin{remark}\label{rmk:subprob_uniqueness}
    We detail a few simplicifications we have made from the much more general result of Bogachev, Krylov, R\"{o}ckner and Shaposhnikov. Firstly, the H\"{o}lder regularity of $b,V$ makes local integrability conditions trivial, and moreover allows us to formulate the definition in terms of classical solutions rather than weak solutions. Secondly, constant diffusion coefficient trivially satisfies all assumptions. Thirdly, the linear growth condition reduces integrability condition $|b|/(1+|x|)\in L^1(\rho(x,t)\;dxdt,\R^d\times(0,T))$ of~\cite[Theorem 9.4.3.]{BKRS2022FPK} to
    \begin{equation}\label{eq:totalmass_bd}
        \int_0^T \varrho(t,x)\;dx <+\infty.
    \end{equation}
    However, this follows from \eqref{eq:massloss}, as $V\geq 0$. Finally, if \eqref{eq:massloss} holds with the the spacetime integral of $V\rho$ being possibly infinite, the integral must be finite, as the left-hand side is nonnegative and $\nu(\R^d)$ is finite; thus the global integrability condition $V\in L^1(\rho(t,x)dxdt,\R^d\times(0,T))$ in the definition follows.
\end{remark}

Let $u$ be a solution of \eqref{eq:hje} and $(\lambda,\phi)$ of \eqref{eq:ep}. Let $v=u-\lambda t-\phi$, $w = u-\phi$. As $\nabla w = \nabla v$, we have
\begin{align*}
    \partial_t v=\partial_t u -\lambda
    =&\frac12\Delta u - \frac12|\nabla u|^2-b\cdot\nabla u +V+\nabla\cdot b
    \\&-\frac12\Delta \phi+\frac12|\nabla\phi|^2+b\cdot\nabla\phi - V-\nabla\cdot b
    \\=&\frac12\Delta v - \frac12 (\nabla u+\nabla\phi)\cdot\nabla v - b\cdot\nabla v
    \\=&\frac12\Delta v - \frac12 (\nabla w+2\nabla\phi)\cdot\nabla v - b\cdot\nabla v
    \\=&\frac12\Delta v -\frac12|\nabla w|^2-(b+\nabla\phi)\cdot\nabla v.
\end{align*}
Thus $f_t(x):= e^{-v}$ satisfies precisely the equation \eqref{eq:PDE;semigroup}. Provided $(\lambda,\phi)$, or equivalently $\bar\varrho$, this suggests a one-to-one correspondence between the semigroup $P_t^\phi$ and the solution of \eqref{eq:parabolic}. We make this correspondence precise in the following proposition.

\begin{proposition}\label{prop:subprob_sol}
    Let $\varrho_0\in L^1(\R^d)$ be a nonnegative function with mass at most one, additionally satisfying $f_0:=\varrho_0 e^{\phi}\in L^1(\m_\phi)$. Then the unique subprobability solution of the Fokker-Planck equation \eqref{eq:parabolic} with vague initial trace $\varrho_0(x)\;dx$ is given by
    \begin{equation}\label{eq:subprobsol}
        \varrho_t(x):=f_t(x)e^{-\lambda t-\phi(x)}
    \end{equation}
    where $f_t$ is the semigroup solution of \eqref{eq:PDE;semigroup} with initial data $f_0$.
\end{proposition}

\begin{proof}
    By the uniqueness of the subprobability solution, Theorem~\ref{prop:subprob_uniqueness}, it suffices to show that $\varrho_t$ defined as in \eqref{eq:subprobsol} is indeed a solution of \eqref{eq:parabolic} satisfying items (i)-(ii).

    That $\varrho_t$ solves \eqref{eq:parabolic} follows from the calculations preceding the proposition.

    To verify item (i), let $\psi\in C_c(\R^d)$ and note
    \begin{align*}
        \int_{\R^d} \psi(x)\varrho_t(x)\;dx = \int_{\R^d} e^{-\lambda t-\phi(x)}\psi(x) f_t(x)\;dx;
    \end{align*}
    as $e^{-\phi}\psi\in C_c(\R^d)$ is compactly supported whereas $f_t\rightarrow f_0$ in $L^1(\m_\phi)$ hence vaguely, it immediately follows that $\rho_t(x)\;dx$ converge vaguely to $\varrho_0(x)\;dx$.

    To verify item (ii), we first assume $f_0\in L^\infty(\m_\phi)$, which will be relaxed by approximation later. By Jensen's inequality, 
    \[\|P_t f_0\|_{L^\infty(\m_\phi)}\leq \|f_0\|_{L^\infty(\m_\phi)}.\]
    Thus,
    \[|\varrho_t(x)|=f_t(x)e^{-\lambda t-\phi(x)}\leq e^{-\lambda t}\|f_0\|_{L^\infty(\m_\phi)}e^{-\phi(x)},\]
    and thus $\int_{\R^d} \rho_t(x)\;dx$ is uniformly bounded on any finite time horizon $t\in[0,T]$.
    
    For each $R>0$, let $\chi_R$ be a smooth cutoff function that is 1 on $\overline B(0,R)$ and vanishes outside of $\overline B(0,2R)$. Using $\chi_R$ as a test function, we have
    \begin{align*}
        \int_{\R^d} \chi_R \varrho_t\;dx - \int_{\R^d} \chi_R\varrho_0\;dx = \int_0^t \int_{\R^d} \left(\frac12\Delta\chi_R +b\cdot\nabla\chi_R\right)\varrho_s\;dx\;ds-\int_0^t \int_{\R^d} \chi_R V\varrho_s\;dxds.
    \end{align*}
    We can pass to the limit as $R\rightarrow\infty$ by the monotone convergence all but the first term on the right-hand side; note there is no issue with infinity as $\varrho_0\in L^1(\R^d)$. Thus, it suffices to show that
    \[\int_0^t \int_{\R^d} \left(\frac12\Delta\chi_R +b\cdot\nabla\chi_R\right)\rho_s\;dx\;ds\rightarrow 0 \text{ as }R\rightarrow\infty.\]
    Indeed,
    \[|\nabla \chi_R(x)|\lesssim R^{-1}, \qquad |\Delta\chi_R(x)|\leq R^{-2} \text{ for } R\leq |x|\leq 2R,\]
    with derivatives vanish outside of the annulus, on which $|b(x)|\lesssim R$ as $R\rightarrow\infty$, we can find some constant $C>0$ such that
    \begin{align*}
        \left|\frac12\Delta\chi_R(x) +b\cdot\nabla\chi_R(x)\right|\rho_s(x)\leq C \one_{|x|\geq R}\varrho_s(x).
    \end{align*}
    Of course the right-hand side is dominated by $\varrho_s(x)$ which is integrable, 
    we can let $R\rightarrow\infty$ and apply the dominated convergence theorem to verify the claim.

    Thus we have shown that
    \begin{align*}
        \int_{\R^d} \varrho_t \;dx = \int_{\R^d} \varrho_0\;dx - \int_0^t\int_{\R^d} V\varrho_s\;dxds 
    \end{align*}
    whenever $f_0\in L^\infty(\m_\phi)$. Note each integrand is nonnegative, whereas $P_t (f_0\wedge n) \uparrow P_t f_0$ pointwise $\m_\phi$-a.e. as $n\rightarrow\infty$. Thus, for arbitrary $f_0\in L^1(\m_\phi)$ approximate by $f_0\wedge n$ and conclude by the monotone convergence theorem; as $\varrho_0\in L^1$ and $\varrho_t\geq 0$, we conclude that the limiting integrals are all finite.
\end{proof}

\subsection{Consistency with linear systems theory}\label{ssec:consistency}
We now restrict our attention to the linear coefficient case, where $b(x)=Fx$ and $h(x)=Hx$ for some $F\in\R^{d\times d}$ and $H\in\R^{k\times k}$. Under detectability of $(F,H)$, we verify that Assumption~\ref{ass:2} holds, and establish facts about the invariant measure and speed of convergence of the covariance of $X_t^\phi$, which will be essential for our quantitative stability results in Section~\ref{sec:asymp}.

We verify that detectability \eqref{ass:detectable} implies Assumption~\ref{ass:2}.
\begin{lemma}\label{lem:detectability}
    Suppose $(F,H)$ satisfy Assumption~\ref{ass:detectable}. Then the solution of the ergodic problem \eqref{eq:ep} with $b(x)=\sF x$ and $f(x)=|\sH x|_{\Gamma}^2+\tr\sF$ is given by 
    \begin{equation}\label{def:lambda,phi;detectable}
        \lambda=\frac12\tr R_\phi+\tr\sF, \qquad \phi(x)=\frac12x^T R_\phi x
    \end{equation}
    where $R_\phi\in\R^{d\times d}_{++}$ is the solution of the matrix Riccati equation
    \begin{equation}\label{def:Riccati:P}
    \sF^TR_\phi+R_\phi\sF+R_\phi^2-\sH^T\Gamma^{-1}\sH   =0.
    \end{equation}
    Moreover, 
    \begin{equation}\label{def:alpha_phi}
        \alpha_\phi:=-\alpha(-F-R_\phi)>0,
    \end{equation}
    hence the diffusion~\eqref{def:Xtphi} is the Ornstein-Uhlenbeck process
    \begin{equation}\label{eq:dXphi;linear}
        dX_t^\phi = -(F+R_\phi)X_t^\phi\;dt + dB_t,
    \end{equation}
    with $\law X_t^\phi = N(0,Q_t)$, where the covariance $Q_t$ is defined by
    \begin{equation}\label{eq:Qt}
        Q_t=\int_0^t e^{-(\sF+R_\phi)s}e^{-(\sF+R_\phi)^T s}\;ds.
    \end{equation}
    In particular, $Q_\infty:=\lim_{t\rightarrow\infty}Q_t$ is well-defined, and $\phi$ satisfies Assumptions~\ref{ass:1}-\ref{ass:2}, with the invariant measure $\m_\phi=N(0,Q_\infty)$, and satisfies
    \begin{equation}\label{eq:Q_inf<Sigma_phi}
        \Sigma_\phi := R_\phi^{-1}\succeq Q_\infty.
    \end{equation}
    Finally, for each $\eps\in(0,\alpha_\phi)$, there exists $C(\eps,M)>0$ such that
    \begin{equation}\label{eq:Q_infty-Q_t}
        \|Q_\infty-Q_t\|\leq C(\eps,M) e^{-2(\alpha_\phi-\eps) t}\|Q_\infty\|.
    \end{equation}
\end{lemma}

\begin{proof}
    The solution of \eqref{eq:ep} in this case should be the quadratic $\phi(x)=\frac12 x^T R_\phi x$, where the matrix $R_\phi$ is characterized by the matrix Riccati equation \eqref{def:Riccati:P}.
    Indeed, plugging in the Ansatz $\phi(x)=x^T R_\phi x$ to \eqref{eq:ep}, we have
    \begin{align*}
    \lambda&=-\frac12\Delta \phi(x) +\frac12 |\nabla\phi(x)|^2 + (\sF x)^T \nabla\phi(x) - \frac12|\sH x|_{\sG}^2-\tr \sF\\
    &=-\frac12\tr R_\phi +\frac12 x^T R_\phi^2 x + x^T \sF^T R_\phi x -\frac12 x \sH^T\sG^{-1}\sH x -\tr\sF.
    \end{align*}
    By considering the quadratic and constant terms, we deduce \eqref{def:Riccati:P} and that $\lambda=\frac12\tr R_\phi+\tr \sF$. Moreover, by detectability of $(F,H)$ we conclude $R_\phi\succ 0$; see for instance \cite[Theorem 9.1.2]{LanRod95}.
    
    Now consider the controlled diffusion
    \[dX_t^\phi = -(\sF X_t^\phi+ \nabla\phi(X_t^\phi))\;dt+ dW_t,\]
    which by plugging in $\phi(x)=\frac12 x^T R_\phi x$, is the Ornstein-Uhlenbeck process the form \eqref{eq:dXphi;linear}, which is explicitly solvable, with $\law X_t^\phi=N(0,Q_t)$ with $Q_t$ as defined in \eqref{eq:Qt}.
    
    Note we can reorganize \eqref{def:Riccati:P} into the form
    \begin{equation}\label{eq:Riccati:P;Lyapunov}
    (-\sF- R_\phi)^T R_\phi + R_\phi (-\sF- R_\phi)=-\sH^T\Gamma^{-1}\sH-R_\phi^2.
    \end{equation}
    Since $R_\phi\succ 0$ and the right-hand side is negative definite, we deduce $\alpha(-\sF-R_\phi)<0$ by the Lyapunov stability condition. In particular, \eqref{eq:dXphi;linear} is ergodic with the invariant measure $\m_\phi= N(0,Q_\infty)$, where $Q_\infty\in\R^{d\times d}_{++}$ is a well-defined positive definite matrix; in particular, $\m_\phi$ satisfies Assumption~\ref{ass:2}.

    Finally, to see \eqref{eq:Q_inf<Sigma_phi}, first multiply the Riccati equation \eqref{eq:Riccati:P;Lyapunov} by $\Sigma_\phi$ on each side to obtain
    \begin{equation}\label{eq:Riccati;Sigma;Lyapunov}
        \Sigma_\phi(-\sF-R_\phi)^{T} + (-\sF-R_\phi)\Sigma_\phi  = -\Sigma_\phi \sH^T\Gamma^{-1}\sH \Sigma_\phi - I.
    \end{equation}
    On the other hand, $Q_\infty$ satisfies
    \begin{equation}\label{eq:Riccati;Q_infty;Lyapunov}
        Q_\infty(-\sF-R_\phi)^{T} + (-\sF-R_\phi)Q_\infty = -I.
    \end{equation}
    Thus,
    \begin{align*}
        (\Sigma_\phi-Q_\infty)(-\sF-R_\phi)^{T} + (-\sF-R_\phi)(\Sigma_\phi-Q_\infty) = -\Sigma_\phi \sH^T\Gamma^{-1}\sH \Sigma_\phi \preceq 0.
    \end{align*}
    As $\alpha(-\sF-R_\phi)<0$, by the Lyapunov stability condition we can characterize the difference $(\Sigma_\phi-Q_\infty)$ by
    \begin{equation}\label{eq:Sig_phi-Q_infty;gap}
        \Sigma_\phi-Q_\infty=\int_0^\infty e^{-(\sF+R_\phi)t}\Sigma_\phi \sH^T\Gamma^{-1}\sH \Sigma_\phi e^{-(\sF+R_\phi)^T t}\;dt \succeq 0.
    \end{equation}

    Finally, noting
    \begin{align}\label{eq:minh q_infty minus q_t}
    Q_\infty-Q_t = \int_0^\infty e^{-(F+R_\phi)(t+s)}e^{-(F+R_\phi)^T(t+s)}\;ds = e^{-(F+R_\phi)t} Q_\infty e^{-(F+R_\phi)^T t}.
    \end{align}
    we deduce \eqref{eq:Q_infty-Q_t} by applying \eqref{eq:op-norm;Ceps} with $M=(F+R_\phi)$.
\end{proof}

\begin{remark}\label{rmk:alpha_phi}
    Denoting by $K_\infty = \Sigma_\phi H^T \Gamma^{-1}$ the steady-state Kalman gain, we note that
    \begin{equation}\label{eq:alpha_phi;K_infty}
        \alpha_\phi = -\alpha(-F-R_\phi) = -\alpha(F-K_\infty H).
    \end{equation}
    Indeed, using the matrix Riccati equation \eqref{def:Riccati:P}, one can readily verify
    \[F-K_\infty H=R_\phi^{-1}(-F-R_\phi)^{T}R_\phi,\]
    which implies \eqref{eq:alpha_phi;K_infty}.
\end{remark}

\section{Characterization of the optimal filter in the linear setting}\label{sec:filter_charac}
We now fully restrict our attention to the where the vector fields $b,h$ in the signal-observation model~\eqref{eq:signal-obs} are linear, $b(x)=Fx$ and $h(x)=Hx$ for some $F\in\R^{d\times d}, H\in\R^{k\times d}$. The main goal of this section is to characterize the optimal filter via gauge transformation~\eqref{eq:OF_gauge;summary}, which corresponds to our contribution \textbf{(C2)}: let $q_0\in\R^d$ and $\mu(dx)=e^{-w_0(x-q_0)}\;dx \in\P(\R^d)$, then the solution of the Zakai equation $\rho_t^\mu$ and optimal filter $\pi_t^\mu$~\eqref{def:optimal filter} with initial law $\mu$ are given by
\[
    \rho_t^\mu(x)= \exp\left(-w(t,x-q_t)-x\cdot p_t-r_t\right),\qquad \pi_t^\mu(dx)=\frac{\rho_t^\mu(x)\;dx}{\int_{\R^d} \rho_t^\mu(z)\;dz}, 
\]
where $(q,p,r)_{t\geq 0}$ is a solution of a suitable linear SDE; see Theorem~\ref{thm:gauge} for the precise statement. As the characterization is most naturally provided in terms of the exponent of the solution of the Zakai equation, we will make use of the negative log-Zakai equation \eqref{eq:-logZakai}, which in the linear setting is given by \eqref{eq:-logZakai;linear}
\[
    du = \left(\frac12\Delta u - \frac12|\nabla u|^2-(\sF x)\cdot\nabla u+\frac12|\sH x|_\Gamma^2+\tr\sF\right)\;dt-(\sH x)^T\sG^{-1} dY_t.
\]

To obtain the characterization, we need to establish three facts: first, that the unnormalized filter is a solution of the Zakai equation in the sense of Definition~\ref{def:Zakai;classical}, second, that such a solution is unique, and third, that the gauge transformed solution we construct by analyzing \eqref{eq:-logZakai;linear} is also a solution. Section~\ref{ssec:canonical_filter} establishes the first, and in Section~\ref{ssec:OUdiffusion;unique} we prove the uniqueness of the backward Kolmogorov equation that is critical for the second fact. Finally, Section~\ref{ssec:C2;charac} contains the main result, Theorem~\ref{thm:gauge}, which proves the third fact as well as the required uniqueness for the Zakai equation. All of the results permit the initial law $\mu\in\P(\R^d)$ to be a probability measure, and no additional regularity or moment assumptions are imposed. Moreover, detectability assumption on $(F,H)$ is required only for the uniqueness of the Zakai equation; see Remark~\ref{rmk:gauge}.

\subsection{Canonical filter solution of the Zakai equation}\label{ssec:canonical_filter}
In this section, we establish that the unnormalized filter, which we define below, is a solution of the Zakai equation in the sense of Definition~\ref{def:Zakai;classical}.

For each finite time horizon $[0,T]$, define $\bQ^\mu$ by
\begin{equation}\label{def:Qmu}
    \left.\frac{d\bQ^\mu}{d\bP^\mu}\right|_{\F_T}:=L_T \text{ where } L_t:=\exp\left(-\int_0^t (HX_s)^T \Gamma^{-1}\;dY_s+\frac12\int_0^t |H X_s|_{\Gamma}^2\;ds\right) \text{ for each }t\geq 0;
\end{equation}
more rigorously, the left-hand side at this point should be interpreted as $d\bQ^\mu_T/d\bP^\mu|_{\F_T}$ for some $T$-dependent $\bQ^\mu_T$, but the dependence can be removed 
once consistency ($\bQ^\mu_T=\bQ^\mu_{T'}|_{\F_T}$ for any $T'\geq T$) is established, and we can understand \eqref{def:Qmu} literally.

Then, by Girsanov's theorem, under $\bQ^\mu$ and filtration $(\Y_t)_{t\in[0,T]}$, $(\Gamma^{-1/2}Y_t)_{t\geq 0}$ is a standard Brownian motion, and $(X,Y)$ are independent. The unnormalized filter is then defined by
\begin{equation}\label{def:canonical_filtersol:Zakai}
        \rho_t^\mu(\psi):=\E^{\bQ^\mu}[\psi(X_t)L_t^{-1}|\Y_t] \text{ for } t\in[0,T].
\end{equation}
Due to the independence of $X,Y$ under $\bQ^\mu$, the conditional expectation \eqref{def:canonical_filtersol:Zakai} is much easier to analyze than the original form \eqref{def:optimal filter}, and leads to the derivation of the Zakai equation~\cite{Zakai69optimal} (of which \eqref{def:canonical_filtersol:Zakai} is a solution). Moreover, it is directly related to the optimal filter itself by normalization, which leads to the Kallianpur-Striebel formula~\cite{KalStr68} (see also Zakai~\cite{Zakai69optimal})
\begin{equation}\label{eq:Kallianpur-Striebel}
    \pi_t^\mu(\psi):=\E^{\bP^\mu}[\psi(X_t)|\Y_t] = \frac{\rho_t^\mu(\psi)}{\rho_t^\mu(\one)}= \frac{\E^{\bQ^\mu}[\psi(X_t)L_t^{-1}|\Y_t]}{\E^{\bQ^\mu}[L_t^{-1}|\Y_t]} \text{ for each }\psi\in C_b(\R^d).
\end{equation}
Thus, we refer to \eqref{def:canonical_filtersol:Zakai} as the canonical filter solution of the Zakai equation, to distinguish it from generic solution of the Zakai equation which, a priori, may not be associated to the optimal filter.

While $(L_t)_{t\in[0,T]}$ is automatically a local martingale under $(\F_t)_{t\in[0,T]}$, to justify Girsanov's theorem, it must be shown that it a true $\bP^\mu$-martingale; by nonnegativity, it suffices to show $\E^{\bP^\mu}[L_T]=1$. In a general setting encompassing our case~\eqref{eq:signal-obs}, this can be justified by fixing $(x_0,w^1)\in\R^d\times\Omega^1$ (in the notation of our canonical probability space in Section~\ref{sssec:filtering_eq}) then Girsanov's theorem is justified on each fiber as Novikov's condition trivially follows from continuity of $h$. As $\bP^\mu$~\eqref{def:Pmu} is a product measure, from this follows that $\E^{\bP^\mu}[L_T]=1$, hence that $(L_t)_{t\in[0,T]}$. For a more detailed argument, see Section 1.1.3 of van Handel's thesis, which also motivated our canonical setting \eqref{sssec:SigObs}. Once $(L_t)_{t\in[0,T]}$ is verified to be a true martingale, consistency of $\bQ^\mu$~\eqref{def:Qmu} follows, thus we can unambiguously refer to the unnormalized filter measure \eqref{def:canonical_filtersol:Zakai} on infinite-time horizon $[0,+\infty)$.

To verify that the canonical filter solution of the Zakai equation is a classical solution in the sense of Definition~\ref{def:Zakai;classical}, we rely on the Gaussian mixture representation of the unnormalized filter density $\rho_t$ due to Bene\v{s} and Karatzas~\cite[Theorem 5.1]{BenesKaratzas83}
\begin{equation}\label{eq:Gauss-mixture;BK}
    \rho_t^\mu(x)=\int_{\R^d} k_t(x;z)\Lambda_t(z)\mu(dz)
\end{equation}
where $k_t(x;z)$ is $\Y_t$ a random Gaussian density and $\Lambda_t$ is a likelihood ratio (see \eqref{eq:Lambda_tx} in the proof of Lemma~\ref{lem:Benes-Karatzas}). Once this formula is verified, using the rapid decay of the kernel $k_t$, the necessary regularity for Definition~\ref{def:Zakai;classical} is immediately verified. We will not reproduce the full formula here as it is complicated and not used for the current paper. Instead, we focus on justifying the formula without relying on the uniqueness for the Zakai equation; such a possibility is suggested by the authors in the beginning of their Section 5. The alternative justification is necessary for us due to our use of Definition~\ref{def:Zakai;classical}, especially as we do not require the initial law for $X_0$ to have a density.

As the proof heavily refers to specific parts of the paper of Bene\v{s} and Karatzas without reproducing them here, we delay it to the appendix, and provide only the statement of the result necessary for this paper.
\begin{lemma}\label{lem:Benes-Karatzas}
    Let $\mu\in\P(\R^d)$, and let $\rho_t^\mu$ be the unnormalized filter defined by \eqref{def:canonical_filtersol:Zakai} under $\bQ^\mu$ in \eqref{def:Qmu}.
    Then, $\rho_t^\mu$ is a classical solution of the Zakai equation with initial law $\mu\in\P(\R^d)$ in the sense of Definition~\ref{def:Zakai;classical}.
\end{lemma}

\subsection{Uniqueness of nonnegative solutions of OU diffusion equation}\label{ssec:OUdiffusion;unique}

In this section, we prove a short lemma establishing the uniqueness of nonnegative classical solution to the Kolmogorov backward equation with linear coefficients and vague initial trace. The ratio $\rho_t^\mu/\rho_t^{\mu_\phi}$ translated by a suitable time-dependent random vector $q_t$ satisifies such a deterministic equation, hence the uniqueness for the Zakai equation will be reduced to this lemma. As $(q_t)_{t\geq 0}$ is defined by the gauge transform itself, the proof of uniquenss for the Zakai equation is included as part of Theorem~\ref{thm:gauge} in the next section.

\begin{lemma}\label{lem:OU;unique}
    Let $M\in\R^{d\times d}$, and $\eta\in\M_+(\R^d)$ such that
    \begin{equation}\label{eq:eta;tail}
        \int_{\R^d}e^{-c|x|^2}\;d\eta(x) \text{ for some} c>0.
    \end{equation}
    Then there a unique nonnegative classical solution $f$ to the equation
    \begin{equation}\label{eq:OU;M}
        \partial_t f(t,x) = \frac12\Delta f-(M x)\cdot\nabla f(t,x) \text{ in } (0,\infty)\times\R^d
    \end{equation}
    with initial trace $\eta\in\M_+(\R^d)$ with respect to the vague convergence -- i.e.
    \begin{equation}\label{eq:init_trace;vague}
        \int_{\R^d} \psi(x) f(t,x)\;dx \xrightarrow[]{t\rightarrow 0} \int_{\R^d} \psi(x)\;d\eta(x) \text{ for each } \psi\in C_c(\R^d).
    \end{equation}
\end{lemma}

\begin{proof}
    The classical result due to Aronson~\cite{Aronson71add} (see also \cite[Theorem 12]{Aronson68}) states that nonnegative weak solutions to uniformly elliptic linear parabolic equations with  the initial trace of a nonnegative Borel measure satisfying \eqref{eq:eta;tail} in sense of \eqref{eq:init_trace;vague} is uniquely characterized by the convolution with the fundamental solution, once the coefficients are bounded.
    Thus, it suffices to reduce equation \eqref{eq:OU;M} to such a form.

    To this end, write $\tilde f(t,x)= f(t,e^{Mt}x)$, and once can readily verify that
    \[\partial_t \tilde f = \frac12\tr (e^{-Mt}e^{-M^T t} D^2 \tilde f) \text{ in } (0,\infty)\times\R^d.\]
    Moreover, this transformation preserves the initial trace, as
    \begin{align*}
        \int_{\R^d} \psi(x)\tilde f(t,x)\;dx= \int_{\R^d} \psi(x) f(t,e^{Mt}x)\;dx = e^{-t\tr M}\int_{\R^d} \psi(e^{-Mt}y) f(t,y)\;dy,
    \end{align*}
    and $\tilde\psi_t(x):=e^{-t\tr M}\psi(e^{-Mt}y)$ converges uniformly to $\psi(y)$ as $t\rightarrow 0$.
    Thus, such a solution $\tilde f$ is unique, and the uniqueness of $f$ follows.
\end{proof}

\begin{remark}\label{rmk:OU_unique}
    The transformation $\tilde f(t,x)=f(t,e^{Mt}x)$ is not new, and is a special case of the calculations appearing in Thiffeault~\cite{Thiffeault03} for instance. Indeed, we can replace $x\mapsto e^{Mt}x$ with the flow maps, the same argument can be extended to the case of vector fields with suitably regularity, say $C_b^1(\R^d;\R^d)$.
    
    We have restricted our attention to classical solutions for simplicity, but the same argument should apply to weak solutions in the suitable sense, to which Aronson's theorem applies, as long as the change of variables formula is justified.
\end{remark}

\subsection{Characterization of the optimal filter via gauge transform}\label{ssec:C2;charac}
We now build on the results of the previous sections to characterize the solution of the Zakai equation, hence the optimal filter, via a gauge transform that separates the stochastic, finite-dimensional from the deterministic, infinite-dimensional dynamics.

We first establish a technical lemma useful for verifying the vague initial trace condition (iv) Definition~\ref{def:Zakai;classical}.
\begin{lemma}\label{lem:vague_test}
    Let $(\eta_t)_{t> 0}$ be a family of locally finite Borel measures on $\R^d$ vaguely converging to $\eta_0\in\M_+(\R^d)$ --i.e.
    \begin{equation}\label{eq:nut;vague}
        \int_{\R^d} \psi(x)\;d\eta_t(x)\rightarrow \int_{\R^d}\psi(x)\;d\eta_0(x) \text{ as }t\rightarrow 0 \text{ for every }\psi\in C_c(\R^d).
    \end{equation}
    
    Then for any $\psi\in C_c(\R^d)$, $g\in C(\R^d\times\R^{d_1}\times\R^{d_2})$ and continuous paths $(z_t^i)_{t\geq 0}$ in $\R^{d_i}$ for $i=1,2$,
    \begin{equation}\label{eq:vague_test}
        \int_{\R^d} \psi(x+z_t^1) g(x,z_t^1,z_t^2)\;d\eta_t(x) \rightarrow \int_{\R^d} \psi(x+z_0^1) g(x,z_0^1,z_0^2) \;d\eta_0(x) \text{ as }t\rightarrow 0.
    \end{equation}
\end{lemma}

\begin{proof}
    Let
    \[\tilde\psi_t(x):=\psi(x+z_t^1)g(x,z_t^1,z_t^2).\]
    Then
    \begin{align*}
        &\left|\int_{\R^d} \tilde\psi_t(x)\;d\eta_t(x)-\int_{\R^d} \tilde\psi_0(x)\;d\eta_0(x)\right|
        \leq \left|\int_{\R^d} (\tilde\psi_t(x)-\tilde\psi_0(x))\;d\eta_t(x)\right|+\left|\int_{\R^d} \tilde\psi_0(x)\;(d\eta_t-d\eta_0)(x)\right|.
    \end{align*}
    As $\tilde\psi_0\in C_c(\R^d)$ is a valid test function for vague convergence, the second term on the right-hand side vanishes as $t\rightarrow 0$.

    It remains to show the first term on the right-hand side vanishes. As $\eta_t$ converges vaguely to the locally finite Borel measure $\eta_0$, by the Portmanteau theorem for Radon measures~\cite[Theorem 3.2]{Resnick07heavy},
    \[\limsup_{t\rightarrow 0}\eta_t(K_0)\leq \eta_0(K_0).\]
    Thus, we can fix some $t_0>0$ such that
    \begin{equation}\label{eq:sup-nu_t}
        \sup_{t\leq t_0}\eta_t(K_{0})<+\infty.
    \end{equation}
    
    As $z_t^i$ are continuous, we can fix compact sets $K_i\subset\R^{d_i}$ for $i=1,2$ such that $(z_t^1,z_t^2)\subset K_1\times K_2$ for all $t\leq t_0$. Letting $K_0:=\supp\psi-K_1$, we see that
    \[(x,z^1,z^2)\mapsto \psi(x+z^1)g(x,z^1,z^2)\]
    is uniformly continuous on $K_0\times K_1\times K_2$, with its support in $x$ contained in $K_0$.
    \begin{align*}
        \sup_{x\in\R^d}|\tilde\psi_t(x)-\tilde\psi_0(x)|=\sup_{x\in K_0}|\tilde\psi_t(x)-\tilde\psi_0(x)|\rightarrow 0.
    \end{align*}
    Thus we conclude
    \begin{align*}
        \left|\int_{\R^d} (\tilde\psi_t(x)-\tilde\psi_0(x))\;d\eta_t(x)\right|
        \leq \sup_{x\in K_0}|\tilde\psi_t(x)-\tilde\psi_0(x)|\sup_{t\leq t_0}\eta_t(K_{0})\xrightarrow[]{t\rightarrow 0} 0.
    \end{align*}
\end{proof}

We are now ready to establish the main result of this section. 
\begin{theorem}[Gauge transformation]\label{thm:gauge}
    Suppose $(F,H)$ satisfies Assumption~\ref{ass:detectable}.
    Let $(q_t,p_t,r_t)$ be the solution of the system of SDEs
    \begin{equation}\label{eq:qpbeta}
    \begin{cases}
        dq_t&=(\sF q_t+p_t)\;dt,\\
        dp_t&=(\sH^T\sG^{-1}\sH q_t-\sF^Tp_t)\;dt-\sH^T\sG^{-1}dY_t,\\
        dr_t &= \left((\sF q_t)^Tp_t +\frac12|p_t|^2-\frac12|\sH q_t|_{\sG}^{2}\right)\;dt-p_t\;dq_t
    \end{cases}
    \end{equation}
    with deterministic initial data $p_0=0, r_0=0$ and $q_0\in\R^d$. 
    
    Let $\mu\in\P(\R^d)$. The solution of the Zakai equation in the sense of Definition~\ref{def:Zakai;classical} with initial law $\mu$ is given by
    \begin{equation}\label{eq:rho;gauge}
        \rho_t^\mu(x)=\exp\left(-w(t,x-q_t)-p_t\cdot x-r_t\right)
    \end{equation}
    where $\varrho(t,x):=e^{-w(t,x)}$ is the subprobability solution (in the sense of Proposition~\ref{prop:subprob_uniqueness}) of
    \begin{equation}\label{eq:parabolic;varrho}
        \partial_t\varrho =\frac12\Delta\varrho - \nabla\cdot((Fx)\varrho)-\frac12|Hx|_{\Gamma}^2\varrho \text{ in } (0,+\infty)\times\R^d
    \end{equation}
    with vague initial trace $\tau_{-q_0}\mu\in\P(\R^d)$.

    In particular, the optimal filter $\pi_t^\mu$ is given by
    \begin{equation}\label{eq:OF;gauge}
        \pi_t^\mu(dx)=\frac{\rho_t^\mu(x)\;dx}{\int_{\R^d}\rho_t^\mu(z)\;dz}.
    \end{equation}
\end{theorem}

\begin{proof}
First note that $(q_t,p_t,r_t)_{t\geq 0}$ is well defined, as $(r_t)_{t\geq 0}$ is fully determined by $(q_t,p_t)_{t\geq 0}$, which satisfy a linear system of SDEs \eqref{eq:qpbeta}. Moreover, note that $(q_t)_{t\geq 0},(r_t)_{t\geq 0}$ have finite variation. 

\medskip

\emph{Step 1.} We first verify that \eqref{eq:rho;gauge} is a solution in the sense of Defintition~\ref{def:Zakai;classical}.

As $\varrho$ \eqref{eq:parabolic;varrho} is strictly positive (by the parabolic strong maximum principle) and smooth at every positive time $t>0$, $w=-\log\varrho$ is well-defined and smooth at each $t>0$.
Items (i)-(ii) are immediate from this regularity and that $(q_t,p_t,r_t)_{t\geq 0}$ is $(\Y_t)_{t\geq 0}$-adapted. Moreover, the It\^{o} formula is valid for positive time $t>0$. For Definition~\ref{def:Zakai;classical} (iii), and it suffices to show that
\begin{equation}\label{eq:u;ansatz}
    u(t,x)=w(t,x-q_t)+p_t\cdot x + r_t.
    \end{equation}
satisfies the negative log-Zakai equation~\eqref{eq:-logZakai;linear} as a pointwise continuous semimartingale on each interval $[\eps,T]$ for $0<\eps<T<+\infty$; applying the It\^{o} formula to $\rho_t^\mu=e^{-u(t,x)}$, item (iii) immediately follows.

Now we verify the claim about $u$.
\begin{align*}
    du(t,x)=\partial_t w(t,x-q_t)\;dt-\nabla w(t,x-q_t)\cdot dq_t + x^T dp_t+ dr_t.
\end{align*}
Writing $y=x-q_t$ and noting
\begin{align*}
   \nabla w(t,y)=\nabla u(t,x)-p_t, \quad \Delta w(t,y)=\Delta u(t,x),
\end{align*}
we have
\begin{align*}
    du(t,x)=&\partial_t w(t,y)\;dt-\nabla w(t,y)\cdot dq_t+ x^T dp_t+dr_t\\
    =& \left(\frac12\Delta w(t,y) - \frac12|\nabla w(t,y)|^2-(\sF y)\cdot\nabla w(t,y) + \frac12|\sH y|_\Gamma^2+\tr\sF\right)\;dt\\
    &-\nabla w(t,y)\cdot dq_t+ x^T dp_t+dr_t    \\
    =& \left(\frac12\Delta u(t,x)-\frac12|\nabla u(t,x)-p_t|^2 - (\sF (x-q_t))\cdot (\nabla u(t,x)-p_t) + \frac12|\sH(x-q_t)|_\Gamma^2 + \tr\sF\right)\;dt\\
    &-(\nabla u(t,x)-p_t)\cdot dq_t + x^T\;dp_t +\;dr_t \\
    =& \left(\frac12\Delta u(t,x)-\frac12|\nabla u(t,x)|^2-(\sF x)\cdot \nabla u(t,x)+\frac12|\sH x|_\Gamma^2+\tr\sF\right)\;dt\\
    &+ \nabla u(t,x)\cdot ((p_t+\sF q_t)\;dt - dq_t)
    +(\sF x)^Tp_t\;dt-(\sH x)^T\Gamma^{-1}\sH q_t\;dt+x^T dp_t \\
    &+ \frac12|\sH q_t|_\Gamma^2\;dt
    -\frac12|p_t|^2\;dt
    -(\sF q_t)^T p_t\;dt
    +p_t\cdot dq_t+dr_t.
\end{align*}
By \eqref{eq:qpbeta}, the terms in the last two lines above equal $-(\sH x)^T\Gamma^{-1} dY_t$.  Thus, $u$ is a pointwise continuous semimartingale solution of the equation \eqref{eq:-logZakai;linear}, with quadratic variation
\[d[u(\cdot,x)]_t = |Hx|_\Gamma^2\;dt.\]

Finally, to verify Definition~\ref{def:Zakai;classical} (iv), fix any $\psi\in C_c(\R^d)$ and note
\begin{align*}
    \int_{\R^d} \psi(x)e^{-w(t,x-q_t)-p_t\cdot x-r_t}\;dx = \int_{\R^d} \psi(x+q_t)e^{-p_t\cdot(x+q_t)-r_t} e^{-w(t,x)}\;dx.
\end{align*}
Thus, we can apply Lemma~\ref{lem:vague_test} pathwise to 
\[z_t^1=q_t,\; z_t^2=(p_t,r_t),\; g(x,q_t,(p_t,r_t))=\psi(x+q_t)e^{-p_t\cdot(x+q_t)-r_t}\] 
to conclude.
\medskip

\emph{Step 2.} We now turn to uniqueness.
    Let $\rho(t,x)$ be the solution to the Zakai equation in the sense of Definition~\ref{def:Zakai;classical} with initial trace $\mu\in\P(\R^d)$, driven by the same observations $(Y_t)_{t\geq 0}$ as $(q_t,p_t,r_t)_{t\geq 0}$. Let
    \begin{equation}\label{eq:barrho}
        \bar\rho(t,x)=\exp\left(-\lambda t-\phi(x-q_t)-p_t\cdot x-r_t\right)
    \end{equation}
    where $(\lambda,\phi)$ are as defined in Lemma~\ref{lem:detectability}, and recall that $(t,x)\mapsto \lambda t+\phi(x)$ solves the deterministic Hamilton-Jacobi equation \eqref{eq:hje}, hence $\bar\rho$ by Step 1. solves the Zakai equation. Motivated by the gauge transformation from Step 1., we show that
    \begin{equation}\label{eq:ratio+qt}
        f(t,x):=\frac{\rho(t,x+q_t)}{\bar\rho(t,x+q_t)}
    \end{equation}
    is a classical solution of the equation
    \begin{equation}\label{eq:Fokker-Planck;ratio}
        \partial_t f(t,x) =\frac12\Delta f-(F+R_\phi)x\cdot\nabla f \text{ in } (0,\infty)\times\R^d
    \end{equation}
    with initial trace $f(t,\cdot)\;dx\rightharpoonup e^{\phi(x)}\tau_{-q_0}\mu$ vaguely as $t\rightarrow 0$. 
    
    Suppose the above claim is true. Then, by applying Lemma~\ref{lem:OU;unique} with $M=F+R_\phi$, we deduce that $f$ is a unique solution of \eqref{eq:Fokker-Planck;ratio}; in particular, as the initial datum $e^{\phi(x)}\tau_{-q_0}\mu$ is deterministic, $f$ is also deterministic by uniqueness. As $f$ is simply a division by $\bar\rho$ and translation by $q_t$, this implies that the solution of the Zakai equation is unique; note that we are not invoking Proposition~\ref{prop:subprob_sol} here. Finally, \eqref{eq:OF;gauge} then follows from Lemma~\ref{lem:Benes-Karatzas}, which states that the canonical filter solution is also a solution of the Zakai equation in this sense.

    To verify the claim \eqref{eq:Fokker-Planck;ratio}, note that $\bar\rho>0$ on $\R^d$ and $\rho,\bar\rho,\gamma_{Q_\infty}$ are $C^2(\R^d)$ at any positive time $t>0$. Thus, we can apply the It\^{o} formula to $g_t:=\frac{\rho_t}{\bar\rho_t}$ to obtain
    \begin{align*}
        dg_t=\frac{d\rho_t}{\bar\rho_t}-\frac{\rho_t d\bar\rho_t}{\bar\rho_t^2}-\frac{d[\rho,\bar\rho]_t}{\bar\rho_t^2}+\frac{\rho_t d[\bar\rho]_t}{\bar\rho_t^3}.
    \end{align*}
    As $d[Y,Y]_t=\Gamma\;dt$, we have
    \begin{align*}
        d[\rho,\bar\rho]_t=\rho_t\bar\rho_t d[(Hx)^T\Gamma^{-1} Y, (Hx)^T\Gamma^{-1} Y]_t = \rho_t\bar\rho_t (Hx)^T\Gamma^{-1} \Gamma (Hx)\;dt = \rho_t\bar\rho_t|Hx|_{\Gamma}^2\;dt.
    \end{align*}
    Similarly,
    \[d[\bar\rho]_t = \bar\rho_t^2|Hx|_{\Gamma}^2\;dt.\]
    Thus, we have the exact cancellation
    \begin{align*}
        -\frac{d[\rho,\bar\rho]_t}{\bar\rho_t^2}+\frac{\rho_t d[\bar\rho]_t}{\bar\rho_t^3} = -\frac{\rho_t }{\bar\rho_t}|Hx|_{\Gamma}^2\;dt+\frac{\rho_t}{\bar\rho_t}|Hx|_{\Gamma}^2\;dt=0.
    \end{align*}
    Moreover, 
    \begin{align*}
        \frac{d\rho_t}{\bar\rho_t}-\frac{\rho_t d\bar\rho_t}{\bar\rho_t^2}=&\frac{(\frac12\Delta\rho_t-\nabla\cdot(\rho_t Fx) )\;dt+\rho_t (Hx)^T\Gamma^{-1}dY_t }{\bar\rho_t}
        \\&-\frac{\rho_t}{\bar\rho_t}\frac{(\frac12\Delta\bar\rho_t-\nabla\cdot(\bar\rho_t Fx))\;dt+\bar\rho_t (Hx)^T\Gamma^{-1}dY_t }{\bar\rho_t}
        \\=&\frac{\frac12\Delta\rho_t-\nabla\rho_t\cdot(Fx)-g_t\frac12\Delta\bar\rho_t+ g_t\nabla \bar\rho_t\cdot Fx}{\bar\rho_t}\;dt
        \\=& \frac12\frac{\Delta\rho_t-g_t\Delta\bar\rho_t}{\bar\rho_t}\;dt-\nabla g_t\cdot (Fx)\;dt.
    \end{align*}
    As
    \begin{align*}
    \Delta g_t=\nabla\cdot\left(\frac{\nabla\rho_t}{\bar\rho_t}-\frac{\rho_t\nabla\bar\rho_t}{\bar\rho_t^2}\right)=&\frac{\Delta\rho_t-g_t\Delta\bar\rho_t}{\bar\rho_t}-\frac{2\nabla\rho_t\cdot\nabla\bar\rho_t}{\bar\rho_t^2}+\frac{2\rho_t|\nabla\bar\rho_t|^2}{\bar\rho_t^3}
    \\=& \frac{\Delta\rho_t-g_t\Delta\bar\rho_t}{\bar\rho_t}-2\nabla g_t\cdot\nabla\log\bar\rho_t,
    \end{align*}
    we have
    \begin{align*}
        dg_t=\frac{d\rho_t}{\bar\rho_t}-\frac{\rho_t d\bar\rho_t}{\bar\rho_t^2}
        =\left(\frac12\Delta g_t + (\nabla\log\bar\rho_t-(Fx))\cdot\nabla g_t\right)\;dt.
    \end{align*}
    By construction of $\bar\rho_t$ from Theorem~\ref{thm:gauge},
    \[\nabla\log\bar\rho_t(x)=-\nabla\phi(x-q_t)-p_t.\]
    Thus, defining $f_t(x)=g_t(x+q_t)$ and recalling that $[q,q]_t=0$, we see that $f_t$ satisfies the deterministic equation~\eqref{eq:Fokker-Planck;ratio}:
    \begin{align*}
        df_t(x)=&dg_t(x+q_t)+dq_t\cdot\nabla g_t(x+q_t)
        \\=& \left(\frac12\Delta f_t(x)-(F(x+q_t)+p_t+\nabla\phi(x))\cdot \nabla f_t(x) + (Fq_t+p_t)\cdot\nabla f_t(x)\right)\;dt
        \\=&\left(\frac12\Delta f_t(x)-(Fx+\nabla \phi(x))\cdot\nabla f_t(x)\right) \;dt.
    \end{align*}

    Now it remains to verify that $f$ satisfies the initial condition in the vague trace sense. To this end, fix $\psi\in C_c(\R^d)$. Note
    \begin{align*}
        \int_{\R^d} \psi(x) f(t,x)\;dx = \int_{\R^d} \psi(x)\frac{\rho(t,x+q_t)}{\bar\rho(t,x+q_t)}\;dx = \int_{\R^d} \psi(x-q_t)e^{\lambda t+\phi(x-q_t)+p_t\cdot x+r_t}\rho_t(x)\;dx.
    \end{align*}
    Thus, by applying Lemma~\ref{lem:vague_test} with $\eta_t=\rho_t(x)\;dx$, $z_t^1=-q_t$, $z_t^2=(q_t,p_t,r_t)$ to conclude.
\end{proof}

We conclude this section with a couple remarks on Theorem~\ref{thm:gauge}.
\begin{remark}\label{rmk:gauge}
    Note that the detectability condition Assumption~\ref{ass:detectable} is used only in \emph{Step 2.} to establish uniqueness.
    In particular, that $\rho_t$ defined in \eqref{eq:rho;gauge} solves the Zakai equation holds for general $F,H$.

    Moreover, the choice of the initial data $p_0=0$ and $r_0=0$ is not necessary, as long as both $e^{-w_0(x-q_0)-x\cdot p_t-r_t}$, $e^{-w_0(x)}$ are integrable in $x$. 
\end{remark}

\section{Asymptotic behavior of the optimal filter}\label{sec:asymp}
Finally, we analyze the asymptotic behavior of the optimal filter under the the assumption that $(F,H)$ satisfies the detectability condition, Assumption~\ref{ass:detectable}. Section~\ref{ssec:C3;ZakaiEnt} provides a nonasymptotic exponential decay of the weighted entropy of the ratio of unnormalized filters. From the simple proof, we will see that the characterization by the gauge transformation in Theorem~\ref{thm:gauge} interacts particularly well with Doob's $h$-transform. However, this statement should not be understood as a filter stability (see Remark~\ref{rmk:entropydissipation}), for which a more involved analysis is required. To this end, Section~\ref{ssec:asymp;SDE} first characterizes the asymptotic behavior of the quantities related $(q,p,r)$ from \eqref{eq:qpbeta}. Building on this, Section~\ref{ssec:C3;filterstability} establishes the main results of this paper corresponding to our contribution \textbf{(C3)}, Theorem~\ref{thm:OF_stability}, which states asymptotic exponential stability of the optimal filter with respect to $\chi^p$-divergences (with respect to a reference optimal filter) for any finite $p\geq 1$, and Corollary~\ref{cor:TV,W2->0}, which establish filter stability in total variation and the $2$-Wasserstein distance. In Remark~\ref{rmk:OF_stability;comparison}, we compare our results with previously known filter stability results in similar settings.

We leave a few remarks on the assumptions we impose on the initial law $\mu$. As the main quantitative statements are formulated in terms of weighted entropy and $\chi^p$-divergence (filter stability in total variation and $W_2$ is a consequence of the latter), absolute continuity of $\mu$ is necessary for the statements to be meaningful. However, our careful analysis tries to avoid unnecessary assumptions on $\mu$, such as finite second moments; in fact, some of our proofs (that of Lemma~\ref{lem:Et}, for instance) can be shortened if such an assumption is imposed. While we require an integrability assumption involving the initial law in  Theorem~\ref{thm:OF_stability}, we believe such an assumption is in fact necessary in some form; see Remark~\ref{rmk:observability}. 

Before proceeding, we recall some objects defined in Lemma~\ref{lem:detectability} in the linear detectable setting which play a central role in this section: let $R_\phi\in\R_{++}^{d\times d}$ be the solution of the matrix Riccati equation \eqref{def:Riccati:P}, and
\begin{equation}\label{def:matrices}
    \Sigma_\phi=R_\phi^{-1}, \phi(x)=\frac12 x^T \Sigma_\phi^{-1} x,\;\mu_\phi=\gamma_{\Sigma_\phi},\;\alpha_\phi=-\alpha(-F-R_\phi)<0.
\end{equation}
Let $Q_t$ be the covariance of $X_t^\phi$ from \eqref{eq:dXphi;linear}, and $\m_\phi=\gamma_{Q_\infty}$, recalling from \eqref{eq:Q_inf<Sigma_phi} that $\Sigma_\phi\succeq Q_\infty$.

\subsection{Zakai equation and weighted entropy decay}\label{ssec:C3;ZakaiEnt}
In this section, we combine the characterization of unnormalized filter, Theorem~\ref{thm:gauge}, with the entropy dissipation formula, Lemma~\ref{lem:ent_dissipation}, to deduce a clean weighted entropy decay of the ratio of unnormalized filter to a reference unnormalized filter, initiated at a Gaussian with covariance $\Sigma_\phi$. 

\begin{corollary}\label{cor:entropydissipation}
    Let $(F,H)$ satisfy Assumption~\ref{ass:detectable}. Let $q_0\in\R^d$, and let
    \[\mu_\phi=\tau_{q_0}\gamma_{\Sigma_\phi},\; \mu=e^{-w_0(x-q_0)}\;dx\in\P(\R^d),\] 
    and suppose they $\Ent_{\tau_{q_0}\gamma_{Q_\infty}}(d\mu/d\mu_\phi)<+\infty$.
    
    Let $(q_t)_{t\geq 0}$ be the solution of the SDE \eqref{eq:qpbeta} with initial data $q_0$.
    Then, denoting by $\kappa_\phi$ the log-Sobolev constant of $\m_\phi=\gamma_{Q_\infty}$,
   \begin{equation}\label{eq:entropydissipation}
       \Ent_{\tau_{q_t}\gamma_{Q_\infty}}\left(\frac{d\rho_t^\mu}{d\rho_t^{\mu_\phi}}\right)\leq e^{-\kappa_\phi t}\Ent_{\tau_{q_0}\gamma_{Q_\infty}}\left(\frac{d\mu}{d\mu_\phi}\right) \text{ for all }t\geq 0 \text{ almost surely,}
   \end{equation}
\end{corollary}

\begin{proof}
    By \eqref{eq:OF;gauge} from Theorem~\ref{thm:gauge} , $\pi_t^\nu=\frac{\rho_t^\nu}{\rho_t^\nu(\one)}$ for $\nu=\mu,\mu_\phi$, where
    \begin{align*}
        \rho_t^\mu(x)=\exp\left(-w(t,x-q_t)-p_t\cdot x-r_t\right), \rho_t^{\mu_\phi}(x)=\exp\left(-\lambda t-\phi(x-q_t)-p_t\cdot x-r_t\right)
    \end{align*}
    and $(q_t,p_t,r_t)_{t\geq 0}$ are given by the SDE \eqref{eq:qpbeta}.
    Thus,
    \begin{align*}
        \frac{d\rho_t^\mu}{d\rho_t^{\mu_\phi}}(x)=e^{-w(t,x-q_t)+\lambda t+\phi(x-q_t)}=f_t(x-q_t),
    \end{align*}
    where, by Proposition~\ref{prop:subprob_sol} (and that $\varrho=e^{-w(t,x)}$ is a subprobability solution of \eqref{eq:parabolic;varrho}), $f_t$ is a semigroup solution of \eqref{eq:Fokker-Planck;ratio} with initial data $f_0(x)=e^{-w_0(x)+\phi(x)}$. Thus, applying Lemma~\ref{lem:ent_dissipation}, we deduce
    \begin{align*}
        \Ent_{\tau_{q_t}\gamma_{Q_\infty}}\left(\frac{d\rho_t^\mu}{d\rho_t^{\mu_\phi}}\right)=\Ent_{\gamma_{Q_\infty}}(f_t)
        \leq e^{-\kappa_\phi t}\Ent_{\gamma_{Q_\infty}}(f_0)= e^{-\kappa_\phi t}\Ent_{\tau_{q_0}\gamma_{Q_\infty}}\left(\frac{d\mu}{d\mu_\phi}\right).
    \end{align*}
\end{proof}

\begin{remark}\label{rmk:entropydissipation}
    The clean decay of the weighted entropy \eqref{eq:entropydissipation} is aesthetically pleasing, but is somewhat far from a meaningful filter stability result for two reasons. Firstly, the inequality is formulated in terms of the unnormalized filter densities, where as the main object of interest in filtering is the optimal filter. The appropriate normalization constant ($Z_t$ defined in \eqref{def:Zt}) needed to convert \eqref{eq:entropydissipation} into a statement about optimal filter is in fact random, and additional hard analysis is required to characterize its behavior.

    Moreover, we will see in the following section that the accurate estimator of the true signal $X_t$ is $q_t-\Sigma_\phi p_t$, rather than $q_t$; in fact, as $t\rightarrow\infty$, $q_t$ may diverge exponentially far from the true signal, in which case the reference measure $\tau_{q_t}\gamma_{Q_\infty}$ will not be able to witness how well the ratio $\rho_t^\mu/\rho_t^{\mu_\phi}$ is behaved in the most relevant regions.

    Thus, Corollary~\ref{cor:entropydissipation} is best seen as capturing a certain geometric property of the solutions of the Zakai equations.
\end{remark}

\subsection{Asymptotic behavior of the gauge SDE}\label{ssec:asymp;SDE}
As we have discussed in the introduction, Section~\ref{sssec:summar;large-time},
\[\frac{d\pi_t^\mu}{d\pi_t^{\mu_\phi}}(x)=Z_t^{-1}\frac{\rho_t^\mu(x)}{\rho_t^{\mu_\phi}(x)} \text{ where } Z_t=\int_{\R^d} f_t(x-\Sigma_\phi p_t)\;d\gamma_{\Sigma_\phi},\]
where $f_t$ is the semigroup solution of \eqref{eq:Fokker-Planck;ratio}. Thus, if one wishes to translate Corollary~\ref{cor:entropydissipation} into a meaningful statement about the optimal filter, one must show that $Z_t$ converges as $t\rightarrow\infty$. 

In general $p_t$ from \eqref{eq:qpbeta} may grow exponentially as $t\rightarrow\infty$, whereas $f_t$ exponentially stabilizes to a constant. Thus, in order to understand the asymptotic behavior of $Z_t$, a finer quantitative analysis on these two competing effects is necessary. We will see in Section~\ref{ssec:C3;filterstability} understanding this precisely is the key to establishing quantitative filter stability.

We first record the following classical fact on the asymptotic behavior of solutions of linear ODEs. Let $M\in\R^{d\times d}$ for any $\kappa<-\alpha(-M)$, there exists $C(\kappa,\alpha(-M))>0$ such that,
\begin{equation}\label{eq:op-norm;Ceps}
        \|e^{-Mt}\|\leq C e^{-\kappa t} \text{ for all } t\geq 0,
\end{equation}
where for our purposes $\|\cdot\|$ is the operator norm with respect to the Euclidean distance; see for instance, \cite[Corollary 3.6]{Teschll12ODE}.

In light of Theorem~\ref{thm:gauge}, the optimal filter $\pi_t^{\mu_\phi}$ with initial datum $\mu_\phi=\tau_{q_0}\gamma_{\Sigma_\phi}$ is characterized by
\[\pi_t^{\mu_\phi}=N(q_t-\Sigma_\phi p_t,\Sigma_\phi).\]
Thus, if $(X_t)_{t\geq 0}$ is the true signal producing the observations, we expect the error process
\begin{equation}\label{def:Et}
    E_t:=q_t-\Sigma_{\phi}p_t-X_t
\end{equation}
to stabilize as $t\rightarrow\infty$. We verify this fact as well as a few related facts about the asymptotic behavior of the solution $(q_t,p_t)_{t\geq 0}$ to the SDE \eqref{eq:qpbeta}.

\begin{lemma}\label{lem:Et}
    The error process $E_t$ defined in \eqref{def:Et} satisfies the ergodic Ornstein-Uhlenbeck SDE
    \begin{equation}\label{eq:et;OU}
        dE_t = (F-\Sigma_\phi H^T\Gamma^{-1} H)E_t\;dt +\Sigma_\phi H^T\Gamma^{-1/2}\;dW_t^2-dW_t^1.
    \end{equation}
    The rescaled shift term 
    \begin{equation}\label{def:rt}
        S_t:=-e^{-(F+R_\phi)t}\Sigma_{\phi}p_t,
    \end{equation}
    which is $\Y_t$-measurable, satisfies
    \begin{equation}\label{eq:rt}
        S_t=S_0-\int_0^t e^{-(F+R_\phi)s}\Sigma_\phi H^T\Gamma^{-1}HE_s\;ds + \int_0^t e^{-(F+R_\phi)s}\Sigma_\phi \sH\sG^{-1/2}\;dW_s^2.
    \end{equation}
    In particular, letting $\alpha_\phi:=-\alpha(-F-R_\phi)>0$, there exists a $\Y$-measurable, almost surely finite random vector $S_\infty$ such that, for every $\eps\in (0,\alpha_\phi)$, there exists some almost surely finite constant $C=C(\eps,\omega)>0$ satisfying
    \begin{equation}\label{eq:r_infty-S_t}
        |S_t-S_\infty|\leq C(\eps,\omega) e^{-(\alpha_\phi-\eps)t} \text{ for all } t\geq 0.
    \end{equation}
\end{lemma}

\begin{remark}\label{rmk:lem:Et}
   Although Lemma~\ref{lem:Et} involves the signal process $X=(X_t)_{t\geq 0}$ hence depends on the initial data $X_0$, the probabilistic statements only concern $r=(S_t)_{t\geq 0}$ and its limit $S_\infty$ are $\Y$-measurable. Moreover, the proof will clarify that the almost sure statements are independent of the choice of the initial law for $X_0$. 
\end{remark}

\begin{proof}
We split the proof into three steps. In \emph{Step 1.}, we verify \eqref{eq:et;OU} and \eqref{eq:rt}. Then \emph{Step 2.} verifies almost sure asymptotic exponential decay of $|E_t|$, independent of the choice of initial law for $X_0$. Finally, in \emph{Step 3. } we use this to deduce \eqref{eq:r_infty-S_t}.
\medskip

\emph{Step 1.} We first verify the equations \eqref{eq:et;OU} and \eqref{eq:rt}. The former follows from straightforward calculations. Indeed, 
    \begin{align*}
        dE_t =& (dq_t-\Sigma_\phi dp_t-dX_t)\\
        =&(F-\Sigma_\phi H^T\Gamma^{-1} H)(q_t-X_t)\;dt + (R_\phi+\Sigma_\phi F^T R_\phi)\Sigma_\phi p_t\;dt 
        + \Sigma_\phi H^T\Gamma^{-1/2} dW_t^2 - dW_t^1.
    \end{align*}
    By \eqref{def:Riccati:P},
    \[R_\phi+\Sigma_\phi F^T R_\phi = \Sigma_\phi (R_\phi^2+ F^T R_\phi)=\Sigma_\phi(H^T\Gamma^{-1} H- R_\phi F)=-(F-\Sigma_\phi H^T\Gamma^{-1} H),\]
    hence we deduce \eqref{eq:et;OU}. To verify the ergodicity of $(E_t)_t$, note
    \begin{align*}
        F-\Sigma_\phi H^T\Gamma^{-1}H=R_\phi^{-1}(R_\phi F-H^T\Gamma^{-1}H)=-R_\phi^{-1}(R_\phi^2+F^T R_\phi)=R_\phi^{-1}(-F-R_\phi)^T R_\phi.
    \end{align*}
    As $\alpha(-F-R_\phi)<0$, we deduce that $\alpha(F-\Sigma_\phi H^T\Gamma^{-1}H)<0$.

    Note
    \begin{align*}
       d(q_t-X_t)=F(q_t-X_t)\;dt + p_t\;dt - dW_t^1.
    \end{align*}
    Thus, recalling $\Sigma_\phi=R_\phi^{-1}$, we have
    \begin{align*}
        d\Sigma_\phi p_t=d(q_t-X_t)-dE_t = \Sigma_\phi H^T\Gamma^{-1}H (q_t-X_t)\;dt - \Sigma_\phi F^T R_\phi(\Sigma_\phi p_t)\;dt - \Sigma_\phi H^T\Gamma^{-1/2}dW_t^2.
    \end{align*}
    From the Riccati equation \eqref{def:Riccati:P},
    \[-\Sigma_\phi F^T R_\phi = \Sigma_\phi (R_\phi F+ R_\phi^2 - H^T\Gamma^{-1}H)=F+R_\phi-\Sigma_\phi H^T\Gamma^{-1}H.\]
    Thus,
    \begin{align*}
        d \Sigma_\phi p_t = (F+R_\phi)\Sigma_\phi p_t\;dt + \Sigma_\phi H^T\Gamma^{-1}H E_t\;dt  - \Sigma_\phi H^T\Gamma^{-1/2}dW_t^2.
    \end{align*}
    Hence $S_t=-e^{-(F+R_\phi)t}\Sigma_\phi p_t$ satisfies
    \begin{align*}
        dS_t=-e^{-(F+R_\phi)t}\left(\Sigma_\phi H^T\Gamma^{-1}H E_t\;dt  - \Sigma_\phi H^T\Gamma^{-1/2}dW_t^2\right),
    \end{align*}
    which is precisely \eqref{eq:rt}.
\medskip

\emph{Step 2.}
    To separate the issue of dependence on the initial law of $X_0$, let
    \[A=F-\Sigma_\phi H^T\Gamma^{-1}H,\qquad M_t=\Sigma_\phi H^T\Gamma^{-1/2} W_t^2- W_t^1,\]
    noting that $M=(M_t)_{t\geq 0}$ is a square-integrable martingale under any $\bP^\mu$, as $M$ is independent of $\mu$. Define $U=(U_t)_{t\geq 0}$ by
    \[U_t:=\int_0^t e^{A(t-s)}dM_s,\]
    so that we have the decomposition
    \begin{equation}\label{eq:Et;decomposition}
    E_t=e^{At}E_0+U_t.
    \end{equation}
    In particular, $Z$ is independent of $X_0$. 
    
    We claim that, for each $\eps>0$, there exists an event $\Omega_\eps\subset\F$ such that $\bP^\mu(\Omega_\eps)=1$ for all $\mu\in\P(\R^d)$, and
    \begin{equation}\label{eq:Et;bd}
        |E_t|\leq C(\eps,\omega)e^{\eps t} \text{ for all } t\geq 0,\;\omega\in\Omega_\eps.
    \end{equation}
    As $E_t=e^{At}E_0+U_t$ with $\alpha(A)<0$, it suffices to show the analogous claim for $U_t$
    \begin{equation}\label{eq:Zt}
        |U_t|\leq C(\eps,\omega) e^{\eps t} \text{ for all }t\geq 0,\;\omega\in\Omega_\eps.
    \end{equation}
    
    To this end, write for each $t\in[n,n+1]$
    \[U_t=e^{A(t-n)}\left(Z_{n}+\int_{n}^{t}e^{-A(s-n)}dM_s\right).\]
    Then, setting $c=\sup_{r\in[0,1]}\|e^{Ar}\|$ and applying Doob's maximal inequality,
    \begin{align*}
        \sup_{n}\E^\mu \sup_{t\in[n,n+1]}|U_t|^2 &\leq 2c^2\sup_{n} \E^\mu\sup_{t\in[n,n+1]}\left(|U_n|^2+\left|\int_n^{t} e^{-A(s-n)}dM_s\right|^2\right)
        \\ &\leq  2c^2\sup_{n}\left(\E|U_n|^2+4 \E^\mu\left|\int_n^{n+1} e^{-A(s-n)}dM_s\right|^2\right).
    \end{align*}
    Note that the right-hand is finite. Indeed, as $\alpha(A)<0$, It\^{o} isometry yields $\sup_{t\geq 0}\E|U_t|^2<+\infty$, and the integral term is also easily controlled by It\^{o} isometry. Thus, from Markov's inequality
    \[        \sum_{n=1}^\infty \bP^\mu \left( \sup_{t\in[n,n+1]}|U_t| > e^{\eps n}\right)<+\infty,\]
    and by the Borel-Cantelli lemma, we can fix a set $\Omega_\eps\subset\F$ such that \eqref{eq:Zt} holds. As $U=(U_t)_{t\geq 0}$ only depends on $W^1,W^2$, we can choose $\Omega_\eps$ to be independent of $\mu$.
\medskip

\emph{Step 3.}
    As $\alpha(-F-R_\phi)<0$, choosing $\eps<\alpha_\phi$ in \eqref{eq:Et;bd}, we see that
    \[S_\infty = S_0-\int_0^\infty e^{-(F+R_\phi)s}\Sigma_\phi H^T\Gamma^{-1}HE_s\;ds + \int_0^\infty e^{-(F+R_\phi)s}\Sigma_\phi \sH\sG^{-1/2}\;dW_s^2\]
    is a well-defined random vector that is $\bP^\mu$-a.s. finite for every $\mu\in\P(\R^d)$.
    
    To deduce \eqref{eq:r_infty-S_t}, note that
    \begin{align*}
        S_\infty-S_t = -e^{-(F+R_\phi)t}\int_0^{\infty} e^{-(F+R_\phi)s}\Sigma_\phi H^T\Gamma^{-1} HE_{t+s}\;ds+ \int_t^\infty e^{-(F+R_\phi)s}\Sigma_\phi \sH\sG^{-1/2}\;dW_{s}^2.
    \end{align*}
    Now fix any $\eps<\alpha_\phi$. By \eqref{eq:Et;bd} applied to $\eps/2>0$,
    \[\left|\int_0^{\infty} e^{-(F+R_\phi)s}\Sigma_\phi H^T\Gamma^{-1} H^T E_{t+s}\;ds\right| \leq C(\eps,\omega) e^{\eps t/2},\]
    up to modifications of the constant. Combining with \eqref{eq:op-norm;Ceps} applied to $\kappa'=\alpha_\phi-\eps/2$, we have the desired $O(e^{-(\alpha_\phi-\eps)t})$ decay rate of the deterministic integral.
    
    On the other hand, for the stochastic integral, applying the analogous argument as in \emph{Step 2.} via Doob's maximal inequality, It\^{o} isometry, and Borel-Cantelli lemma, $|\int_t^\infty e^{-(F+R_\phi)s}\Sigma_\phi \sH^T\sG^{-1/2}\;dW_{s}^2|$ can be controlled by $C'(\eps,\omega) e^{-(\alpha_\phi-\eps) t}$ on an event of full-measure for all $\bP^\mu$, and we conclude \eqref{eq:r_infty-S_t}. Moreover, since $S_t$ is $\mathcal Y_t$-measurable and $S_t\to S_\infty$ almost surely, we may choose a version of $S_\infty$ that is $\mathcal Y$-measurable.

\end{proof}

\subsection{Asymptotic exponential stability of optimal filter}\label{ssec:C3;filterstability}
We are now equipped to analyze quantitative stability of optimal filter. In the proofs of this section, we will use the following notation: given $a,b\in\R^d$, $[a,b]$ refers to the line segment between $a$ and $b$.

We begin by establishing a useful lemma.
\begin{lemma}\label{lem:cov_err}
    Suppose $(F,H)$ satisfies Assumption~\ref{ass:detectable}. Let $Q_t,Q_\infty$ be as defined in Lemma~\ref{lem:detectability}, and let $\eta\in \M_+(\R^d)$ be a nonnegative Radon measure and $\beta>1$ such that $(\gamma_{\beta Q_\infty}\eta)(\R^d)<+\infty$. 
    Then, for all $t>0$ sufficiently large such that $\frac12Q_\infty \preceq Q_t \preceq Q_\infty$, we have
    \begin{equation}\label{eq:cov_err_1}
    |\eta\ast \gamma_{Q_t}(a)-\eta\ast \gamma_{Q_\infty}(a)| \leq  C(\beta,d,Q_\infty) \|Q_t-Q_\infty\| (\gamma_{\beta Q_\infty}\eta)(\R^d) (1+|a|^2)\exp\left(\frac{a^T Q_\infty^{-1}a}{2(\beta-1)}\right).
    \end{equation} and
    \begin{equation}\label{eq:cov_err_2}
    |\eta\ast \gamma_{Q_t}(a)-\eta\ast \gamma_{Q_t}(b)| C(\beta,d,Q_\infty)\leq |a-b|(\gamma_{\beta Q_\infty}\eta)(\R^d) \sup_{c\in[a,b]}(1+|c|^2)\exp\left(\frac{c^T Q_\infty^{-1}c}{2(\beta-1)}\right)
    \end{equation}
\end{lemma}
\begin{proof}
    We have that
    \[
    \left|\eta\ast \gamma_{Q_t}(a)-\eta\ast \gamma_{Q_\infty}(a)\right|\leq \int_{\R^d}\left|\gamma_{Q_t}(a-z)-\gamma_{Q_\infty}(a-z)\right|\ d\eta(z).
    \] 
    Observe the following pointwise bound: Let $D_t = Q_t -Q_\infty\preceq 0$ and write $\tilde Q_r:=Q_\infty+r D_t\leq Q_\infty$ be a reparametrization with $r\in[0,1]$. Then
     \begin{align*}
        |\gamma_{Q_t}(a-z)-\gamma_{Q_\infty}(a-z)|=&\left|\int_0^1 \frac{d}{dr}\gamma_{\tilde Q_r}(a-z)\;dr\right| \\
        \leq&\int_0^1 \left|\gamma_{\tilde Q_r}(a-z)[(a-z)^T\tilde Q_r^{-1}D_t\tilde Q_r^{-1}(a-z)-\tr(\tilde Q_r^{-1}D_t)]\right|\;dr \\
        \leq& C \gamma_{Q_\infty}(a-z)\left(\|Q_\infty^{-1}\|^2 \|D_t\||a-z|^2+d\|Q_\infty^{-1}\|\|D_t\|\right);
    \end{align*}
    Then by the previous estimates and the fact that $f$ is non-negative
    \[
    \int_{\R^d}|\gamma_{Q_t}(a-z)-\gamma_{Q_\infty}(a-z)|\ d\eta(z) \leq  C \|Q_\infty^{-1}\| \|D_t\|\int_{\R^d}  (1+|a-z|^2)\gamma_{Q_\infty}(a-z)\;d\eta(z). 
    \]
     Now observe that
 \begin{align}\label{eq:complete sq lem}
     & \int_{\R^d} (1+|a-z|^2)\frac{\gamma_{ Q_\infty}(a-z)}{\gamma_{\beta Q_\infty}(z)}\gamma_{\beta  Q_\infty}(z)\;d\eta(z)\notag\\
     &= \beta^{d/2}\int_{\R^d}(1+|a-z|^2) \exp\left(-\frac12 (a-z)^T Q^{-1}_\infty (a-z)+\frac{1}{2\beta}z^T Q_\infty^{-1}z\right)\gamma_{\beta Q_\infty}(z)\;d\eta(z). 
 \end{align} Then noting that $|a-z|^2=\left|\frac{\beta a}{\beta-1}-\frac{a}{\beta-1}-z\right|^2 \leq \frac{2|a|^2}{(\beta-1)^2}+2\left|z-\frac{\beta a}{\beta-1}\right|$ and applying H\"{o}lder's inequality, we see that the integral on the right-hand side of \eqref{eq:complete sq lem} controlled by
 \begin{align*}
     &(\gamma_{\beta Q_\infty}\eta)(\R^d) \sup_{z\in \R^n}\left(1+\frac{2|a|^2}{(\beta-1)^2}+2\left|z-\frac{\beta a }{\beta-1}\right|^2\right)\exp\left(-\frac12 (a-z)^T Q^{-1}_\infty (a-z)+\frac{1}{2\beta}z^T Q_\infty^{-1}z\right) \notag\\
     &\leq \exp\left(\frac{a^T Q^{-1}_\infty a}{2(\beta-1)}\right)(\gamma_{\beta Q_\infty}\eta)(\R^d) \notag\\ 
     &\quad\times\sup_{z} \left(1+\frac{2|a|^2}{(\beta-1)^2}+2\left|z-\frac{\beta a }{\beta-1}\right|^2\right)\exp\left(-\frac{\beta-1}{2\beta}\left(z-\frac{\beta}{\beta-1}a\right)^TQ^{-1}_\infty \left(z-\frac{\beta}{\beta-1}a\right)\right)\notag\\
     &\leq \exp\left(\frac{a^T Q^{-1}_\infty a}{2(\beta-1)}\right)(\gamma_{\beta Q_\infty}\eta)(\R^d) \left(1+\frac{2|a|^2}{(\beta-1)^2} +\frac{4\beta}{e\lambda_{\min}(Q^{-1}_\infty)(\beta-1)}\right)
 \end{align*} 
 where in the last line $\lambda_{\min}(Q_\infty^{-1})$ is the smallest eigenvalue of $Q_\infty^{-1}$, and we used the fact that $r^2e^{-kr^2}\leq \frac{1}{ek}$ for $k>0$. We can then pick a constant $C(\beta,d,Q_\infty)$ such that 
 \begin{equation}\label{eq:complete sq}
     \int_{\R^d}(1+|a-z|^2)\gamma_{Q_\infty}(a-z)d\eta(z)\leq C(\beta,d,Q_\infty) (\gamma_{\beta Q_\infty}\eta)(\R^d)(1+|a|^2)\exp\left(\frac{a^T Q_\infty^{-1}a}{2(\beta-1)}\right).
 \end{equation} Then it follows that
 \[
 |\eta\ast \gamma_{Q_t}(a)-\eta\ast \gamma_{Q_\infty}(a)|\leq C(\beta,d,Q_\infty) \|Q_t-Q_\infty\| (\gamma_{\beta Q_\infty}\eta)(\R^d) (1+|a|^2)\exp\left(\frac{a^T Q_\infty^{-1}a}{2(\beta-1)}\right).
 \]
 This concludes the first claim. For the second claim we proceed similarly: 
    \[
    |\eta\ast \gamma_{Q_t}(a)-\eta\ast\gamma_{Q_t}(b)|\leq \int_{\R^d}|f(z)||\gamma_{Q_t}(a-z)-\gamma_{Q_t}(b-z)|dz.
    \] It follows that
    \begin{align*}
    |\gamma_{Q_t}(a-z)-\gamma_{Q_t}(b-z)| &\leq \left|\int_0^1 \frac{d}{dr}\gamma_{Q_t}(r(a-z)+(1-r)(b-z))dr\right|\\
    &\leq|a-b| \|Q_t^{-1}\| \int_0^1 \left|(b+r(a-b)-z)\gamma_{Q_t}(b+r(a-b)-z)\right|dr.
    \end{align*} Using the fact that $Q_t \succeq \frac12 Q_\infty$ and $Q_t \preceq Q_\infty$, we have that
    \begin{align*}
            |\gamma_{Q_t}(a-z)-\gamma_{Q_t}(b-z)| &\leq 2\|Q^{-1}_\infty\||a-b|\int_0^1|b+r(a-b)-z|\gamma_{Q_t}(b+r(a-b)-z)\ dr\\
            &\leq C\|Q^{-1}_\infty\||a-b|\int_0^1|b+r(a-b)-z|\gamma_{Q_\infty}(b+r(a-b)-z)\ dr.
    \end{align*}
    Plugging back into the original expression
    \begin{align*}
        &|\eta\ast \gamma_{Q_t}(a)-\eta\ast\gamma_{Q_t}(b)| \\
        &\leq C\|Q^{-1}_\infty\||a-b| \int_{\R^d}\int_0^1|b+r(a-b)-z|\gamma_{Q_\infty}(b+r(a-b)-z)\ dr\;d\eta(z)\\
        &\leq C\|Q^{-1}_\infty\||a-b|\sup_{c\in[a,b]}\int_{\R^d}(1+|c-z|^2)\gamma_{Q_\infty}(c-z)\;d\eta(z)\\
        &\leq   C(\beta,d,Q_\infty)|a-b|(\gamma_{\beta Q_\infty}\eta)(\R^d) \sup_{c\in[a,b]}(1+|c|^2)\exp\left(\frac{c^T Q_\infty^{-1}c}{2(\beta-1)}\right).
    \end{align*} where we applied Tonelli in the second to last line, applied \eqref{eq:complete sq} and used that $\eta$ is non-negative.
\end{proof}

     We now show that we can prove our main result under the assumption under a mild tail condition on the initial law.
\begin{theorem}\label{thm:OF_stability}
    Suppose $(F,H)$ satisfy Assumption~\ref{ass:detectable}, and let $\alpha_\phi=-\alpha(-F-R_\phi)>0$ as defined in \eqref{def:alpha_phi}. Let $q_0\in\R^d$, $\phi(x)=\frac12 x^T\Sigma_\phi^{-1} x$, $\mu_\phi=\tau_{q_0}\gamma_{\Sigma_\phi}$, and $\mu\in\P(\R^d)$. 
    
    Suppose that for some $\beta>1$, $\mu$ satisfies
    \begin{equation}\label{eq:tail_cond}
        \left(\gamma_{\beta Q_\infty}e^{\phi}\tau_{-q_0}\mu\right)(\R^d)<+\infty.
    \end{equation}
    Then, for any $\eps>0$ and $1\leq p<+\infty$, there exists an almost surely finite constant $C=C(\beta,d,Q_\infty,S_\infty,\alpha_\phi,p,\eps,\mu,\omega)>0$ such that, for sufficiently large $t>0$,
    \begin{equation}\label{eq:Chip->0}
        \left(\int_{\R^d}\left|\frac{d\pi_t^\mu}{d\pi_t^{\mu_\phi}}-1\right|^p\;d\pi_t^{\mu_\phi}\right)^{1/p}\leq C(\beta,d,Q_\infty,S_\infty,\alpha_\phi,p,\eps,\mu,\omega) e^{-(\alpha_\phi-\eps)t} 
        \left(\gamma_{\beta Q_\infty}e^{\phi}\tau_{-q_0}\mu\right)(\R^d)
    \end{equation}
\end{theorem}

\begin{remark}\label{rmk:observability}
Before proceeding to the proof, we make a few remarks on the tail condition \eqref{eq:tail_cond}. When $\mu\in\P(\R^d)$ is absolutely continuous, the condition is equivalent to
\begin{equation}\label{eq:f0;betaQ_infty}
    f_0:=\tau_{-q_0}\left(\frac{d\mu}{d\mu_\phi}\right)\in L^{1}(\gamma_{\beta Q_\infty}),
\end{equation}
and $f_t$ in the proof below is the $L^1(\m_\phi)$ semigroup solution with initial data $f_0$. In fact, $\|f_0\|_{L^1(\gamma_{\beta Q_\infty})}$ equals $\left(\gamma_{\beta Q_\infty}e^{\phi}\tau_{-q_0}\mu\right)(\R^d)$ up to a multiplicative constant. In fact, our initial result was formulated in terms of $f_0$, but in the final proofreading of the manuscript, it was suggested by ChatGPT-6 Astra that the same proof yields a more general result without the absolute continuity assumption.

Moreover, if $(F,H)$ is observable, one can show that the strict ordering $Q_\infty\prec \Sigma_\phi$ holds. In particular, there exists some $\beta>1$ such that $\Sigma_\phi\succeq \beta Q_\infty$. Thus, in this case, the tail condition \eqref{eq:tail_cond} holds for all $\mu\in\P(\R^d)$. 
\end{remark}

\begin{proof}
We split the proof into two steps. In \emph{Step 1.,} we reduce the question of $\chi^p$-divergence to a simpler statement, which is addressed in \emph{Step 2.,} by application of Lemma~\ref{lem:cov_err}
\medskip

\emph{Step 1.}
    From Theorem~\ref{thm:gauge}, we have
    \[\pi_t^{\mu_\phi}=N(q_t-\Sigma_\phi p_t,\Sigma_\phi)\propto e^{-\phi(x-q_t)-p_t\cdot x}.\]
    Let
    \[f_t(x)=e^{-w(t,x)+\lambda t+\phi(x)},\]
    where $w(t,x)=-\log\varrho(t,x)$, $\varrho$ being the solution of \eqref{eq:parabolic;varrho} with initial trace $\tau_{-q_0}\mu$. As $f_t$ satisfies the Kolmogorov backward equation for the OU process \eqref{eq:OU;M} with $M=F+R_\phi$, we have the representation
    \begin{equation}\label{OU;convolution}
        f_t(x)=(\gamma_{Q_t}\ast \eta_\mu)(e^{-(F+R_\phi)t}x) \text{ where } \eta_\mu:=e^{\phi(x)}(\tau_{-q_0}\mu);
    \end{equation}
    from the tail condition \eqref{eq:tail_cond} it follows that the above convolution is well-defined for all $t>0$, and the characterization is valid by Lemma~\ref{lem:OU;unique}. 

    Moreover,
    \begin{align*}
        \frac{d\pi_t^\mu}{d\pi_t^{\mu_\phi}}(x)&=\frac{\exp(-w(x-q_t)-p_t\cdot x +\phi(x-q_t)+p_t \cdot x)}{\int_{\R^d}\exp(-w(y-q_t)-p_t\cdot y +\phi(y-q_t)+p_t \cdot y)\;d\pi_t^{\mu_\phi}(y)}\\
        &=\frac{f_t(x-q_t)}{\int f_t(\cdot-q_t)d\pi_t^{\mu_\phi}}
        =\frac{f_t(x-q_t)}{\int f_t(\cdot-\Sigma_\phi p_t)\;d\gamma_{\Sigma_\phi}}.
    \end{align*}
    Writing
    \[Z_t:=\int f_t(\cdot-\Sigma_\phi p_t)\;d\gamma_{\Sigma_\phi},\]
    the $\chi^p$-divergence is
    \begin{align*}
        \int_{\R^d}\left|\frac{d\pi_t^\mu}{d\pi_t^{\mu_\phi}}-1\right|^p\;d\pi_t^{\mu_\phi}=\int_{\R^d} |Z_t^{-1} f_t(x-q_t)-1|^p  \;d\pi_t^{\mu_\phi}=\int_{\R^d} |Z_t^{-1} f_t(x-\Sigma_\phi p_t) -1|^p\;d\gamma_{\Sigma_\phi}(x).
    \end{align*}
    Recall from \eqref{eq:rt} that $S_t:=e^{-(F+R_\phi)t}\Sigma_\phi p_t\rightarrow S_\infty<+\infty$ almost surely. Thus, we have the for all $\gamma_{\Sigma_\phi}$-a.e. $x\in\R^d$,
    \begin{equation}\label{eq:ft->Z_infty}
        f_t(x-\Sigma_\phi p_t)=\eta_\mu\ast \gamma_{Q_t}(e^{-(F+R_\phi)t}x -S_t)\rightarrow \eta_\mu\ast \gamma_{Q_\infty}(-S_\infty)=:Z_\infty,
    \end{equation}
    where $Z_\infty >0$ almost surely positive.
    In particular, if $f_t(\cdot-\Sigma_\phi p_t)\rightarrow Z_\infty$ in $L^1(\gamma_{\Sigma_\phi})$, $Z_t\rightarrow Z_\infty$ follows, hence for sufficiently large $t>0$ so that $Z_t\geq Z_\infty/2$,
    \begin{align*}
        \left(\int_{\R^d}\left|\frac{d\pi_t^\mu}{d\pi_t^{\mu_\phi}}-1\right|^p\;d\pi_t^{\mu_\phi}\right)^{1/p}&=Z_t^{-1} \|f_t(\cdot-\Sigma_\phi p_t)-Z_t\|_{L^p(\gamma_{\Sigma_\phi})} \\
        &\leq 2 Z_\infty^{-1} \|f_t(\cdot-\Sigma_\phi p_t)-Z_t\|_{L^p(\gamma_{\Sigma_\phi})}\leq 4 Z_\infty^{-1} \|f_t(\cdot-\Sigma_\phi p_t)-Z_\infty\|_{L^p(\gamma_{\Sigma_\phi})}.
    \end{align*}
    Thus, it suffices to quantify the rate of $L^p(\gamma_{\Sigma_\phi})$-convergence of $f_t(\cdot-\Sigma_\phi p_t)$ to $Z_\infty$.
\medskip

\emph{Step 2.}
     To this end, note that that by the triangle inequality 
 \begin{equation}\label{eq:triangle}
 \begin{split}
     |f_t(x-\Sigma_\phi p_t)-Z_\infty|=&|\eta_\mu\ast\gamma_{Q_t}(e^{-(F+R_\phi)t}x-S_t)-\eta_\mu\ast\gamma_{Q_\infty}(-S_\infty)|\\
        \leq& |\eta_\mu\ast\gamma_{Q_t}(e^{-(F+R_\phi)t}x-S_t)-\eta_\mu\ast\gamma_{Q_t}(-S_t)|\\
        &+|\eta_\mu\ast\gamma_{Q_t}(-S_t)-\eta_\mu\ast\gamma_{Q_t}(-S_\infty)|+|\eta_\mu\ast\gamma_{Q_t}(-S_\infty)-\eta_\mu\ast\gamma_{Q_\infty}(-S_\infty)|.
\end{split}
 \end{equation} We now analyze each term in \eqref{eq:triangle}. Set $M= F+R_\phi$. Then the first term can be bounded using \eqref{eq:cov_err_2}:
 \begin{align*}
     &\int_{\R^d} |\eta_\mu\ast\gamma_{Q_t}(e^{-Mt}x-S_t)-\eta_\mu\ast\gamma_{Q_t}(-S_t)|^p d\gamma_{\Sigma_{\phi}}(x)\\
     &\leq C(\beta,d,Q_\infty)|(\gamma_{\beta Q_\infty}\eta_\mu)(\R^d)|^p \int_{\R^d}|e^{-Mt}x|^p\left(\sup_{c\in[e^{-Mt}x-S_t,-S_t]} (1+|c|^2) \exp\left(\frac{c^T Q_\infty^{-1} c}{2(\beta-1)}\right)\right)^p d\gamma_{\Sigma_\phi}(x)\\
     &\leq C(\beta,d,Q_\infty)|(\gamma_{\beta Q_\infty}\eta_\mu)(\R^d)|^p\\
     &\times\int_{\R^d}|e^{-Mt}x|^p (1+2|e^{-Mt}x|^2+2|S_t|^2)^p \exp\left(\frac{\|Q_\infty^{-1}\|}{(\beta-1)}(|S_t|^2 +|e^{-Mt}x|^2
     )\right)^p d\gamma_{\Sigma_\phi}(x).
 \end{align*} For $t$ sufficiently large, we can bound the integral above by
 \begin{align*}
     &\int_{\R^d}|e^{-Mt}x|^p (1+2|e^{-Mt}x|^2+2|S_t|^2)^p \exp\left(\frac{\|Q_\infty^{-1}\|}{(\beta-1)}(|S_t|^2 +|e^{-Mt}x|^2
     )\right)^p d\gamma_{\Sigma_\phi}(x)\\
     &\leq C(\beta,d,Q_\infty,S_\infty)e^{-p(\alpha_\phi-\eps)t}\int_{\R^d}|x|^p(1+|x|^2)^p\exp\left(\frac{\|Q_\infty^{-1}\|}{\beta-1}\|e^{-Mt}\|^2 |x|^2\right)^pd\gamma_{\Sigma_\phi}(x)\\
     &\leq C(\beta,d,Q_\infty,S_\infty,\alpha_\phi,p,\eps)e^{-p(\alpha_\phi-\eps)t}
 \end{align*} Note that in the second to last line, we used that the tail decay of $\gamma_{\Sigma_\phi}$ cancels out the growth of $\exp\left(\frac{\|Q_\infty^{-1}\|}{\beta-1}\|e^{-Mt}\|^2 |x|^2\right)$ with when $t$ is sufficiently large, ensuring that $\|e^{-Mt}\|$ is small. Thus the integral remains bounded (uniform in sufficiently large $t$). Hence, the first term decays to 0 as $t\to\infty$.
 
 Decay of the second term in the right-hand side of \eqref{eq:triangle} follows directly from and by \eqref{eq:cov_err_1}:
 \begin{align*}
     |\eta_\mu\ast \gamma_{Q_t}(-S_\infty)-\eta_\mu\ast \gamma_{Q_\infty}(-S_\infty)| &\leq C(\beta,d,Q_\infty,S_\infty)\|Q_t-Q_\infty\| (\gamma_{\beta Q_\infty}\eta_\mu)(\R^d).
 \end{align*}
 
Lastly, using \eqref{eq:cov_err_2}, the third term in \eqref{eq:triangle} is bounded by
 \begin{align*}
&|\eta_\mu\ast \gamma_{Q_t}(-S_t)-\eta_\mu\ast\gamma_{Q_t}(-S_\infty)|\\
&\leq C(\beta,d,Q_\infty)|S_t-S_\infty| (\gamma_{\beta Q_\infty}\eta_\mu)(\R^d)\sup_{c\in[-S_t,-S_\infty]} (1+|c|^2)\exp\left(\frac{c^T Q_\infty^{-1}c}{2(\beta-1)}\right)\\
&\leq |S_t-S_\infty| C(\beta,d, Q_{\infty},S_\infty) (\gamma_{\beta Q_\infty}\eta_\mu)(\R^d).
 \end{align*} In the last line, we choose $t$ sufficiently large. Then \eqref{eq:Chip->0} follows from \eqref{eq:Q_infty-Q_t} and \eqref{eq:r_infty-S_t}
 \begin{align*}
      \left(\int_{\R^d}\left|\frac{d\pi_t^\mu}{d\pi_t^{\mu_\phi}}-1\right|^p\;d\pi_t^{\mu_\phi}\right)^{1/p} &\leq C (\gamma_{\beta Q_\infty}\eta_\mu)(\R^d)\left(|S_t-S_\infty|+\|Q_t-Q_\infty\|+e^{-(\alpha_\phi-\eps)t}\right)\\
      &\leq C(\gamma_{\beta Q_\infty}\eta_\mu)(\R^d) e^{-(\alpha_\phi-\eps)t}
 \end{align*}
 where the random constant changes in each line, but is almost surely finite, and its dependence on parameters is specified in \eqref{eq:Chip->0}.
 \end{proof}

\begin{remark}\label{rmk:hypercontractivity}
    Note that under the assumptions, there is no reason why $\chi^p(\pi_t^\mu||\pi_t^{\mu_\phi})$ should be finite initially. Yet, the tail condition \eqref{eq:tail_cond} ensures that the $\chi^p$-divergence becomes finite at later times. This is reminiscent of hypercontractivity, which gains higher integrability at later time given $f_0$ from \eqref{eq:f0;betaQ_infty} is in $L^{1+\delta}(\gamma_{Q_\infty})$) for some $\delta>0$, and equivalent condition to the log-Sobolev inequality~\cite{Gross1975}. However, our initial tail condition is weaker than initial higher integrability.
\end{remark}

\begin{corollary}\label{cor:TV,W2->0}
    Suppose $(F,H)$ satisfies Assumption~\ref{ass:detectable} for any $\mu,\nu\in\P(\R^d)$ satisfying \eqref{eq:tail_cond} for some $\beta>0$,
    \begin{equation}\label{eq:TV->0}
        \limsup_{t\rightarrow\infty}\frac1t \log\|\pi_t^\mu-\pi_t^\nu\|_{TV}\leq -\alpha_\phi.
    \end{equation} 
    and
    \begin{equation}\label{eq:W2->0}
        \limsup_{t\rightarrow\infty}\frac1t \log W_2(\pi_t^\mu,\pi_t^{\nu})\leq -\alpha_\phi.
    \end{equation}
\end{corollary}

\begin{proof}
    As the $\chi^2$-divergence controls the $\kl$-divergence, from \eqref{eq:Chip->0} it immediately follows that
    \begin{equation}\label{eq:KL->0}
        \limsup_{t\rightarrow\infty}\frac1t\log\kl(\pi_t^\mu||\pi_t^{\mu_\phi})\leq 2\alpha(-F-R_\phi).
    \end{equation}
    Then \eqref{eq:TV->0} follows from Pinsker's inequality and the triangle inequality. Moreover, as $\pi_t^{\mu_\phi}$ is normally distributed with a fixed covariance $\Sigma_\phi$, by Talagrand's T2 inequality
    \begin{align*}
        \frac12W_2^2(\pi_t^\mu,\pi_t^{\mu_\phi})\leq C_{\Sigma_\phi} \kl(\pi_t^\mu||\pi_t^{\mu_\phi}) \text{ for all }t\geq0,
    \end{align*}
    where the constant $C_{\Sigma_\phi}$ only depends on the smallest eigenvalue of $\Sigma_\phi$. Thus, by triangle inequality \eqref{eq:W2->0} follows.
\end{proof}

\begin{remark}\label{rmk:OF_stability;comparison}
    We mention previously known results under comparable settings.
    Under detectability of $(F,H)$, Ocone and Pardoux~\cite[Theorem 2.6]{OconePardeoux96asymptotic} showed that, when $\mu,\nu\in\P(\R^d)$ have finite second moments, given any bounded uniformly continuous test function $\psi\in C(\R^d)$,
    \begin{equation}\label{eq:Ocone-Pardoux}
        |\pi_t^\mu(\psi)-\pi_t^\nu(\psi)|\rightarrow 0 \text{ almost surely. }
    \end{equation}
    Moreover, the exponent $\alpha_\phi$ for our asymptotic exponential decay coincides with the exponent for their quantitative result~\cite[Theorem 2.3, Remark 2.7]{OconePardeoux96asymptotic} for the conditional mean.
    
    On the other hand, van Handel applied his general filter stability result under uniform observability condition in the linear coefficient setting to deduce that, when $(F,H)$ satisfies the observability condition, for all $\mu,\nu\in\P(\R^d)$ (in our setting where the signal noise is nondegenerate, hence controllability condition holds automatically)~\cite{vanHandel09uniform}
    \begin{equation}\label{eq:vanHandel;observ}
        \|\pi_t^\mu-\pi_t^\nu\|_{BL}\rightarrow 0 \text{ almost surely; }
    \end{equation}
    here $\|\cdot\|_{BL}$ is the bounded Lipschitz norm, which metrizes the narrow convergence of probability measures, hence is weaker than both $W_2$ and the total variation norm.

    Corollary~\ref{cor:TV,W2->0} provides a quantitative rate of convergence in metrics stronger than both. Moreover, filter stability in the much stronger $\chi^p$-divergence in Theorem~\ref{thm:OF_stability} suggests that, filter stability with respect to a correct `reference optimal filter' may be easier to study and hold in a stronger sense.
\end{remark}

\bigskip

\noindent{\bf Acknowledgments.}
SP is partially supported by the Department of Defense Vannevar Bush Faculty Fellowship, award N00014-22-1-2790 (PI: Andrew M. Stuart). MVN is partially supported by NSF CAREER Award 2340762 (PI: Franca Hoffmann).

The authors would like to thank Franca Hoffmann and Andrew M. Stuart for their support and helpful discussions.

\bigskip

\noindent{\bf AI statement.}
In preparing this manuscript, ChatGPT 5.5 and 5.6 Sol were used in limited capacity for literature search, and performing and cross-verifying routine calculations.

Beyond basic tasks, the above models were used to assist in the following parts of the manuscript: slight sharpening of the estimates by suggesting useful elementary inequalities (such as the Gaussian covariance derivative formula) in the proof of Lemma~\ref{lem:cov_err}, which led to a minor relaxation of the assumption on the initial law; decomposition of $E_t$ in \eqref{eq:Et;decomposition} in \emph{Step 2} of the proof of Lemma~\ref{lem:Et} which allowed us to remove finite second moment assumption on the initial law; justification the Gaussian mixture representation of Bene\v{s} and Karatzas~\cite{BenesKaratzas83} by Girsanov's formula in the proof of Lemma~\ref{lem:Benes-Karatzas}; detailed probabilistic setup of the signal-observation model in Section~\ref{sssec:SigObs}, which is inspired by Section 1.1.2 of van Handel's thesis~\cite{vanHandel06Thesis}. While the suggestions by the AI models are limited to elementary or standard techniques available in the literature, it nevertheless saved the authors a nontrivial amount of time. The authors have carefully verified, digested, and reorganized any suggestion that has been made by AI models.

ChatGPT-6 Astra was used to provide feedback on a complete draft of the manuscript, where the model pointed out various typos and suggested a minor generalization of the statement of Theorem~\ref{thm:OF_stability} (hence Corollary~\ref{cor:TV,W2->0}) that immediately followed from our proofs; see Remark~\ref{rmk:observability}.

We emphasize that AI models were not used to produce any of the main ideas contained in this manuscript. The ideas, as outlined in Section~\ref{sssec:main ideas}, were conceived between February and March of 2026 by the authors while following the trails of thought in previous works by human authors, who we have credited carefully in our summary, Section~\ref{ssec:summary}. Moreover, AI models were not used to write any part of the manuscript, and the authors assume full responsibility for all content of this paper.

\bibliography{bib}

@article{HoffmannParkStuart2026EnKF,
  title={Large-time behavior and accuracy of the Mean-Field Ensemble Kalman Filter in the Linear Detectable Setting},
  author={Hoffmann, Franca and Park, Sangmin and Stuart, Andrew M.},
  journal={arXiv preprint arXiv:2606.07729},
  year={2026}
}

@article{FujisakiKallianpurKunita1972,
  author  = {Fujisaki, Masatoshi and Kallianpur, Gopinath and Kunita, Hiroshi},
  title   = {Stochastic Differential Equations for the Non Linear Filtering Problem},
  journal = {Osaka Journal of Mathematics},
  volume  = {9},
  number  = {1},
  pages   = {19--40},
  year    = {1972},
  doi     = {10.18910/5728}
}

@article{Pardoux1979,
  author  = {Pardoux, Etienne},
  title   = {Stochastic Partial Differential Equations and Filtering of Diffusion Processes},
  journal = {Stochastics},
  volume  = {3},
  pages   = {127--167},
  year    = {1979},
  doi     = {10.1080/17442507908833142}
}

@article{KrylovRozovskii1979,
  author  = {Krylov, Nikolai V. and Rozovski{\u\i}, Boris L.},
  title   = {Stochastic Evolution Equations},
  journal = {Journal of Soviet Mathematics},
  volume  = {16},
  number  = {4},
  pages   = {1233--1277},
  year    = {1981},
  note    = {English translation of Itogi Nauki i Tekhniki, Ser. Sovrem. Probl. Mat., vol. 14 (1979), pp. 71--146},
  doi     = {10.1007/BF01084893}
}

@article{ChaleyatMaurelMichelPardoux1990,
  author  = {Chaleyat-Maurel, Mireille and Michel, Dominique and Pardoux, Etienne},
  title   = {Un th{\'e}or{\`e}me d'unicit{\'e} pour l'{\'e}quation de {Z}aka{\"i}},
  journal = {Stochastics and Stochastic Reports},
  volume  = {29},
  number  = {1},
  pages   = {1--12},
  year    = {1990},
  doi     = {10.1080/17442509008833605}
}

@article{FlemingMitter1982,
  author  = {Fleming, Wendell H. and Mitter, Sanjoy K.},
  title   = {Optimal Control and Nonlinear Filtering for Nondegenerate Diffusion Processes},
  journal = {Stochastics},
  volume  = {8},
  number  = {1},
  pages   = {63--77},
  year    = {1982},
  doi     = {10.1080/17442508208833228}
}

@article{Pardoux1982,
  author  = {Pardoux, Etienne},
  title   = {{\'E}quations du filtrage non lin{\'e}aire de la pr{\'e}diction et du lissage},
  journal = {Stochastics},
  volume  = {6},
  number  = {3-4},
  pages   = {193--231},
  year    = {1982},
  doi     = {10.1080/17442508208833204}
}

@article{BarasBlankenshipHopkins1983,
  author  = {Baras, John S. and Blankenship, Gilmer L. and Hopkins, William E.},
  title   = {Existence, Uniqueness, and Asymptotic Behavior of Solutions to a Class of {Z}akai Equations with Unbounded Coefficients},
  journal = {IEEE Transactions on Automatic Control},
  volume  = {28},
  number  = {2},
  pages   = {203--214},
  year    = {1983},
  doi     = {10.1109/TAC.1983.1103218}
}

@article{KurtzOcone1988,
  author  = {Kurtz, Thomas G. and Ocone, Daniel L.},
  title   = {Unique Characterization of Conditional Distributions in Nonlinear Filtering},
  journal = {The Annals of Probability},
  volume  = {16},
  number  = {1},
  pages   = {80--107},
  year    = {1988},
  doi     = {10.1214/aop/1176991887}
}

@article{BhattKallianpurKarandikar1995,
  author  = {Bhatt, Abhay G. and Kallianpur, Gopinath and Karandikar, Rajeeva L.},
  title   = {Uniqueness and Robustness of Solution of Measure-Valued Equations of Nonlinear Filtering},
  journal = {The Annals of Probability},
  volume  = {23},
  number  = {4},
  pages   = {1895--1938},
  year    = {1995},
  doi     = {10.1214/aop/1176987808}
}

@article{Simon81JFA,
  title={Large time behavior of the {$L^p$} norm of {S}chr{\"o}dinger semigroups},
  author={Simon, Barry},
  journal={Journal of Functional Analysis},
  volume={40},
  number={1},
  pages={66--83},
  year={1981},
  publisher={Elsevier}
}

@article{Murata84Largetime,
  title={Positive solutions and large time behaviors of {S}chr{\"o}dinger semigroups, {S}imon's problem},
  author={Murata, Minoru},
  journal={Journal of {F}unctional {A}nalysis},
  volume={56},
  number={3},
  pages={300--310},
  year={1984},
  publisher={Elsevier}
}

@article{Murata1986,
  author  = {Murata, Minoru},
  title   = {Structure of Positive Solutions to
             {$(-\Delta+V)u=0$} in {$\mathbb{R}^n$}},
  journal = {Duke Mathematical Journal},
  volume  = {53},
  number  = {4},
  pages   = {869--943},
  year    = {1986},
  doi     = {10.1215/S0012-7094-86-05347-0}
}

@article{Pinchover88Duke,
author = {Yehuda Pinchover},
title = {{On positive solutions of second-order elliptic equations, stability results, and classification}},
volume = {57},
journal = {Duke Mathematical Journal},
number = {3},
publisher = {Duke {U}niversity {P}ress},
pages = {955 -- 980},
year = {1988},
doi = {10.1215/S0012-7094-88-05743-2},
URL = {https://doi.org/10.1215/S0012-7094-88-05743-2}
}

@article{Simon82BAMS,
author = {Barry Simon},
title = {{Schrödinger semigroups}},
volume = {7},
journal = {Bulletin (New Series) of the American Mathematical Society},
number = {3},
publisher = {American Mathematical Society},
pages = {447 -- 526},
year = {1982},
}

@article{Agmon82,
  title={On positivity and decay of solutions of second order elliptic equations on {R}iemannian manifolds},
  author={Agmon, Shmuel},
  journal={Methods of functional analysis and theory of elliptic equations},
  pages={19--52},
  year={1982},
  publisher={Liguori Editore}
}

@book{Pinsky95book,
  title={Positive harmonic functions and diffusion},
  author={Pinsky, Ross G.},
  volume={45},
  year={1995},
  publisher={Cambridge {U}niversity {P}ress}
}

@article{Doob57htransform,
  title={Conditional {B}rownian motion and the boundary limits of harmonic functions},
  author={Doob, Joseph L.},
  journal={Bulletin de la Soci{\'e}t{\'e} math{\'e}matique de France},
  volume={85},
  pages={431--458},
  year={1957}
}

@article{Krylov10well-posedness,
  title={{Kalman--Bucy filter and SPDEs with growing lower-order coefficients in $W_p^{1} $ spaces without weights}},
  author={Krylov, Nikolai V.},
  journal={Illinois Journal of Mathematics},
  volume={54},
  number={3},
  pages={1069--1114},
  year={2010},
  publisher={Duke {U}niversity {P}ress}
}

@article{SunZhouYau26well-posedness,
  title={Well-posedness of Filtering Equations in Weighted {S}obolev Spaces with Unbounded System Coefficients},
  author={Sun, Zeju and Zhou, Songlin and Yau, Stephen S-T},
  journal={arXiv preprint arXiv:2609.03549},
  year={2026}
}

@article{KaliathZakai71AnnMathStat,
  title={Absolute continuity and Radon-Nikodym derivatives for certain measures relative to {W}iener measure},
  author={Kailath, Thomas and Zakai, Moshe},
  journal={The Annals of Mathematical Statistics},
  pages={130--140},
  year={1971},
  publisher={JSTOR}
}

@article{KalStr68,
  title={Estimation of stochastic systems: Arbitrary system process with additive white noise observation errors},
  author={Kallianpur, Gopinath and Striebel, Charlotte},
  journal={The Annals of Mathematical Statistics},
  volume={39},
  number={3},
  pages={785--801},
  year={1968},
  publisher={JSTOR}
}

@article{BhattKarandikar01,
  title={Path continuity of the nonlinear filter},
  author={Bhatt, Abhay G. and Karandikar, Rajeeva L.},
  journal={Statistics \& probability letters},
  volume={54},
  number={1},
  pages={75--78},
  year={2001},
  publisher={Elsevier}
}

@book{LipsterShirayev01,
  title={Statistics of Random Processes: I. General Theory},
  author={Liptser, Robert S. and Shiryaev, Albert N.},
  edition={2nd, rev. and expanded},
  volume={5},
  series={Applications of Mathematics},
  year={2001},
  publisher={Springer-Verlag},
  address={Berlin, Heidelberg},
  isbn={978-3-540-63929-9}
}

@book{Harrison13,
  title={Brownian models of performance and control},
  author={Harrison, J. Michael},
  year={2013},
  publisher={Cambridge {U}niversity {P}ress}
}

@book{Resnick07heavy,
  title={Heavy-tail phenomena: probabilistic and statistical modeling},
  author={Resnick, Sidney I.},
  year={2007},
  publisher={Springer}
}

@book{BKRS2022FPK,
  title={Fokker--{P}lanck--{K}olmogorov Equations},
  author={Bogachev, Vladimir I. and Krylov, Nicolai V. and R{\"o}ckner, Michael and Shaposhnikov, Stanislav V.},
  volume={207},
  year={2022},
  publisher={American Mathematical Society}
}

@article{Aronson71add,
  title={Non-negative solutions of linear parabolic equations: An addendum},
  author={Aronson, DG},
  journal={Annali della Scuola Normale Superiore di Pisa-Scienze Fisiche e Matematiche},
  volume={25},
  number={2},
  pages={221--228},
  year={1971}
}

@article{Aronson68,
  title={Non-negative solutions of linear parabolic equations},
  author={Aronson, Donald G.},
  journal={Annali della Scuola Normale Superiore di Pisa-Scienze Fisiche e Matematiche},
  volume={22},
  number={4},
  pages={607--694},
  year={1968}
}

@article{Thiffeault03,
  title={Advection--diffusion in {L}agrangian coordinates},
  author={Thiffeault, Jean-Luc},
  journal={Physics Letters A},
  volume={309},
  number={5-6},
  pages={415--422},
  year={2003},
  publisher={Elsevier}
}

@book{RozovskyLototsky18book,
  title={Stochastic evolution systems},
  author={Rozovsky, Boris L. and Lototsky, Sergey V.},
  year={2018},
  publisher={Springer}
}

@article{BenesKaratzas83,
  title={Estimation and control for linear, partially observable systems with non-Gaussian initial distribution},
  author={Bene{\v{s}}, V{\'a}clav E. and Karatzas, Ioannis},
  journal={Stochastic Processes and their Applications},
  volume={14},
  number={3},
  pages={233--248},
  year={1983},
  publisher={Elsevier}
}

@book{Teschll12ODE,
  title={Ordinary differential equations and dynamical systems},
  author={Teschl, Gerald},
  volume={140},
  year={2012},
  publisher={American Mathematical Soc.}
}

@article{ArnoldCarlenJu08,
  title={{Large-time behavior of non-symmetric Fokker-Planck type equations}},
  author={Arnold, Anton and Carlen, Eric and Ju, Qiangchang},
  journal={Communications on Stochastic Analysis},
  volume={2},
  number={1},
  pages={11},
  year={2008}
}

@article{FontbonaJourdain16,
  author  = {Joaquin Fontbona and Benjamin Jourdain},
  title   = {A trajectorial interpretation of the dissipations of entropy and {F}isher information for stochastic differential equations},
  journal = {The Annals of Probability},
  volume  = {44},
  number  = {1},
  pages   = {131--170},
  year    = {2016},
  doi     = {10.1214/14-AOP969}
}

@article{Gross1975,
  author  = {Gross, Leonard},
  title   = {Logarithmic {S}obolev {I}nequalities},
  journal = {American Journal of Mathematics},
  volume  = {97},
  number  = {4},
  pages   = {1061--1083},
  year    = {1975},
  doi     = {10.2307/2373688}
}

@book{BGL14,
  title={Analysis and geometry of Markov diffusion operators},
  author={Bakry, Dominique and Gentil, Ivan and Ledoux, Michel},
  volume={103},
  year={2014},
  publisher={Springer}
}

@article{CrisanPardoux26,
  title={Uniqueness of the solution of the filtering equations in spaces of measures for general signal and observation processes},
  author={Crisan, Dan and Pardoux, Etienne},
  journal={The Annals of Applied Probability},
  volume={36},
  number={2},
  pages={1347--1376},
  year={2026},
  publisher={Institute of Mathematical Statistics}
}

@article{barles2020large,
  title={Large-time behavior of unbounded solutions of viscous {H}amilton--{J}acobi equations in {$\mathbb{R}^n$}},
  author={Barles, Guy and Quaas, Alexander and Rodr{\'\i}guez-Paredes, Andrei},
  journal={Communications in Partial Differential Equations},
  volume={46},
  number={3},
  pages={547--572},
  year={2020},
  publisher={Taylor \& Francis}
}

@article{barles2016unbounded,
  title={On unbounded solutions of ergodic problems in {$\mathbb{R}^m$} for viscous {H}amilton--{J}acobi equations},
  author={Barles, Guy and Meireles, Joao},
  journal={Communications in Partial Differential Equations},
  volume={41},
  number={12},
  pages={1985--2003},
  year={2016},
  publisher={Taylor \& Francis}
}

@article{IchiharaSheu13,
  title={Large time behavior of solutions of {H}amilton--{J}acobi--{B}ellman equations with quadratic nonlinearity in gradients},
  author={Ichihara, Naoyuki and Sheu, Shuenn-Jyi},
  journal={SIAM Journal on Mathematical Analysis},
  volume={45},
  number={1},
  pages={279--306},
  year={2013},
  publisher={SIAM}
}

@article{Ichi11recurrence,
  title={Recurrence and transience of optimal feedback processes associated with {B}ellman equations of ergodic type},
  author={Ichihara, Naoyuki},
  journal={SIAM {J}ournal on {C}ontrol and {O}ptimization},
  volume={49},
  number={5},
  pages={1938--1960},
  year={2011},
  publisher={SIAM}
}

@article{Ichihara12,
  title={Large time asymptotic problems for optimal stochastic control with superlinear cost},
  author={Ichihara, Naoyuki},
  journal={Stochastic Processes and their Applications},
  volume={122},
  number={4},
  pages={1248--1275},
  year={2012},
  publisher={Elsevier}
}

@book {BainCrisan09,
    AUTHOR = {Bain, Alan and Crisan, Dan},
     TITLE = {{Fundamentals of stochastic filtering}},
    SERIES = {Stochastic Modelling and Applied Probability},
    VOLUME = {60},
 PUBLISHER = {Springer, New York},
      YEAR = {2009},
     PAGES = {xiv+390},
      ISBN = {978-0-387-76895-3},
   MRCLASS = {60G35 (60-02 93E11)},
  MRNUMBER = {2454694},
MRREVIEWER = {Monique\ Pontier},
       DOI = {10.1007/978-0-387-76896-0},
       URL = {https://doi.org/10.1007/978-0-387-76896-0},
}

@article{MitterNewton03_variational,
  title={A variational approach to nonlinear estimation},
  author={Mitter, Sanjoy K. and Newton, Nigel J.},
  journal={SIAM {J}ournal on {C}ontrol and {O}ptimization},
  volume={42},
  number={5},
  pages={1813--1833},
  year={2003},
  publisher={SIAM}
}

@inproceedings{vanHandel10proceedings,
  title={Nonlinear filtering and systems theory},
  author={van Handel, Ramon},
  booktitle={Proceedings of the 19th International Symposium on Mathematical Theory of Networks and Systems (MTNS semi-plenary paper)},
  year={2010}
}

@article{vanHandel09uniform,
  title={Uniform observability of hidden Markov models and filter stability for unstable signals.},
  author={van Handel, Ramon},
  journal={The Annals of Applied Probability},
  volume={19},
  number={3},
  pages={1172--1199},
  year={2009},
  publisher={Institute of Mathematical Statistics}
}

@article{Atar98AoP,
  title={Exponential stability for nonlinear filtering of diffusion processes in a noncompact domain},
  author={Atar, Rami},
  journal={The Annals of Probability},
  pages={1552--1574},
  year={1998},
  publisher={JSTOR}
}

@article{AtarZeitouni97lyapunov,
  title={Lyapunov exponents for finite state nonlinear filtering},
  author={Atar, Rami and Zeitouni, Ofer},
  journal={SIAM Journal on Control and Optimization},
  volume={35},
  number={1},
  pages={36--55},
  year={1997},
  publisher={SIAM}
}

@article{AtarZeitouni97exp,
  title={Exponential stability for nonlinear filtering},
  author={Atar, Rami and Zeitouni, Ofer},
  journal={Annales de l'Institut Henri Poincare (B) Probability and Statistics},
  volume={33},
  number={6},
  pages={697--725},
  year={1997},
  organization={Elsevier}
}

@article{Duncan68evaluation,
  title={Evaluation of likelihood functions},
  author={Duncan, Tyrone E.},
  journal={Information and Control},
  volume={13},
  number={1},
  pages={62--74},
  year={1968},
  publisher={Elsevier}
}

@article{Duncan70absolute,
  title={On the absolute continuity of measures},
  author={Duncan, Tyrone E.},
  journal={The Annals of Mathematical Statistics},
  volume={41},
  number={1},
  pages={30--38},
  year={1970},
  publisher={JSTOR}
}

@phdthesis{Duncan67,
  author = {Duncan, Tyrone E.},
  title = {Probability Densities for Diffusion Processes with Applications
           to Nonlinear Filtering Theory and Detection Theory},
  school = {Stanford University},
  year = {1967}
}

@phdthesis{Mortensen66,
  author = {Mortensen, Richard E.},
  title = {Optimal Control of Continuous-Time Stochastic Systems},
  school = {University of California, Berkeley},
  year = {1966}
}

@article{Zakai69optimal,
  title={On the optimal filtering of diffusion processes},
  author={Zakai, Moshe},
  journal={Zeitschrift f{\"u}r Wahrscheinlichkeitstheorie und verwandte Gebiete},
  volume={11},
  number={3},
  pages={230--243},
  year={1969},
  publisher={Springer}
}

@article{BaxChiPav04Wonham,
  title={Asymptotic stability of the {W}onham filter: ergodic and nonergodic signals},
  author={Baxendale, Peter and Chigansky, Pavel and Liptser, Robert},
  journal={SIAM {J}ournal on {C}ontrol and {O}ptimization},
  volume={43},
  number={2},
  pages={643--669},
  year={2004},
  publisher={SIAM}
}

@article{Budhiraja2003,
  author  = {Budhiraja, Amarjit},
  title   = {Asymptotic Stability, Ergodicity and Other Asymptotic Properties of the Nonlinear Filter},
  journal = {Annales de l'Institut Henri Poincar{\'e}, Probabilit{\'e}s et Statistiques},
  volume  = {39},
  number  = {6},
  pages   = {919--941},
  year    = {2003},
  doi     = {10.1016/S0246-0203(03)00022-0}
}

@article{KrylovRozovskii77,
  title={{On the Cauchy problem for linear stochastic partial differential equations}},
  author={Krylov, Nikolai V. and Rozovski{\u\i}, Boris L.},
  journal={Mathematics of the USSR-Izvestiya},
  volume={11},
  number={6},
  pages={1267--1284},
  year={1977}
}

@article{KrylovRozovskii82,
  title={Stochastic partial differential equations and diffusion processes},
  author={Krylov, Nikolai V. and Rozovski{\u\i}, Boris L.},
  journal={Russian Mathematical Surveys},
  volume={37},
  number={6},
  pages={81--105},
  year={1982}
}

@article{Kushner64,
  title={{On the differential equations satisfied by conditional probablitity densities of Markov processes, with applications}},
  author={Kushner, Harold J.},
  journal={Journal of the Society for Industrial and Applied Mathematics, Series A: Control},
  volume={2},
  number={1},
  pages={106--119},
  year={1964},
  publisher={SIAM}
}

@article{Stratonovich59,
  title={On the theory of optimal non-linear filtering of random functions},
  author={Stratonovich, Ruslan L.},
  journal={Theory of Probability \& Its Applications},
  volume={4},
  pages={223--225},
  year={1959}
}

@book{LanRod95,
  title={{Algebraic Riccati equations}},
  author={Lancaster, Peter and Rodman, Leiba},
  year={1995},
  publisher={Clarendon {P}ress}
}

@article{Bougerol93,
  title={{Kalman filtering with random coefficients and contractions}},
  author={Bougerol, Philippe},
  journal={SIAM Journal on Control and Optimization},
  volume={31},
  number={4},
  pages={942--959},
  year={1993},
  publisher={SIAM}
}

@article{KalmanBucy61,
  title={New Results in Linear Filtering and Prediction Theory},
  author={Kalman, Rudolf E. and Bucy, Richard S.},
  journal={Journal of Basic Engineering},
  volume={83},
  number={1},
  pages={95--108},
  year={1961},
  publisher={ASME International}
}

@article{Kalman60,
  title={A New Approach to Linear Filtering and Prediction Problems},
  author={Kalman, Rudolf E.},
  journal={Journal of Basic Engineering},
  volume={82},
  number={1},
  pages={35--45},
  year={1960},
  publisher={ASME International}
}

@article{anderson71stability,
  title={{Stability properties of Kalman-Bucy filters}},
  author={Anderson, Brian D.O.},
  journal={Journal of the Franklin Institute},
  volume={291},
  number={2},
  pages={137--144},
  year={1971},
  publisher={Elsevier}
}

@article{Makowski86,
  title={Filtering formulae for partially observed linear systems with non-{G}aussian initial conditions},
  author={Makowski, Armand M.},
  journal={Stochastics: An International Journal of Probability and Stochastic Processes},
  volume={16},
  number={1-2},
  pages={1--24},
  year={1986},
  publisher={Taylor \& Francis}
}

@inproceedings{SowersMakowski90,
  title={Discrete-time filtering for linear systems in correlated noise with non-{G}aussian initial conditions: formulas and asymptotics},
  author={Sowers, Richard B and Makowski, Armand M},
  booktitle={Robust Control of Linear Systems and Nonlinear Control: Proceedings of the International Symposium MTNS-89, Volume II},
  pages={407--419},
  year={1990},
  organization={Springer}
}

@article{Bucy67global,
  title={{Global theory of the {Riccati} equation}},
  author={Bucy, Richard S.},
  journal={Journal of Computer and System Sciences},
  volume={1},
  number={4},
  pages={349--361},
  year={1967},
  publisher={Elsevier}
}

@article{Wonham68riccati,
  title={{On a matrix Riccati equation of stochastic control}},
  author={Wonham, William M.},
  journal={SIAM Journal on Control},
  volume={6},
  number={4},
  pages={681--697},
  year={1968},
  publisher={SIAM}
}

@article{OconePardeoux96asymptotic,
  title={Asymptotic stability of the optimal filter with respect to its initial condition},
  author={Ocone, Daniel and Pardoux, Etienne},
  journal={SIAM Journal on Control and Optimization},
  volume={34},
  number={1},
  pages={226--243},
  year={1996},
  publisher={SIAM}
}

@article{Stratonovich60conditional,
author = {Stratonovich, Ruslan L.},
title = {{Conditional Markov Processes}},
journal = {Theory of Probability \& Its Applications},
volume = {5},
number = {2},
pages = {156-178},
year = {1960},
doi = {10.1137/1105015},

URL = { 
    
        https://doi.org/10.1137/1105015
    
    

},
eprint = { 
    
        https://doi.org/10.1137/1105015
    
    

}
}

@phdthesis{vanHandel06Thesis,
  title={Filtering, stability, and robustness},
  author={van Handel, Ramon},
  year={2006},
  school={California Institute of Technology}
}

@article{vanH09PTRF,
  title={Observability and nonlinear filtering},
  author={van Handel, Ramon},
  journal={Probability Theory and Related Fields},
  volume={145},
  number={1},
  pages={35--74},
  year={2009},
  publisher={Springer}
}

@article{Stannat06stability,
  title={Stability of the optimal filter via pointwise gradient estimates},
  author={Stannat, Wilhelm},
  journal={Lecture Notes in Pure and Applied Mathematics},
  volume={245},
  pages={281},
  year={2006},
  publisher={New York; Marcel Dekker; 1999}
}

@article{Stannat05time-dep,
  title = {Stability of the Filter Equation for a Time-Dependent Signal on $\mathbb{R}^d$},
  author = {Stannat, Wilhelm},
  journal = {Applied Mathematics and Optimization},
  volume = {52},
  number = {1},
  pages = {39--71},
  year = {2005},
  publisher = {Springer}
}

@article{AGS14Inventiones,
  title={{Calculus and heat flow in metric measure spaces and applications to spaces with Ricci bounds from below}},
  author={Ambrosio, Luigi and Gigli, Nicola and Savar{\'e}, Giuseppe},
  journal={Inventiones mathematicae},
  volume={195},
  number={2},
  pages={289--391},
  year={2014},
  publisher={Springer}
}

@article{vanHandel09stability,
  title={The stability of conditional {M}arkov processes and {M}arkov chains in random environments},
  author={van Handel, Ramon},
  journal={The Annals of Probability},
  volume={37},
  number={5},
  pages={1876--1925},
  year={2009}
}
\bibliographystyle{alpha}

\begin{appendix}

\section{Proof of Lemma~\ref{lem:Benes-Karatzas}}
In this section, we provide a proof of Lemma~\ref{lem:Benes-Karatzas} that verifies that the unnormalized filter \eqref{def:canonical_filtersol:Zakai} is a classical solution of the Zakai equation in the sense of Definition~\ref{def:Zakai;classical}.
\begin{proof}[Proof of Lemma~\ref{lem:Benes-Karatzas}]
    Item (iv) of Definition~\ref{def:Zakai;classical} regarding initial trace directly follows from the almost sure continuity of $t\mapsto\rho_t^\mu$ with respect to the narrow convergence due to Bhatt and Karandikar~\cite[Theorem 2.2]{BhattKarandikar01}.

    To verify items (i)-(iii), it suffices to verify the Gaussian mixture representation of the unnormalized filter density $\rho_t^\mu$ due to Bene\v{s} and Karatzas~\cite[Theorem 5.1]{BenesKaratzas83}, which takes the form \eqref{eq:Gauss-mixture;BK}. Indeed, once the Gaussian-mixture representation formula for $\rho_t^\mu$, numerator of \cite[(5.17)]{BenesKaratzas83}, is justified, this immediately verifies measurability (item (i)) and regularity (item (ii)). Finally, an application of It\^{o}'s formula verifies item (iii); see \cite[(5.6)]{BenesKaratzas83}.
    
    Thus, we focus on justifying \eqref{eq:Gauss-mixture;BK} directly. As discussed in the beginning of Section~\ref{ssec:canonical_filter}, by Girsanov's theorem, under $\bQ^\mu$, $(\Gamma^{-1/2}Y_s)_{s\leq t}$ is a standard Brownian motion, and $X_t$ and $Y_t$ are independent. Thus, as Bene\v{s} and Karatzas note, we can disintegrate with respect to the initial law $\mu$ to obtain
    \[\rho_t^\mu(\psi)=\int_{\R^d} \E^{\bQ^x}[\psi(X_t)L_t^{-1}|\Y_t]\;d\mu(x),\]
    where $\bQ^x=\bQ^{\delta_{x}}$ is defined as in~\eqref{def:Qmu}.
    By applying the Kallianpur-Striebel formula~\eqref{eq:Kallianpur-Striebel} with $\bP^x,\bQ^x$, we obtain
    \[\E^{\bQ^x}[\psi(X_t)L_t^{-1}|\Y_t]=\E^{\bQ^x}[L_t^{-1}|\Y_t]\E^{\mathbb{P}^x}[\psi(X_t)|\Y_t].\]
    As $\delta_{x}$ can be considered a Gaussian measure with mean $x$ and covariance $0$, $\E^{\mathbb{P}^x}[\psi(X_t)|(Y_s)_{s\leq t}]$ can be expressed in terms of the solutions of the Kalman-Bucy filter~\cite[Section 5]{BenesKaratzas83}. 
    
    Write $m_t(x):=\E^{\mathbb{P}^x}[X_t|\Y_t]$. To justify the formula \cite[(5.4)]{BenesKaratzas83}
    \begin{equation}\label{eq:Lambda_tx}
        \Lambda_t(x):=\E^{\bQ^x}[L_t^{-1}|\Y_t]=\exp\left(\int_0^t (H m_s(x))^T\Gamma^{-1}\;dY_s - \frac12\int_0^t |Hm_s(x)|_\Gamma^2\;ds\right),
    \end{equation}
    note that $\Lambda_t(x)=\frac{d\bP^x|_{\Y_t}}{d\bQ^x|_{\Y_t}}$. As under $\bQ^x$ and filtration $(\Y_t)_{t\geq 0}$, $(\Gamma^{-1/2}Y_t)_{t\geq 0}$ is a standard Brownian motion and $\Lambda_t(x)$ is a strictly positive (a.s.) martingale, there exists some $(\Y_t)_{t\geq 0}$-predictable process $(\theta_t)_{t\geq 0}$ such that
    \[d\Lambda_t(x)= \Lambda_t(x)\theta_t\cdot \Gamma^{-1/2}dY_t.\]
    Then, by Girsanov's theorem,
    \[\Gamma^{-1/2} Y_t-\int_0^t\theta_s\;ds\]
    is a Brownian motion under $\bP^x$ and $(\Y_t)_{t\geq 0}$. On the other hand, it is well known (see for instance \cite[Theorem 7.16]{LipsterShirayev01}, \cite[Theorem 1.12]{Harrison13}) that under $\bP^x$
    \[\Gamma^{-1/2} Y_t-\int_0^t\Gamma^{-1/2} Hm_s(x)\;ds\]
    is a Brownian motion with respect to $(\Y_t)_{t\geq 0}$. Thus, we conclude $\theta_t=\Gamma^{-1/2}Hm_t(x)$ by the uniqueness of the semimartingale decomposition, and \eqref{eq:Lambda_tx} is justified.
\end{proof}

\end{appendix}

\end{document}